\documentclass[10pt,twoside, a4paper, english, reqno]{amsart} %%,draft
\usepackage[foot]{amsaddr}%%
\usepackage{a4wide}
\usepackage{amscd}
\usepackage{mathtools}
\usepackage{upgreek}
\usepackage{soul}  %%for cutting the text
\usepackage{amssymb}
\usepackage{amsmath}
\usepackage{amsthm}
\usepackage{graphicx}
\usepackage{amsmath,calrsfs}
\usepackage{mathrsfs}
\usepackage{latexsym}
\usepackage{enumitem,dsfont}
\usepackage[new]{old-arrows}

\usepackage{upref}
\usepackage[colorlinks,backref]{hyperref}
\usepackage{color,graphics}
\usepackage{upref,hyperref,color}
\usepackage{eucal}

\usepackage[dvipsnames]{xcolor}
\usepackage{ulem,cancel,pgf}
\usepackage{frcursive}
\theoremstyle{plain}
\newtheorem{theorem}{Theorem}[section]
\newtheorem{proposition}[theorem]{Proposition}
\theoremstyle{plain}
\newtheorem{lemma}[theorem]{Lemma}

\newtheorem{corollary}[theorem]{Corollary}
\newtheorem{example}[theorem]{Example}

\theoremstyle{definition}
\newtheorem{definition}{Definition}[section]

\newtheorem{remark}{Remark}[section]

\newtheorem*{maintheorem*}{Main Theorem}
\newtheorem*{maincorollary*}{Main Corollary}

\newtheorem{hypothesis}[theorem]{Hypothesis}

\DeclareFontFamily{U}{BOONDOX-calo}{\skewchar\font=50 }
\DeclareFontShape{U}{BOONDOX-calo}{m}{n}{
	<-> s*[1.05] BOONDOX-r-calo}{}
\DeclareFontShape{U}{BOONDOX-calo}{b}{n}{
	<-> s*[1.05] BOONDOX-b-calo}{}
\DeclareMathAlphabet{\mathcalb}{U}{BOONDOX-calo}{m}{n}
\SetMathAlphabet{\mathcalb}{bold}{U}{BOONDOX-calo}{b}{n}
\DeclareMathAlphabet{\mathbcalb}{U}{BOONDOX-calo}{b}{n}
\newcommand{\R}{\ensuremath{\mathbb{R}}}
\newcommand{\N}{\ensuremath{\mathbb{N}}}
\newcommand{\E}{\ensuremath{\mathbb{E}}}

\def\N{\mathbb{N}}
\def\V{V}

\def\u{\boldsymbol{u}}
\def\v{\boldsymbol{v}}
\def\w{\boldsymbol{w}}
\def\f{\boldsymbol{f}}
\def\y{\boldsymbol{y}}
\def\z{\boldsymbol{z}}
\def\x{\boldsymbol{x}}

\def\n{\boldsymbol{n}}
\def\D{\mathrm{D}}
\def\V{\mathbb{V}}
\def\U{\mathbb{U}}
\def\L{\mathbb{L}}
\def\H{\mathbb{H}}
\def\P{\mathrm{P}}
\def\I{\mathrm{I}}
\def\A{\mathcal{A}}
\def\Arm{\mathrm{A}}
\def\B{\mathcal{B}}
\def\J{\mathcal{J}}
\def\K{\mathcal{K}}
\def\d{\mathrm{d}}
\def\C{\mathrm{C}}

\def\Tr{\mathrm{Tr}}
\def\X{\mathbb{X}}

\def\O{\mathcal{O}}
\def\W{\mathrm{W}}

\def\diver{\mathrm{div}}

\numberwithin{equation}{section} \allowdisplaybreaks

\title[Random dynamics of stochastic $3^{rd}$-grade fluids equations]
{Random dynamics and invariant measures for a class of non-Newtonian fluids of differential type on 2D and 3D Poincar\'e domains 
}

\author[Kush Kinra]{Kush Kinra}
\address[Kush Kinra]{Center for Mathematics and Applications (NOVA Math), NOVA School of Science and Technology (NOVA SST),	Portugal.}
\email[Kush Kinra]{kushkinra@gmail.com, k.kinra@fct.unl.pt}

\author[Fernanda Cipriano]{Fernanda Cipriano}
\address[Fernanda Cipriano]{
	Center for Mathematics and Applications (NOVA Math) and Department of Mathematics, NOVA School of Science and Technology (NOVA SST),	Portugal.}
\email[Fernanda Cipriano]{cipriano@fct.unl.pt}

\allowdisplaybreaks[1]

\usepackage{comment}
\begin{document}

	\begin{abstract}
			In this article, we consider a class of incompressible stochastic third-grade fluids (non-Newtonian fluids) equations on two- as well as three-dimensional Poincar\'e domains $\mathcal{O}$ (which may be bounded or unbounded). Our aims are to study the well-posedness and asymptotic analysis for the solutions of the underlying system. Firstly, we prove that the underlying system defined on $\mathcal{O}$ has a unique weak solution (in the analytic sense) under Dirichlet boundary condition and it also generates random dynamical system $\Psi$. Secondly, we consider the underlying system on bounded domains. Using the compact Sobolev embedding $\H^1(\mathcal{O}) \hookrightarrow\L^2(\mathcal{O})$, we prove the existence of a unique random  attractor for the underlying system on bounded domains with external forcing in $\H^{-1}(\O)$. Thirdly, we consider the underlying system on unbounded Poincar\'e domains with external forcing in $\L^{2}(\O)$ and show the existence of a unique random attractor. In order to obtain the existence of a unique random attractor on unbounded domains, due to the lack of compact Sobolev embedding $\H^1(\mathcal{O}) \hookrightarrow\L^2(\mathcal{O})$, we use the uniform-tail estimates method which helps us to demonstrate the asymptotic compactness of $\Psi$. 
		 Note that due to the presence of several nonlinear terms in the underlying system, we are not able to use the energy equality method to obtain the asymptotic compactness of $\Psi$ in unbounded domains, which makes the analysis of this work in unbounded domains more difficult and interesting. Finally, as a consequence of the existence of random attractors, we address the existence of invariant measures for underlying system.  To the best of authors' knowledge, this is the first work which consider a class of the 2D as well as 3D incompressible stochastic third-grade fluids equations and establish the existence of random attractor in bounded as well as unbounded domains. In addition, this is the first work which address the existence of invariant measures for underlying system on unbounded domains.
  	\end{abstract}

	\maketitle
			\begin{footnotesize}	
			\hspace*{0.45cm}\textbf{MSC 2020:}   76A05, 35R60, 35Q35, 37L30, 35B40. \\
		\end{footnotesize}

\section{Introduction}
 In this work, we are concerned with the well-posedness and the existence of random attractors for a stochastic version of a class of incompressible non-Newtonian fluids filling a  two- or  three-dimensional domain (which may be bounded or unbounded) under  Dirichlet	boundary condition. The majority of research on fluid dynamics has focused on Newtonian fluids, which are defined by a linear relationship between shear stress and strain rate. As a result, these fluids are modeled by the Navier-Stokes equations, which have been the subject of substantial mathematical and physical study. Nonetheless, many actual and industrial fluids exhibit nonlinear viscoelastic behaviour that defies Newton's equation and, as a result, are not amenable to the description provided by the traditional viscous Newtonian fluid models. These fluids include natural biological fluids such as blood, geological flows  and others, see the works \cite{DR95,FR80,yas+fer_2023} and their references.  As such, consideration of more broad fluid models is required. Non-Newtonian viscoelastic fluids of differential type have received a lot of attention lately; see for instance \cite{Cioran2016}.	However, in order to comprehend and explain the properties of various nanofluids, a number of simulation studies have been conducted using third-grade fluid models (see for instance \cite{PP19,RHK18} and the references therein). Nanofluids are engineered colloidal suspensions of nanoparticles in a base fluid, such as water, ethylene glycol, or oil, and they exhibit enhanced thermal conductivity in comparison to the base fluid, which shows great promise for use in technology and microelectronics. Thus, in order to comprehend the behaviors of these fluids, it is crucial to perform a mathematical analysis of third-grade fluids equations.

\subsection{The underlying system} Let us now briefly recall how to obtain the fluids equations for non-Newtonian fluids of differential type,	for	more details about Kinematics of such fluids we	refer to \cite{Cioran2016}. Denote  the velocity field of the fluid by $\v$  and introduce the Rivlin-Ericksen kinematic tensors   $\Arm_n, n\geq 1$, (see \cite{RE55}), defined by  
\begin{align*}
	\Arm_1(\v)&=\nabla \v+(\nabla \v)^T;	\,\Arm_n(\v)=\dfrac{\d}{\d t} \Arm_{n-1}(\v)+\Arm_{n-1}(\v)(\nabla \v)+(\nabla \v)^T\Arm_{n-1}(\v), \quad n=2,3,\cdots.
\end{align*}

The constitutive law of fluids of grade $n$ reads $\mathbb{T}=-p\mathrm{I} + F(\Arm_1,\cdots,\Arm_n),$ where $\mathbb{T}$ is the Cauchy stress tensor, $p$ is the pressure and $F$ is an isotropic polynomial function of degree $n$, subject to the usual	requirement of material frame indifference,	see	\cite{Cioran2016}. The constitutive law of   third-grade fluids	$(n=3)$ is given by the following equation	 
\begin{align*}
	\mathbb{T}=-p\mathrm{I}+\nu \Arm_1+\alpha_1\Arm_2+\alpha_2\Arm_1^2+\beta_1 \Arm_3+\beta_2(\Arm_1\Arm_2+\Arm_2\Arm_1)+\beta_3\Tr(\Arm_1^2)\Arm_1,
\end{align*}
where $\nu$ is the viscosity and $\alpha_1, \alpha_2, \beta_1, \beta_2, \beta_3$ are material moduli. We recall that the momentum equations, established by the Newton's second law, are given by
$$\dfrac{\mathrm{D}\v}{\mathrm{D}t}=\dfrac{\d\v}{\d t}+\v\cdot \nabla \v=\text{div}(\mathbb{T}).$$ 
If $\beta_1=\beta_2=\beta_3=0$, the constitutive equations correspond to second grade fluids. It has been shown that the Clausius-Duhem inequality and the assumption that the Helmholtz free energy is a minimum in equilibrium requires the viscosity and material moduli to satisfy
\begin{align}\label{secondlaw}
	\nu \geq 0,\quad \alpha_1+\alpha_2=0, \quad \alpha_1\geq 0. 
\end{align}
Although second grade fluids are mathematically more treatable,	dealing with several non-Newtonian fluids, the rheologists have not confirmed these restrictions \eqref{secondlaw}, thus give  the conclusion that the fluids that have been tested are not fluids of second grade but are fluids that are characterized by a different constitutive structure, we refer to \cite{FR80} and references therein for more details.  Following \cite{FR80}, in order to allow the motion of the fluid to be compatible with thermodynamic, it should be imposed that
\begin{equation}\label{third-grade-paremeters}
	\nu \geq 0, \quad \alpha_1\geq 0, \quad |\alpha_1+\alpha_2 |\leq \sqrt{24\nu\beta}, \quad \beta_1=\beta_2=0, \beta_3=\beta \geq 0.
\end{equation}	
Consequently, the velocity field $\v$  satisfies
the incompressible third-grade fluids equations
\begin{equation}
	\label{third-grade-fluids-equations}
	\left\{\begin{aligned}
		\partial_t(z(\v))-\nu \Delta \v+(\v\cdot \nabla)z(\v)&+\displaystyle\sum_{j=1}^dz(\v)^j\nabla v^j-(\alpha_1+\alpha_2)\text{div}((\Arm(\v))^2)\\ -\beta \text{div}[\Tr(\Arm(\v)\Arm(\v)^T)\Arm(\v)]
		& = -\nabla \mathbf{P} + \f,  \\[0.2cm]	\text{div}\;\v& =0,	  \\[0.2cm]
		z(\v)&:=\v-\alpha_1\Delta \v, \\[0.2cm]
		\Arm(\v) & := \nabla \v+(\nabla \v)^T,
	\end{aligned}\right.
\end{equation}
where the viscosity $\nu$ and the material moduli	 $\alpha_1,\alpha_2$, $\beta$ 	verify	\eqref{third-grade-paremeters},	  $\mathbf{P}$ denotes the pressure	and	 $\f$	denotes	an	external	force. 	Notice	that	if	$\alpha_1=\alpha_2=0$	and $\beta$=0,	we	recover		the	Navier	Stokes	equations (NSE). From mathematical point of view,  fluids of grade  $n$ constitute an hierarchy of fluids with  increasing complexity and more  nonlinear terms,  then comparing with Newtonian (grade $1$) or second grade fluids, third-grade fluids are more complex  and require more involved analysis.

  Let $\mathcal{O}$ be an open and connected subset of $\R^d$ ($d\in\{2,3\}$). Our	aim	in	this	work	is	to	study the following stochastic version of system	\eqref{third-grade-fluids-equations}	with a subset of physical conditions \eqref{third-grade-paremeters}, namely
  	\begin{equation}\label{condition1}
  		\nu > 0, \quad \alpha_1= 0, \ \alpha_2=\alpha, \quad |\alpha | < \sqrt{2\nu\beta}, \quad \beta_1=\beta_2=0, \beta_3=\beta > 0,
  	\end{equation}	
  in the presence of infinite-dimensional Wiener noise:
\begin{equation}\label{equationV1}
		\left\{ 
		\begin{aligned}
			\d \v & = [\f-\nabla \textbf{P}+\nu \Delta \v-(\v\cdot \nabla)\v+\alpha\text{div}((\Arm(\v))^2) +\beta \text{div}(|\Arm(\v)|^2\Arm(\v)) ]\d t && \\ & \quad + \d\mathrm{W}  && \hspace{-15mm} \text{in }  \mathcal{O} \times (0,\infty),\\
			\text{div}\; \v&=0 \quad && \hspace{-15mm} \text{in } \mathcal{O} \times [0,\infty),\\
			\v &= \boldsymbol{0} && \hspace{-15mm} \text{on } \partial\mathcal{O}\times [0,\infty),\\
			\v(x,0)&=\x  \quad && \hspace{-15mm} \text{in } \mathcal{O},
		\end{aligned}
		\right.
	\end{equation}
where $\{\W (t)\}_{t \in \R}$ is a two-sided cylindrical Wiener process in $\H$ (see Subsection \ref{sec-functional-setting} below for the function spaces) with its Cameron-Martin space/Reproducing Kernel Hilbert Space (RKHS) $\mathrm{K}$ satisfying Hypothesis \ref{assump1} below, defined on some probability space $(\Omega,\mathcal{F},\mathbb{P})$. 
\begin{remark}\label{rem-condition1}
	Note that \eqref{condition1} allow us to define $\varepsilon_0:=1-\frac{\alpha^2}{2\nu\beta}\in (0,1)$.
\end{remark}

\subsection{Literature} 
We first recall some initial works which deals with system \eqref{third-grade-fluids-equations} with $\alpha_1>0$. In  \cite{AC97}, the authors established the existence	of local (in time) solution in $\H^3(\O)$ to system \eqref{third-grade-fluids-equations} on bounded domains with Dirichlet  boundary condition. Later the authors in \cite{Bus-Ift-1} proved	the	global (in time) existence in $\mathbb{R}^d,$ $d\in\{2,3\}$ for $\H^2$-valued solution and uniqueness in	$\R^2$. Note that the uniqueness result in 3D for $\H^2$-valued	solution is still an open	question. In \cite{Bus-Ift-2}, the authors considered system 	\eqref{third-grade-fluids-equations} with a Navier-slip boundary condition and demonstrated the existence of a global (in time) solution for initial conditions in $\H^2(\R^d)$, $d\in\{2,3\}$. In addition, they also proved that uniqueness holds in 2D.  We highlight that the complex nonlinearities in \eqref{third-grade-fluids-equations} make it difficult to develop a solution with less regular initial data, and that in order to obtain any results, one must impose an additional limitation on the parameters $\alpha_1,\alpha_2,\beta$, and $\nu$. In fact, the author in \cite{Paicu2008} demonstrated the existence of a global weak solution for \eqref{third-grade-fluids-equations} in $\R^d$, $d\in\{2,3\}$, under some additional restrictions on the parameters, which permits the application of a monotonicity techniques when the initial data belong only to $\H^1(\R^d)$ and $\alpha_1>0$ (\cite{Busuioc-Iftimie-Paicu08} for the stationary case). Moreover, the author established a weak-strong uniqueness result and the validity of the energy equality.

Now, let us focus on the literature related to system \eqref{equationV1} (that is, \eqref{third-grade-fluids-equations} with $\alpha_1=0$). Firstly, we draw reader's attention to the groundbreaking work \cite{Ladyzhenskaya67}, where the author introduced a new system for incompressible viscous fluids with the viscosity depending on the velocity gradient. This  in fact has similar nonlinear terms as system \eqref{equationV1}. In \cite{Hamza+Paicu_2007},	the	authors	proved	the	global	well-posedness to system \eqref{equationV1}	in $\mathbb{R}^3$	with divergence-free initial data belongs to $\L^2(\mathbb{R}^3)$ using a monotonicity method and	some extra	restriction	on the parameters $\nu$, $\alpha$ and $\beta$. We refer readers to the works \cite{Jia+Xie+Duo_2017,Paicu2008}, etc. for the analysis to system \eqref{equationV1}. Some of the stochastic versions of system \eqref{third-grade-fluids-equations} (for $\alpha_1>$ as well as $\alpha=0$) have been studied in \cite{Cip+Did+Gue_2021,Raya+Fernanda+Yassine_Arxiv,yas+fer_2023}, etc.

Examining the asymptotic behavior of dynamical systems is one of the most significant and thorough areas of mechanics and mathematical physics. The concept of attractors is essential to the study of deterministic infinite-dimensional dynamical systems theory (see the books by Temam \cite{R.Temam}, Robinson \cite{Robinson2}, Chepyzhov and Vishik \cite{Chepyzhov+Vishik_2002} and many others). A key topic in the study of random dynamics of SPDEs is obtaining confirmation that an SPDE generates a stochastic flow or random dynamical system (RDS). It is widely known in the literature that RDS can be generated for It\^o stochastic ordinary differential equations and a wide family of PDEs with stationary random coefficients (see \cite{Arnold,PEK}, etc.).  Another crucial area of research into the qualitative characteristics of SPDEs is the examination of infinite-dimensional RDS (see the works \cite{BCF,CF,Crauel}, etc.).

The existence of a random attractor in bounded domains depends on the compactness of the embedding $\H^1(\mathcal{O})\hookrightarrow \L^2(\mathcal{O})$, which facilitates analysis (see \cite{CDF,KM3}, etc.). In the case of unbounded domains, however, the embedding $\H^1(\mathcal{O}) \hookrightarrow\L^2(\mathcal{O})$ is no longer compact. Consequently, we cannot use the compactness requirement to demonstrate the existence of random attractors. This challenge (in unbounded domains) was addressed in the deterministic case using several approaches; see \cite{Abergel,Rosa_1998}, etc., for the autonomous case, and \cite{Caraballo+Lukaszewicz+Real_2006_CRM,Caraballo+Lukaszewicz+Real_2006}, etc., for the non-autonomous case. A number of authors have also generalized the deterministic case approaches for SPDEs (see, for instance, \cite{BLW,Brzezniak+Li_2006,Wang_2012}, etc.). This idea has since been used to demonstrate the existence of random attractors for a number of SPDEs, including the stochastic NSE on the 2D unit sphere (\cite{BGT}), the stochastic $g$-NSE (\cite{FY}), the stochastic reaction-diffusion equations (\cite{BLW}), the 3D stochastic Benjamin-Bona-Mahony equations (\cite{Wang_2008}), the 2D and 3D stochastic convective Brinkman-Forchheimer equations (CBFE) \cite{Kinra+Mohan_2024_DCDS,KM7}, the 3D globally modified stochastic NSE \cite{Kinra+Mohan_GMNSE_IN} and references therein.

\subsection{Scope of the article, approaches and assumptions}
Broadly, this article  is divided into three main parts.  Our first aim is to establish the existence and uniqueness of weak solutions (in the analytic sense) of system \eqref{equationV1}. We begin by defining an Ornstein-Uhlenbeck process which takes values in $\H \cap {\mathbb{W}}^{1,4}(\mathcal{O})$. In order to define a suitable Ornstein-Uhlenbeck process, one of the technical difficulties related to the noise is resolved by using the corresponding  RKHS.  In particular, motivated from  \cite{BL1,Brzezniak+Li_2006,BCLLLR} for 2D NSE on unbounded Poincar\'e domains, we assume that the RKHS $\mathrm{K}$ satisfies the following hypothesis:
\begin{hypothesis}\label{assump1}
	$ \mathrm{K} \subset \H \cap {\mathbb{W}}^{1,4}(\mathcal{O}) $ is a Hilbert space such that for some $\delta\in (0, 1/2),$
	\begin{align}\label{A1}
		\A^{-\delta} : \mathrm{K} \to \H \cap {\mathbb{W}}^{1,4}(\mathcal{O})  \  \text{ is }\ \gamma \text{-radonifying.}
	\end{align}
\end{hypothesis}
We refer readers to Section \ref{sec2} for necessary functional settings.
\begin{remark}
	1. Let us discuss an example which satisfies Hypothesis \ref{assump1}. Since $\D(\A^{1+\frac{r}{2}})\hookrightarrow\H \cap {\mathbb{W}}^{1,4}(\mathcal{O})$ for any $r>0$, one can consider $\mathrm{K}\subset \D(\A^{1+\frac{r}{2}})$ for some $r>0$ as an example, and Hypothesis \ref{assump1} can be reformulated in the following way (see \cite{BL1}).  Let us choose $r>0$ and fix it. $\mathrm{K}$ is a Hilbert space such that $\mathrm{K}\subset  \D(\A^{1+\frac{r}{2}})$ and, for some $\delta\in(0,1/2)$, the map 
	\begin{align}\label{34}
	\A^{-\delta-1-\frac{r}{2}} : \mathrm{K} \to \H \   \text{ is }\ \gamma \text{-radonifying.}
\end{align}
The condition  \eqref{34} also says that the mapping $\A^{-\delta-1-\frac{r}{2}} : \mathrm{K} \to \H$  is Hilbert-Schmidt.	If $\mathcal{O}$ is a bounded domain, then $\A^{-s}:\H\to\H$ is Hilbert-Schmidt if and only if $\sum_{j=1}^{\infty} \lambda_j^{-2s}<+\infty,$ where  $\A e_j=\lambda_j e_j$, $j\in \N$ and $\{e_j\}_{j\in\N}$ is an orthogonal basis of $\H$. In bounded domains, it is well known that $\lambda_j\sim j^{\frac{2}{d}},$ for large $j$ (growth of eigenvalues) and hence $\A^{-s}$ is Hilbert-Schmidt if and only if $s>\frac{d}{4}.$ In other words, with $\mathrm{K}=\D(\A^{s+1+\frac{r}{2}}),$ the embedding $\mathrm{K}\hookrightarrow\D(\A^{1+\frac{r}{2}})\hookrightarrow\H \cap {\mathbb{W}}^{1,4}(\mathcal{O})$ is $\gamma$-radonifying if and only if $s>\frac{d}{4}.$ Thus, Hypothesis \ref{assump1} is satisfied for any $\delta>0.$ In fact, the condition \eqref{A1} holds if and only if the operator $\A^{-(s+1+\frac{r}{2}+\delta)}:\H\to \H \cap {\mathbb{W}}^{1,4}(\mathcal{O})$ is $\gamma$-radonifying. 

2.	 The requirement of  $\delta<\frac{1}{2}$ in Hypothesis \ref{assump1} is necessary because we need (see Subsection \ref{O-Up} below) the corresponding Ornstein-Uhlenbeck process to take values in $\H \cap {\mathbb{W}}^{1,4}(\mathcal{O})$.
\end{remark}

Making use of the Ornstein-Uhlenbeck process, we define a Doss-Sussmann transformation \cite{Doss_1977,Sussmann_1978} (see \eqref{D-S_trans} below) and obtain an equivalent pathwise deterministic system (see   system \eqref{CTGF-with-Pressure} below) to   system \eqref{equationV1}. Next, we show the existence of a unique weak solution to the equivalent pathwise deterministic system \eqref{CTGF-with-Pressure} via a monotonicity argument and a Minty-Browder technique. The following monotonicity result of the operator $\mathscr{G}(\cdot):=\nu\A\cdot+\B(\cdot+\z)+\alpha\J(\cdot+\z)+\beta\K(\cdot+\z)$ plays a crucial role: 
\begin{align*}
	&	\left< \mathscr{G}(\y_1)-\mathscr{G}(\y_2), \y_1-\y_2   \right> + \frac{(C_{S,3} C_K)^2}{\nu\varepsilon_0} \|\Arm(\y_2+\z)\|_{{4}}^2 \|\y_1-\y_2\|_{2}^2 
	\nonumber\\ &  \geq \frac{\nu\varepsilon_0}{4} \|\Arm(\y_1-\y_2)\|^2_{2} +  \frac{\beta\varepsilon_0}{4}\|\Arm(\y_1-\y_2)\|^4_{4} \geq 0,
\end{align*}
for any  $\y_1,\y_2\in\X$ and $\z\in\mathbb{W}^{1,4}(\O)$ (see Lemma \ref{lem-Loc-Monotinicity} for complete details). The global solvability of   system \eqref{CTGF-with-Pressure} helps to establish the existence and uniqueness of weak  solutions to   system \eqref{equationV1}.  We point out here that there is no earlier paper which discuss stochastic third-grade fluids equations on unbounded domains.

Secondly, we are interested in the large-time behavior of solutions of system \eqref{equationV1}, that is, the existence of random attractors for system \eqref{equationV1}. Similar to the case of NSE, it is a challenging problem to study the long-time behavior of system \eqref{equationV1} on the whole space $\R^d$. Nevertheless, one can try these kinds of analysis on Poincar\'e domains which may be bounded or unbounded, see for instance \cite{Brzezniak+Li_2006,BCLLLR,Kinra+Mohan_GMNSE_IN,Wang+Kinra+Mohan_2023}, etc. By a Poincar\'e  domain, we mean a domain in which the Poincar\'e inequality is satisfied. More precisely, we consider the following hypothesis on the domain $\mathcal{O}$:
\begin{hypothesis}\label{assumpO}
	Let $\mathcal{O}$ be an open and connected subset (may be bounded or unbounded) of $\R^d$, $d\in\{2,3\}$, the boundary of which is
uniformly of class $\mathrm{C}^1$ (see \cite{Farwig+Kozono+Sohr_2007}). We assume that, there exists a positive constant $\lambda $ such that the following Poincar\'e inequality  is satisfied:
	\begin{align}\label{poin}
		\lambda\int_{\mathcal{O}} |\psi(x)|^2 \d x \leq \int_{\mathcal{O}} |\nabla \psi(x)|^2 \d x,  \ \text{ for all } \  \psi \in \H^{1}_0 (\mathcal{O}).
	\end{align}
\end{hypothesis}

\begin{example}
	A typical example of a unbounded Poincar\'e domain in $\R^2$ and $\mathbb{R}^3$ are $\O=\R\times(-L,L)$ and $\O=\R^2\times(-L,L)$ with $L>0$, respectively, see \cite[p.306]{R.Temam} and \cite[p.117]{Robinson2}. 
\end{example}
\begin{remark}
	When $\mathcal{O}$ is a bounded domain, Poincar\'e inequality is satisfied automatically with $\lambda=\lambda_1$, where $\lambda_1$ is the first eigenvalue of the Stokes operator defined on bounded domains. 
\end{remark}

 In this work, we apply the abstract theory established in \cite{BCLLLR} to demonstrate the existence of random attractors for system \eqref{equationV1}. In order to use the abstract result, we first show that system \eqref{equationV1} generates random dynamical system $\Psi$ (see Section \ref{sec3}). Next, we need to show that $\Psi$ has a random absorbing set and it is asymptotically compact. We show the existence of random absorbing set using the uniform energy estimates. On the other hand, on bounded domains, we use the compact Sobolev embedding $\H^1(\O) \hookrightarrow\L^2(\O)$ to obtain asymptotic compactness of $\Psi$. In unbounded domains, due to lack of compact Sobolev embedding $\H^1(\O) \hookrightarrow\L^2(\O)$, the commonly use methods in literature are the following:
\begin{itemize}
	\item [1.] the energy equality method introduced by Ball \cite{Ball};
	\item [2.] the uniform-tail estimate method introduced by Wang \cite{Wang_1999}.
\end{itemize}
Note that unlike the Navier-Stokes equations (see \cite{Brzezniak+Li_2006,BCLLLR}), convective Brinkman-Forchheimer equations (see \cite{Kinra+Mohan_2024_DCDS,Kinra+Mohan_2023_DIE}, etc.), etc. we are not able to show that the solutions of system \eqref{equationV1} satisfy the \textit{weak continuity} property with respect to the initial data. Due to this reason, we are not able to use the energy equality method to obtain the asymptotic compactness of $\Psi$ on unbounded domains. Therefore, for unbounded domains, we explore the uniform-tail estimate method to prove the asymptotic compactness of $\Psi$ (Theorem \ref{Main_theorem_2}).

\begin{remark}
	 In the context of attractor when $\O$ is a bounded domains, we have considered $\f\in \H^{-1}(\mathcal{O}) $. But for unbounded domain case, we need to restrict the forcing term to be in a regular space, that is, $\f\in \L^{2}(\mathcal{O})$. If one proves the weak continuity of the solutions with respect to the initial data of system \eqref{equationV1}, it will be possible to work with the same regularity as in bounded domains via energy equality method (see the works \cite{BCLLLR,Kinra+Mohan_2024_DCDS}, etc. for Navier-Stokes equations and related models). Unfortunately, we are not able to prove the weak continuity of the solutions of system \eqref{equationV1} with respect to the initial data, so the method of energy equality is not applied.
\end{remark}

\begin{remark}
Comparing with the Navier-Stokes equations, system \eqref{equationV1} has stronger nonlinear terms, nevertheless due to its symmetry, the nonlinear term
$\text{div}(|\Arm(\v)|^2\Arm(\v))$ has an important regularizing effect, ensuring uniqueness even in 3D. At first glance, one might expect that system \eqref{equationV1} would also possess the weak continuity property, however this property is not easy to show.
 To clarify this issue, let $\v_n(\cdot)$ denote the unique solution of system \eqref{equationV1} corresponding to the initial data $\v_{0,n}$, for $n \in \N$. By definition, weak continuity requires that if $\v_{0,n}$ converges weakly to $\v_0$ in $\H$, then $\v_n(\cdot)$ converges weakly to $\v(\cdot)$ in $\H$, where $\v(\cdot)$ is the unique solution of system \eqref{equationV1} associated with initial data $\v_0$. However, due to the gradient-type nonlinearity, it is necessary to  apply the monotonicity method to pass to the limit $n \to \infty$ in \eqref{equationV1}, which  requires strong convergence of the initial data. The monotonicity method does not appear suitable for establishing weak continuity of solutions to system \eqref{equationV1}. In the Navier-Stokes equations case, the limit can be performed without relying on monotonicity method.
\end{remark}

Essentially unlike the parabolic or hyperbolic equations as considered in \cite{CGTW,LGL,Wang_1999,rwang1}, etc. the fluids equations like \eqref{equationV1} contain the pressure term $\textbf{P}$. When we derive these uniform tail-estimates, the pressure term $\textbf{P}$ can not be simply treated by the divergence theorem (does not vanish). We derive uniform-tail estimates by estimating all nonlinear terms in system \eqref{equationV1} and pressure $\textbf{P}$ carefully, see Lemma \ref{LR} below. As a  result of uniform-tail estimates, the asymptotic compactness of $\Psi$ on unbounded domains follows (see Theorem \ref{Main_theorem_2}).

Our final aim is to demonstrate the existence of an invariant measure for the SGMNS equations \eqref{equationV1}. Crauel et. al. in \cite{CF} demonstrated that a sufficient condition for the existence of invariant measures is the existence of a compact invariant random set. They implemented this idea to demonstrate the existence of invariant measures in bounded domains for 2D stochastic NSE and reaction-diffusion equations. Moreover, in unbounded domains, this idea was also used to establish the existence of invariant measures for 2D stochastic NSE by Brz\'ezniak et. al. (\cite{Brzezniak+Li_2006}) and Kinra et. al. (\cite{KM8}), for 2D and 3D stochastic CBFE by Kinra et. al. (\cite{KM2,Kinra+Mohan_2024_DCDS}), for stochastic globally modified NSE by Kinra at. al. \cite{Kinra+Mohan_GMNSE_IN}. Since we show the existence of a random attractor,  the existence of an invariant measures is guaranteed as the random attractor itself is a compact invariant set.

Let us now state our main results  whose proofs follow from Theorems \ref{STGF-Sol}, \ref{Main_theorem_1}, \ref{Main_theorem_2} and \ref{thm6.3}. 
\begin{theorem}
	Under \eqref{condition1} and Hypotheses \ref{assump1}-\ref{assumpO}, 
	\begin{enumerate}
		\item [(i)] there exists a unique weak solution (in the analytic sense) for system \eqref{equationV1} with external forcing $\f\in \H^{-1}(\O)$,
		\item [(ii)] on bounded domains, there exists a unique random attractor for system \eqref{equationV1} with external forcing $\f\in \H^{-1}(\O)$,
		\item [(iii)] on unbounded domains, there exists a unique random attractor for system \eqref{equationV1} with external forcing $\f\in \L^{2}(\O)$,
		\item [(iv)] on bounded and unbounded domains, there exists an invariant measure for system \eqref{equationV1} with external forcing $\f\in \H^{-1}(\O)$ and $\f\in \L^{2}(\O)$, respectively.
	\end{enumerate}
\end{theorem} 

We remark that to the best of authors' knowledge, the large-time behavior of the solutions to the subclass of stochastic third-grade fluids equations \eqref{third-grade-fluids-equations} under analysis (see system \eqref{equationV1}) has not been investigated in bounded/unbounded domains till the date. In addition, the existence of invariant measures for this subclass of stochastic third-grade fluids equations \eqref{third-grade-fluids-equations} on unbounded domains is being studied here for the first time. It is important to mention that in  2D and 3D bounded domains, this class of fluids has been considered  in the article \cite{yas+fer_2025}, where the authors proved well-posedness and existence of an ergodic invariant measure for multiplicative Lipschitz noise but not discussed the random dynamics of solutions. We note that, in the case of bounded domains, a direct consequence of the existence of a random attractor yields a particular case of the result established in  \cite{yas+fer_2025}. Unfortunately, we are unable to demonstrate the uniqueness of invariant measures, which is a subject we plan to tackle in the future.

\begin{remark}
It is worth mentioning that the notion of pathwise random attractors was first introduced in \cite{CF,FS}, and subsequently, several authors employed this framework to establish the existence of random attractors for various SPDEs. However, it has been consistently observed that the existing works on pathwise random attractors confine
the noise to either additive form, linear multiplicative form, or to coefficients with specific structural properties such as antisymmetry (see \cite{PEK+TL,Wang_2019}). This restriction arises because, in such settings, a random dynamical system (or random cocycle) can be defined via random transformations, most notably the Doss-Sussmann transformation \cite{Sussmann_1978}. To the best of our knowledge, no results are available in the literature regarding the existence of pathwise random attractors for SPDEs with general Lipschitz nonlinear diffusion coefficients, an issue that poses a significant challenge. To address this difficulty, the concepts of mean-square and weak mean-square random attractors were introduced in \cite{PEK+TL}, although these approaches remain quite restrictive in practice (see \cite{Wang_2019} for further discussion). More recently, to handle Lipschitz nonlinear diffusion terms, the concept of weak pullback mean random attractors in spaces of Bochner integrable functions was developed in \cite{Wang_2019}, where the long-term dynamics of stochastic reaction-diffusion equations with both nonlinear drift and nonlinear diffusion were studied. The existence of pathwise random attractors for noise beyond the additive case, as well as the existence of weak pullback mean random attractors in the presence of nonlinear noise coefficients, remains an interesting open problem for the underlying system. These issues will be addressed in future work.
\end{remark}

\subsection{Organization of the article}
The organization of this article is as follows: In the next section, we provide necessary function spaces, some important inequalities, linear and nonlinear operators with their properties for abstract formulation of system \eqref{equationV1} and further analysis and abstract formulation of system \eqref{equationV1}. The metric dynamical system (MDS) and RDS corresponding to system \eqref{equationV1} are constructed in Section \ref{sec3} using a random transformation (known as Doss-Sussmann transformation, \cite{Doss_1977,Sussmann_1978}). The existence and uniqueness of a weak solution satisfying the energy equality to the transformed system (see   system \eqref{CTGF} below) by using monotonicity arguments and a Minty-Browder technique has also been established (Theorem \ref{solution}) in the same section. Section \ref{sec4} provides the existence of a unique random attractor for system \eqref{equationV1} on  bounded domains via compact Sobolev embeddings and abstract theory developed in \cite{BCLLLR} (see Lemma \ref{PCB} and Theorem \ref{Main_theorem_1}). In Section \ref{sec5}, we consider the stochastic system \eqref{equationV1} on unbounded domains. In order to do this, we first present Lemmas \ref{RA1-ubd} and \ref{D-abH-ubd}, which provides us the energy estimates for the solution of   system \eqref{CTGF} and helps us to prove the existence of random absorbing set.  Then, we present Lemma \ref{D-convege} which is one of the key lemma to prove the $\mathfrak{DK}$-asymptotic compactness of RDS. Next we prove the uniform-tail estimates for the solution of system \eqref{CTGF} which is another key result to prove the  $\mathfrak{DK}$-asymptotic compactness of RDS (see Lemma \ref{LR}). Finally, we establish the main result of this section, that is, the existence of unique random attractor for system \eqref{equationV1} on unbounded domains using Lemmas \ref{D-abH-ubd}, \ref{D-convege} and \ref{LR}, and abstract theory developed in \cite{BCLLLR}, see Theorem \ref{Main_theorem_2}. In the final and smallest section, we address the existence of invariant measures for system \eqref{equationV1} as a consequence of the existence of random attractor (see Theorem \ref{thm6.3}). In Appendix \ref{PR}, we have shown that one can retrieve the pressure from the abstract formulation of system \eqref{equationV1} using the approach discussed in \cite[Section 8]{SS11}.

\section{Mathematical formulation and some auxiliary results}\label{sec2}

In this section, we first provide necessary function spaces, inequalities and notations to the analysis of this article. Then, we define linear and nonlinear operators with their properties. Finally, using the operators, we provide the abstract formulation of   system \eqref{equationV1} and the definition of weak solution in the analytic sense.

\subsection{Notations and the functional setting }\label{sec-functional-setting}
Let	$T>0$,	for a Banach space $E$, we define
$$  (E)^k:=\{(f_1,\cdots,f_k): f_l\in E,\quad l=1,\cdots,k\}\;\text{ for	positive integer }	k.$$

Let	$m\in	\mathbb{N}^*$	and	$1\leq	p<	\infty$,	we	denote	by	$\mathbb{W}^{m,p}(\mathcal{O})$ (resp. $\mathrm{W}^{m,p}(\mathcal{O})$)	the	standard Sobolev space of matrix/vector-valued (resp. scalar-valued) functions	whose weak derivative up to order $m$	belong	to	the	Lebesgue	space	$\L^p(\mathcal{O})$ (resp. $\mathrm{L}^p(\mathcal{O})$) and set	$\H^m(\mathcal{O})=\mathbb{W}^{m,2}(\mathcal{O})$	and	$\H^0(\mathcal{O})=\L^2(\mathcal{O})$.

Let us introduce the following spaces:
\begin{equation}
\begin{array}{ll}
\mathcal{V}&:=\{\v \in (\mathrm{C}^\infty_c(\mathcal{O}))^d \,: \text{ div }\v=0\},\\ [1mm]
\H&:=\text{The	closure	of	}\mathcal{V}	\text{	in	}\L^2(\mathcal{O})	=\{ \v \in \L^2(\mathcal{O}) \,: \text{ div } \v=0 \text{ in } \mathcal{O} \;\; \text{and} \;\; \v\cdot\n=0 \text{ on }\partial\mathcal{O} \}, \nonumber\\[1mm]
\V&:=\text{The	closure	of	}\mathcal{V}	\text{	in	}\H^1(\mathcal{O})=\{ \v \in \H^1_0(\mathcal{O}) \,: \text{ div } \v=0 \text{ in } \mathcal{O}  \},\\
{\L}^{p}_{\sigma} &:=\text{the closure of }\ \mathcal{V} \ \text{ in } \L^p(\mathcal{O}), \;	 \text{  for  }  p\in(1,\infty),\\
	\mathbb{V}_s&:=\text{the closure of }\ \mathcal{V} \ \text{ in the Sobolev space } \H^s(\mathcal{O}),
\ \text{  for  } s>1,
\end{array}
\end{equation}
 where $\boldsymbol{n}$ is the unit outward drawn normal to $\partial\mathcal{O}$,  and $\v\cdot\n\big|_{\partial\mathcal{O}}$ should be understood in the sense of trace in $\H^{-1/2}(\partial\mathcal{O})$ (cf. Theorem 1.4 and Remark 1.5, Chapter 1, \cite{Temam_1984}).

Now, let us introduce  the scalar product between two matrices $A:B=\Tr(AB^T)$
and denote $\vert A\vert^2:=A:A.$
The divergence of a  matrix $A\in \mathcal{M}_{d\times d}(E)$ is given by 
$(\text{div}(A)_i)_{i=1}^{i=d}=(\displaystyle\sum_{j=1}^d\partial_ja_{ij})_{i=1}^{i=d}.$
The space $\H$ is endowed with the  $\L^2$-inner product $(\cdot,\cdot)$ and the associated norm $\Vert \cdot\Vert_{2}$. We recall that
\begin{align*}
(\u,\v)&=\sum_{i=1}^d\int_{\mathcal{O}}\u_i\v_i \d x, \quad  \text{ for all } \u,\v \in \L^2(\mathcal{O})
\end{align*}
and 
\begin{align*}
	(A,B)&=\int_{\mathcal{O}} A : B \d x ; \quad  \text{ for all } A,B \in \mathcal{M}_{d\times d}(L^2(\mathcal{O})).
\end{align*}
In view of \eqref{poin}, on the functional space $\V$, we will consider  the following inner product
\begin{align*}
    (\u,\v)_{\V}:= (\nabla	\u,\nabla	\v),
\end{align*}
and denote by $\Vert \cdot\Vert_{\V}$ the corresponding norm. The usual norms on the classical Lebesgue and Sobolev spaces $\mathrm{L}^p(\mathcal{O})$ (resp. $\mathbb{L}^p(\mathcal{O})$) and $\mathrm{W}^{m,p}(\mathcal{O})$ (resp. $\mathbb{W}^{m,p}(\mathcal{O})$) will be denoted by $\|\cdot \|_p$ and 
$\|\cdot\|_{\mathrm{W}^{m,p}}$ (resp. $\|\cdot\|_{\mathbb{W}^{m,p}}$), respectively.
In addition, given a Banach space $\mathrm{E}$, we will denote by $\mathrm{E}^\prime$ its dual, and for any Hilbert space $\mathrm{H}$, we represent inner product by $(\cdot,\cdot)_{\mathrm{H}}$. We denote	by	$\langle	\cdot,\cdot\rangle$ the duality pairing between $\mathrm{E}^{\prime}$ and $\mathrm{E}$.

	Let	us	introduce	the	following	Banach	space	$(\X,\Vert	\cdot\Vert_{\X})$
$$ \X=\mathbb{W}^{1,4}(\mathcal{O})\cap \V,	\quad	\text{	with	\;\;	} \Vert	\cdot\Vert_\X:=\Vert	\cdot\Vert_{\mathbb{W}^{1,4}}+\|\cdot\|_{\V}.$$
Indeed,	we	recall	that	$\mathbb{W}^{1,4}(\mathcal{O})$	endowed	with $\Vert	\cdot\Vert_{\mathbb{W}^{1,4}}$-norm	is		Banach	space,	where		$$\Vert	\w \Vert_{\mathbb{W}^{1,4}}^4=\int_{\mathcal{O}}	\vert	\w\vert^4 \d x+\int_{\mathcal{O}} \vert	\nabla	\w\vert^4 \d x.$$

Let us recall an embedding result from \cite[Subsection 2.4]{Kesavan_1989}. 
\begin{lemma}\label{Sobolev-embedding3}
	For $p>d$, there exists a constant $C=C(p)>0$ such that for all $\w\in \mathbb{W}_0^{1,p}(\mathcal{O})$
	\begin{align*}
		\|\w\|_{\infty} \leq C_{S,3} \|\nabla\w\|_{p}.
	\end{align*}
\end{lemma}

 \begin{remark}
		Let us provide some inequalities which will be useful in the sequel. 
		\begin{enumerate}
			\item [(i)]  Gagliardo-Nirenberg inequality  (\cite[Theorem 1]{Nirenberg_1959}) gives (for $d=2,3$)
		\begin{align}\label{Gen_lady-4}
			\|\v\|_{{4}} \leq C \|\v\|^{\frac{4}{4+d}}_{{2}} \|\nabla \v\|^{\frac{d}{4+d}}_{{4}}, \ \v\in \L^{2} (\mathcal{O})\cap\mathbb{W}_0^{1,4}(\mathcal{O}).
		\end{align}
			\item [(ii)] 
			Korn-type inequality (see	\cite[Theorem 2 (ii)]{DiFratta+Solombrino_Arxiv} and \cite[Theorem 4, p.90]{Kondratev+Oleinik_1988}) gives that there	exists a constant $C_K>0$	such	that
			\begin{align}\label{Korn-ineq}
				\Vert	\nabla \w \Vert_{4}	\leq	C_K\Vert	\Arm(\w)	\Vert_4,\quad \text{	for all } \;	\w\in	\mathbb{W}_0^{1,4}(\mathcal{O}).
			\end{align}		
		\end{enumerate}
\end{remark}

For the sake of simplicity, we do not distinguish between scalar, vector or matrix-valued   notations when it is clear from the context. In particular, $\Vert \cdot \Vert_E$  should be understood as follows
\begin{itemize}
	\item $\Vert f\Vert_E^2= \Vert f_1\Vert_E^2+\cdots+\Vert f_d\Vert_E^2$ for any $f=(f_1,\cdots,f_d) \in (E
	)^d$.
	\item $\Vert f\Vert_{E}^2= \displaystyle\sum_{i,j=1}^d\Vert f_{ij}\Vert_E^2$ for any $f\in \mathcal{M}_{d\times d}(E)$.
\end{itemize}

Throughout the article,  we denote by $C$   generic constant, which may vary from line to line.

\subsection{The Helmholtz projection}
The content of this subsection has been inspired from the articles \cite{Farwig+Kozono+Sohr_2005,Farwig+Kozono+Sohr_2007,Kunstmann_2010} which discuss the Helmholtz projection on general unbounded domains with $\mathrm{C}^1$-boundary. Let us define
\begin{align}
    \widetilde{\L}^p(\mathcal{O}) : = \begin{cases}
        \L^p(\mathcal{O})\cap \L^2(\mathcal{O}), & p\in[2,\infty)\\
        \L^p(\mathcal{O})+ \L^2(\mathcal{O}), & p\in (1,2),
    \end{cases}
\end{align}
and 
\begin{align}
    \widetilde{\L}^p_{\sigma}  : = \begin{cases}
        \L^p_{\sigma} \cap \L^2_{\sigma} , & p\in[2,\infty)\\
        \L^p_{\sigma} + \L^2_{\sigma}, & p\in (1,2).
    \end{cases}
\end{align}

Now, we recall that the Helmholtz projector $\mathcal{P}_{p}: \widetilde{\L}^p(\mathcal{O}) \to \widetilde{\L}^p_{\sigma}$, which is a linear bounded operator characterized by the following decomposition
$\v=\mathcal{P}_p\v+\nabla \varphi, \;  \varphi \in \widetilde{\mathrm{G}}^p(\mathcal{O})$ satisfying $(\mathcal{P}_p)^* = \mathcal{P}_{p'}$ with $\frac{1}{p}+\frac{1}{p'}=1$ (cf. \cite{Farwig+Kozono+Sohr_2007}), where
\begin{align}
    \widetilde{\mathrm{G}}^p(\mathcal{O}) : = \begin{cases}
        \mathrm{G}^p(\mathcal{O})\cap \mathrm{G}^2(\mathcal{O}), & p\in[2,\infty)\\
        \mathrm{G}^p(\mathcal{O})+ \mathrm{G}^2(\mathcal{O}), & p\in (1,2),
    \end{cases}
\end{align}
and 
\begin{align}
    \mathrm{G}^p(\mathcal{O}) = \{\nabla q \in \L^p(\mathcal{O}) \; : \; q\in \mathrm{L}^p_{loc} (\mathcal{O})\}.
\end{align}
Let us fix the notations $\widetilde{\mathbb{W}}^{1,p}_{0,\sigma}  := \widetilde{\mathbb{W}}_0^{1,p}(\mathcal{O}) \cap \widetilde{\L}^{p}_{\sigma}$  and $\widetilde{\mathbb{W}}^{-1,p}_{\sigma} : = (\widetilde{\mathbb{W}}^{1,p'}_{0,\sigma})'$ with $\frac{1}{p}+\frac{1}{p'}=1$ for the dual space, where
\begin{align}
    \widetilde{\mathbb{W}}^{1,p}_0(\mathcal{O}) : = \begin{cases}
        \mathbb{W}^{1,p}_0(\mathcal{O})\cap \mathbb{W}^{1,2}_0(\mathcal{O}), & p\in[2,\infty)\\
        \mathbb{W}^{1,p}_0(\mathcal{O})+ \mathbb{W}^{1,2}_0(\mathcal{O}), & p\in (1,2).
    \end{cases}
\end{align}

\begin{lemma}[{\cite[Proposition 3.1]{Kunstmann_2010}}]
    The Helmholtz projection $\mathcal{P}_{p}$ has a continuous linear extension $\widetilde{\mathcal{P}}_{p}: \widetilde{\mathbb{W}}^{-1,p}(\mathcal{O}) \to \widetilde{\mathbb{W}}^{-1,p}_{\sigma}$ given by the restriction
    \begin{align}
        \widetilde{\mathcal{P}}_{p} \Phi  = \Phi |_{\widetilde{\mathbb{W}}^{1,p'}_{0,\sigma}}, \;\; \text{ for } \; \; \Phi \in \widetilde{\mathbb{W}}^{-1,p}_{\sigma} \;\; \text{ with } \; \;\frac{1}{p}+\frac{1}{p'}=1.
    \end{align}
\end{lemma}

\begin{remark}
We will denote $\widetilde{\mathcal{P}}_{\frac43}$ by $\mathcal{P}$ in the sequel.
\end{remark}

\subsection{Linear and nonlinear operators}\label{sec-operators}
Let us define linear operator by 
\begin{equation*}
	\A\v:=-\mathcal{P}\Delta\v,\ \v\in\X.
\end{equation*}
Remember that the operator $\A$ is a non-negative, self-adjoint operator in $\H$  and \begin{align}\label{2.7a}
	\left<\A\v,\v\right>=\|\v\|_{\V}^2,\ \textrm{ for all }\ \v\in\V, \ \text{ so that }\ \|\A\v\|_{\V^{\prime}}\leq \|\v\|_{\V}.
\end{align}\\

Next, we define the {trilinear form} $b(\cdot,\cdot,\cdot):\V\times\V\times\V\to\R$ by $$b(\u,\v,\w)=\int_{\mathcal{O}}(\u(x)\cdot\nabla)\v(x)\cdot\w(x)\d x=\sum_{i,j=1}^d\int_{\mathcal{O}}\u_i(x)\frac{\partial \v_j(x)}{\partial x_i}\w_j(x)\d x.$$ If $\u, \v$ are such that the linear map $b(\u, \v, \cdot) $ is continuous on $\V$, the corresponding element of $\V^{\prime}$ is denoted by $\B(\u, \v)$. We represent   $\B(\v) = \B(\v, \v)=\mathcal{P}(\v\cdot\nabla)\v$. Using an integration by parts, it is immediate that 
\begin{equation}\label{b0}
	\left\{
	\begin{aligned}
		b(\u,\v,\v) &= 0,\ \text{ for all }\ \u,\v \in\V,\\
		b(\u,\v,\w) &=  -b(\u,\w,\v),\ \text{ for all }\ \u,\v,\w\in \V.
	\end{aligned}
	\right.\end{equation}

The trilinear map $b : \V \times \V \times \V \to \R$ has a unique extension to a bounded trilinear map from
$\L^4(\O) \times \mathbb{W}^{1,4}(\O) \times \H$ to $\R$. For any $\u\in \mathbb{W}^{1,4}(\O)$ and $\v\in\H$, we have	
\begin{align*}
	|b(\u,\u,\v)| \leq \|\u\|_{4}\|\nabla\u\|_{4}\|\v\|_{2},
\end{align*}
which implies
\begin{align*}
	\|\B(\u)\|_2 \leq \|\u\|_{4}\|\nabla\u\|_{4}.
\end{align*}
Therefore, $\B$ maps $\mathbb{W}^{1,4}(\O)$ into $\H$.\\

Let us now define the operator $\J(\v):=-\mathcal{P}[\diver(\Arm(\v)\Arm(\v))]$. Note that for $\v\in\X$, we have 
\begin{align*}
	\|\J(\v)\|_{\X^{\prime}} \leq C \|\Arm(\v)\|^2_{4},
\end{align*}
and hence the map $\J(\cdot):\X\to\X^{\prime}$ is well-defined.\\

Finally, we define the operator $\K(\v):=-\mathcal{P}[\diver(|\Arm(\v)|^{2}\Arm(\v))]$. It is immediate that $$\langle\mathcal{K}(\v),\v\rangle =\frac12\|\Arm(\v)\|_{4}^{4}.$$ Furthermore, the operator $\mathcal{K}(\cdot):\X \to \X^{\prime}$ is well-defined.

\begin{lemma}\label{lem-Loc-Monotinicity}
Assume that $\z\in \mathbb{W}^{1,4}(\mathcal{O})$ and $\varepsilon_0:=1-\sqrt{\frac{\alpha^2}{2\beta\nu}}\in(0,1)$. Then, the operator $\mathscr{G}:\X\to\X^{\prime}$ given by 
	\begin{align}\label{opr-G}
		\mathscr{G}(\cdot):=\nu\A\cdot+\B(\cdot+\z)+\alpha\J(\cdot+\z)+\beta\K(\cdot+\z)
	\end{align}
	 satisfies
	\begin{align}\label{CL1}
		&	\left< \mathscr{G}(\y_1)-\mathscr{G}(\y_2), \y_1-\y_2   \right> + \frac{(C_{S,3} C_K)^2}{\nu\varepsilon_0} \|\Arm(\y_2+\z)\|_{{4}}^2 \|\y_1-\y_2\|_{2}^2 
		\nonumber\\ &  \geq \frac{\nu\varepsilon_0}{4} \|\Arm(\y_1-\y_2)\|^2_{2} +  \frac{\beta\varepsilon_0}{4}\|\Arm(\y_1-\y_2)\|^4_{4} \geq 0,
	\end{align}
	for any  $\y_1,\y_2\in\X.$
\end{lemma}
\begin{proof}
	Firstly, note that $\varepsilon_0 = 1-\sqrt{\frac{\alpha^2}{2\beta\nu}}$ implies $\frac{\alpha^2}{4\nu(1-\varepsilon_0)} = \frac{\beta(1-\varepsilon_0)}{2}$. We know that 
	\begin{align}\label{CL2}
		\left<	\nu\A\y_1 - \nu\A\y_2 , \y_1 -\y_2 \right> = \frac{\nu}{2}\|\Arm(\y_1-\y_2)\|_{2}^2.
	\end{align}
	Now, using \eqref{b0}, H\"older's inequality, Lemma \ref{Sobolev-embedding3}, \eqref{Korn-ineq}  and Young's inequality, we estimate 
	\begin{align*}
		& |\left<	\B(\y_1+\z) - \B(\y_2+\z), \y_1 - \y_2 \right>| 
		\nonumber\\ & \leq  |b(\y_1 - \y_2,\y_1 - \y_2, \y_2+\z)| + |b(\y_2+\z,\y_1 - \y_2, \y_1 - \y_2)|
		\nonumber\\ & \leq 2 \|\y_1 - \y_2\|_{2}\|\nabla(\y_1 - \y_2)\|_{2} \|\y_2+\z\|_{{\infty}}
		\nonumber\\ & \leq C_{S,3} C_K \|\y_1 - \y_2\|_{2}\|\Arm(\y_1 - \y_2)\|_{2} \|\Arm(\y_2+\z)\|_{{4}}
		\nonumber\\ & \leq \frac{\nu\varepsilon_0}{4}\|\Arm(\y_1-\y_2)\|_{2}^2 + \frac{(C_{S,3} C_K)^2}{\nu\varepsilon_0} \|\Arm(\y_2+\z)\|_{{4}}^2 \|\y_1 - \y_2\|_{2}^2
	\end{align*}
	which implies  
	\begin{align}\label{CL3}
		& \left<	\B(\y_1+\z) - \B(\y_2+\z), \y_1 - \y_2 \right>
	\nonumber\\ & 	\geq  - \frac{\nu\varepsilon_0}{4} \|\Arm(\y_1-\y_2)\|^2_{2}  - \frac{(C_{S,3} C_K)^2}{\nu\varepsilon_0} \|\Arm(\y_2+\z)\|_{{4}}^2 \|\y_1 - \y_2\|_{2}^2.
	\end{align}
	From \cite[Equation (2.21)]{Hamza+Paicu_2007}, we have 
	\begin{align*}
		& |\alpha\left<	\J(\y_1+\z) - \J(\y_2+\z), \y_1 - \y_2 \right>| 
		\nonumber\\ & \leq \frac{\nu(1-\varepsilon_0)}{2}\|\Arm(\y_1-\y_2)\|_{2}^2 + \frac{\alpha^2}{4\nu(1-\varepsilon_0)} \int_{\mathcal{O}} |\Arm(\y_1-\y_2)|^2 ( |\Arm(\y_1+\z)|^2+|\Arm(\y_2+\z)|^2)
		\nonumber\\ & \leq \frac{\nu(1-\varepsilon_0)}{2}\|\Arm(\y_1-\y_2)\|_{2}^2 + \frac{\beta(1-\varepsilon_0)}{2}\int_{\mathcal{O}} |\Arm(\y_1-\y_2)|^2( |\Arm(\y_1+\z)|^2+|\Arm(\y_2+\z)|^2),
	\end{align*}
	which implies  
	\begin{align*}
		&\alpha \left<	\J(\y_1+\z) - \J(\y_2+\z), \y_1 - \y_2 \right>
		\nonumber\\ & \geq - \frac{\nu(1-\varepsilon_0)}{2}\|\Arm(\y_1-\y_2)\|_{2}^2 - \frac{\beta(1-\varepsilon_0)}{2}\int_{\mathcal{O}} |\Arm(\y_1-\y_2)|^2( |\Arm(\y_1+\z)|^2+|\Arm(\y_2+\z)|^2).
	\end{align*}
	Now, from \cite[Equation (2.13)]{Hamza+Paicu_2007}, we have 
	\begin{align*}
		&  \beta\left<	\K(\y_1+\z) - \K(\y_2+\z), \y_1 - \y_2 \right>
		\nonumber\\ & = \frac{\beta}{2}\int_{\mathcal{O}}( |\Arm(\y_1+\z)|^2-|\Arm(\y_2+\z)|^2)^2 + \frac{\beta}{2}\int_{\mathcal{O}} |\Arm(\y_1-\y_2)|^2( |\Arm(\y_1+\z)|^2+|\Arm(\y_2+\z)|^2).
	\end{align*}
	In view of \eqref{CL2} and \eqref{CL3}, we obtain \eqref{CL1}. This completes the proof.
\end{proof}

\begin{lemma}\label{lem-demi-G}
Assume that $\z\in \mathbb{W}^{1,4}(\mathcal{O})$. Then, the operator $\mathscr{G}:\X\to \X^{\prime}$ is demicontinuous.
\end{lemma}
\begin{proof}
	Let us take a sequence $\y_m\to\y$ in $\X$, that is, $\|\y_m-\y\|_{\X}=\|\y_m-\y\|_{\V}+\|\y_m-\y\|_{\mathbb{W}^{1,4}}\to 0$ as $m\to\infty$. For any $\boldsymbol{\phi}\in \X$, we consider
	\begin{align}\label{DM1}
		&	\langle\mathscr{G}(\y_m)-\mathscr{G}(\y),\boldsymbol{\phi}\rangle 
		\nonumber\\ &=\nu \langle \A\y_m-\A\y,\boldsymbol{\phi}\rangle+\langle\B(\y_m+\z)-\B(\y+\z),\boldsymbol{\phi}\rangle+ \alpha\langle \J(\y_n+\z)-\J(\y+\z),\boldsymbol{\phi}\rangle 
		\nonumber\\ & \quad +\beta\langle \K(\y_n+\z)-\K(\y+\z),\boldsymbol{\phi}\rangle.
	\end{align} 
	First, we pick $\langle \A\y_m-\A\y,\boldsymbol{\phi}\rangle$ from \eqref{DM1} and estimate as follows:
	\begin{align}
		|\langle \A\y_m-\A\y,\boldsymbol{\phi}\rangle|=|(\nabla(\y_m-\y),\nabla\boldsymbol{\phi})|\leq\|\y_m-\y\|_{\V}\|\boldsymbol{\phi}\|_{\V}\to 0, \ \text{ as } \ m\to\infty, 
	\end{align}
	since $\y_n\to \y$ in $\V$. We estimate the term $\langle\B(\y_m+\z)-\B(\y+\z),\boldsymbol{\phi}\rangle$ from \eqref{DM1} using \eqref{b0}, H\"older's inequality and Lemma \ref{Sobolev-embedding3} as  
	\begin{align}
		&|\langle\B(\y_n+\z)-\B(\y+\z),\boldsymbol{\phi}\rangle|
		\nonumber\\ &  \leq   |b(\y_m-\y,\boldsymbol{\phi}, \y_n-\y) | +|b(\y_m-\y,\boldsymbol{\phi}, \y+\z)| + |b(\y+\z,\boldsymbol{\phi}, \y_m-\y)|
		\nonumber\\ & \leq  \|\y_m-\y\|_{2}\|\y_m-\y\|_{{\infty}}\|\boldsymbol{\phi}\|_{\V} + \|\y_m-\y\|_{2}\|\y+\z\|_{{\infty}}\|\boldsymbol{\phi}\|_{\V} + \|\y_m-\y\|_{2}\|\y+\z\|_{{\infty}}\|\boldsymbol{\phi}\|_{\V}
		\nonumber\\ & \leq  C \|\y_m-\y\|_{2}\|\y_m-\y\|_{\mathbb{W}^{1,4}}\|\boldsymbol{\phi}\|_{\V} + C \|\y_m-\y\|_{2}\|\y+\z\|_{\mathbb{W}^{1,4}}\|\boldsymbol{\phi}\|_{\V} 
		\nonumber\\& \to 0, \ \text{ as } \ m\to\infty, 
	\end{align}
	since $\y_m\to\y$ in $\X$. We estimate the term $\langle \J(\y_m+\z)- \J(\y+\z),\boldsymbol{\phi}\rangle$ from \eqref{DM1} using H\"older's inequality as  
	\begin{align}
		&|\alpha\langle \J(\y_m+\z)- \J(\y+\z),\boldsymbol{\phi}\rangle|
		\nonumber\\ & = |\alpha| |\langle \Arm^2(\y_m+\z)- \Arm^2(\y+\z),\nabla\boldsymbol{\phi}\rangle|
		\nonumber\\ & \leq |\alpha| |\langle \Arm^2(\y_m-\y),\nabla\boldsymbol{\phi}\rangle| + |\alpha| |\langle \Arm(\y_m-\y)\Arm(\y+\z),\nabla\boldsymbol{\phi}\rangle| +  |\alpha| |\langle \Arm(\y+\z)\Arm(\y_m-\y),\nabla\boldsymbol{\phi}\rangle|
		\nonumber\\ & \leq |\alpha|  \|\y_m-\y\|^2_{\mathbb{W}^{1,4}}\|\boldsymbol{\phi}\|_{\V} + 2|\alpha| \|\y_m-\y\|_{\mathbb{W}^{1,4}}\|\y+\z\|_{\mathbb{W}^{1,4}}\|\boldsymbol{\phi}\|_{\V} 
		\nonumber\\& \to 0, \ \text{ as } \ m\to\infty, 
	\end{align}
	since $\y_m\to\y$ in $\mathbb{W}^{1,4}(\mathcal{O})$. We estimate the term $\langle \K(\y_m+\z)- \K(\y+\z),\boldsymbol{\phi}\rangle$ from \eqref{DM1} using H\"older's inequality as  
	\begin{align}
		&|\beta\langle \K(\y_m+\z)- \K(\y+\z),\boldsymbol{\phi}\rangle|
		\nonumber\\ & =  \beta|\langle |\Arm(\y_m+\z)|^2\Arm(\y_m+\z)- |\Arm(\y+\z)|^2\Arm(\y+\z),\nabla\boldsymbol{\phi}\rangle|
		\nonumber\\ & \leq \beta |\langle |\Arm(\y_m-\y)|^2\Arm(\y_m-\y),\nabla\boldsymbol{\phi}\rangle| + \beta |\langle |\Arm(\y_m-\y)|^2\Arm(\y+\z),\nabla\boldsymbol{\phi}\rangle| 
		\nonumber\\ & \quad + 2\beta|\langle [\Arm(\y_m-\y):\Arm(\y+\z)]\Arm(\y_m-\y),\nabla\boldsymbol{\phi}\rangle| + 2 \beta  |\langle [\Arm(\y_m-\y):\Arm(\y+\z)]\Arm(\y+\z),\nabla\boldsymbol{\phi}\rangle| 
		\nonumber\\ & \quad + \beta |\langle |\Arm(\y+\z)|^2\Arm(\y_m-\y),\nabla\boldsymbol{\phi}\rangle|
		\nonumber\\ & \leq  \beta \|\y_m-\y\|^3_{\mathbb{W}^{1,4}}\|\boldsymbol{\phi}\|_{\mathbb{W}^{1,4}} + 3 \beta \|\y_m-\y\|^2_{\mathbb{W}^{1,4}}\|\y+\z\|_{\mathbb{W}^{1,4}}\|\boldsymbol{\phi}\|_{\mathbb{W}^{1,4}}  
		\nonumber\\ & \quad + 3 \beta \|\y_m-\y\|_{\mathbb{W}^{1,4}}\|\y+\z\|^2_{\mathbb{W}^{1,4}}\|\boldsymbol{\phi}\|_{\mathbb{W}^{1,4}}
		\nonumber\\& \to 0, \ \text{ as } \ m\to\infty, 
	\end{align}
	since $\y_m\to\y$ in $\mathbb{W}^{1,4}(\mathcal{O})$.
	
	From the above convergences, it is immediate that $\langle\mathscr{G}(\y_m)-\mathscr{G}(\y),\boldsymbol{\phi}\rangle \to 0$, for all $\boldsymbol{\phi}\in \X$. Hence the operator $\mathscr{G}:\X \to \X'$ is demicontinuous, which implies that, in addition, the operator $\mathscr{G}(\cdot)$ is hemicontinuous also. 
\end{proof}

Finally, we provide the result of locally Lipschitz continuity of the operator $\mathscr{G}(\cdot)$.

\begin{lemma}\label{lem-loc_Lip-G}
	Assume that $\z\in\mathbb{W}^{1,4}(\mathcal{O})$. Then, the operator $\mathscr{G}:\X\to \X^{\prime}$ is locally Lipschitz.
\end{lemma}
\begin{proof}	
	For  $\u,\v,\w\in\X$, we have 
	\begin{align}
		\left<\A\u - \A\v , \w\right> & = (\nabla(\u - \v),\nabla\w)
		\nonumber\\ &  \leq \|\u-\v\|_{\V} \|\w\|_{\V} \\
		\langle \B(\u+\z)-\B(\v+\z),\w\rangle& = b(\u+\z,\u-\v,\w) + b(\u-\v,\v+\z,\w)\nonumber\\
		& \leq \big[\|\u+\z\|_{4} \|\nabla(\u-\v)\|_{4} + \|\u-\v\|_{4} \|\nabla(\v+\z)\|_{4}\big]\|\w\|_{2}
		\nonumber\\
		& \leq C \big[\|\u+\z\|_{4} \|\Arm(\u-\v)\|_{4} + \|\u-\v\|_{4} \|\Arm(\v+\z)\|_{4}\big]\|\w\|_{\V}, \label{213-B}\\
		\langle \J(\u+\z)-\J(\v+\z),\w\rangle& = \frac12 \int_{\mathcal{O}}[\Arm(\u-\v)\Arm(\u+\z)+\Arm(\v+\z)\Arm(\u-\v)]:\Arm(\w)\d x \nonumber
		\\&\leq C\big[\|\Arm(\u+\z)\|_{4}+\|\Arm(\v+\z)\|_{4}\big] \|\Arm(\u-\v)\|_{4} \|\Arm(\w)\|_{2},\label{213-J}\\ 
		\langle \K(\u)-\K(\v),\w\rangle& = \frac12 \int_{\mathcal{O}}[\Arm(\u-\v):\Arm(\u+\z)+\Arm(\v+\z):\Arm(\u-\v)]\Arm(\u+\z):\Arm(\w)\d x 
		\nonumber\\ & \quad + \frac12 \int_{\mathcal{O}}|\Arm(\v+\z)|^2\Arm(\u-\v):\Arm(\w)\d x \nonumber
		\\&\leq C\big[\|\Arm(\u+\z)\|_{4}^2+\|\Arm(\v+\z)\|_{4}^2\big] \|\Arm(\u-\v)\|_{4} \|\Arm(\w)\|_{4}.\label{213-K}
	\end{align}
 Hence, the operator $\mathscr{G}(\cdot):\X\to\X^{\prime}$ is locally Lipschitz. 
\end{proof}

\subsection{A compact operator}\label{C_O} We follow 
\cite[Subsection 2.3]{Brzezniak+Motyl_2013} for the contents of this section. 
Consider the natural embedding $j:\V\hookrightarrow\H$ and its adjoint $j^*:\H\hookrightarrow\V$. Since the range of $j$ is dense in $\H$, the map $j^*$ is one-to-one. Let us define
\begin{align}\label{L1}
	\D(\mathbb{A})&:=j^*(\H)\subset\V,\nonumber \\
	\mathbb{A}\v&:=(j^*)^{-1}\v, \ \ \v\in\D(\mathbb{A}).
\end{align}
Note that for all $\v\in\D(\mathbb{A})$ and $\v\in\V$
\begin{align*}
	(\mathbb{A}\u,\v)_{\H}=(\u,\v)_{\V}.
\end{align*}
Let us assume that $s>2$. It is clear that $\V_s$ is dense in $\V$ and the embedding $j_s:\V_s\hookrightarrow\V$ is continuous. Then, there exists a Hilbert space $\mathbb{U}$ (cf. \cite{Holly+Wiciak_1995}, \cite[Lemma C.1]{Brzezniak+Motyl_2013}) such that $\mathbb{U}\subset\V_s$, $\mathbb{U}$ is dense in $\V_s$ and 
\begin{align*}
	\text{  the natural embedding }\iota_s:\mathbb{U}\hookrightarrow\V_s \text{ is compact.}
\end{align*}
It implies that 
\begin{align*}
	\mathbb{U} \xhookrightarrow[\iota_s]{} \V_s\xhookrightarrow[j_s]{}\V\xhookrightarrow[j]{}\H\cong{\H}{'}\xhookrightarrow[j']{}\V'\xhookrightarrow[j'_s]{}  {\V}'_{s}\xhookrightarrow[\iota'_s]{}\mathbb{U}'.
\end{align*}
Consider the composition $$\iota:=j\circ j_s\circ\iota_s:\mathbb{U}\to\H$$ and its adjoint $$\iota^*:=(j\circ j_s\circ\iota_s)^*=\iota_s^*\circ j^*_s\circ j^*:\H\to \mathbb{U}.$$ We have that $\iota$ is compact and since its range is dense in $\H$, $\iota^*:\H\to\mathbb{U}$ is one-one. Let us define 
\begin{align}\label{L2}
	\D(\mathcal{L})&:=\iota^*(\H)\subset\mathbb{U},\nonumber\\
	\mathcal{L}\v&:=(\iota^*)^{-1}\v, \ \ \v\in\D(\mathcal{L}).
\end{align}
Also we have that $\mathcal{L}:\D(\mathcal{L})\to\H$ is onto, $\D(\mathcal{L})$ is dense in $\H$ and 
\begin{align*}
	(\mathcal{L}\v,\w)_{\H}=(\v,\w)_{\mathbb{U}}, \ \ \ \ \v\in\D(\mathcal{L}), \ \w\in\mathbb{U}.
\end{align*}
Furthermore, for $\v\in\D(\mathcal{L})$,
\begin{align*}
	\mathcal{L}\v=(\iota^*)^{-1}\v=(j^*)^{-1}\circ (j^*_s)^{-1}\circ(\iota^*_s)^{-1}\v=\mathbb{A}\circ (j^*_s)^{-1}\circ(\iota^*_s)^{-1}\v,
\end{align*}
where $\mathbb{A}$ is defined in \eqref{L1}. Since the operator $\mathcal{L}$ is self-adjoint and $\mathcal{L}^{-1}$ is compact, there exists an orthonormal basis $\{\boldsymbol{e}_i\}_{i\in\N}$ of $\H$ such that 
\begin{align}\label{L3}
	\mathcal{L}\boldsymbol{e}_i=\mu_i\boldsymbol{e}_i,\ \ \ \ i\in\N,
\end{align}
that is, $\{\boldsymbol{e}_i\}$ are the eigenfunctions and $\{\mu_i\}$ are the corresponding eigenvalues of operator $\mathcal{L}$. Note that $\boldsymbol{e}_i\in\mathbb{U}$, \ $i\in\N$, because $\D(\mathcal{L})\subset \mathbb{U}$. 

Let us fix $m\in\N$ and let $\P_m$ be the operator from $\mathbb{U}'$ to $\mathrm{span}\{\boldsymbol{e}_1,\ldots,\boldsymbol{e}_m\}=:\H_{m}$ defined by 
\begin{align}
	\P_m\v^*:=\sum_{i=1}^{m}\langle\v^*,\boldsymbol{e}_i\rangle_{\mathbb{U}'\times\mathbb{U}}\boldsymbol{e}_i, \ \ \ \ \ \ \v^*\in\mathbb{U}'.
\end{align}
We will consider the restriction of operator $\P_m$ to the space $\H$ denoted still by the  same. In particular, we have $\H\hookrightarrow\mathbb{U}'$, that is, every element $\v\in\H$ induces a functional $\v^*\in\mathbb{U}'$ by 
\begin{align}
	\langle\v^*,\u\rangle_{\mathbb{U}'\times\mathbb{U}}:=(\v,\u), \ \ \ \ \u\in\mathbb{U}.
\end{align}
Thus the restriction of $\P_m$ to $\H$ is given by 
\begin{align}
	\P_m\v:=\sum_{i=1}^{m}(\v,\boldsymbol{e}_i)\boldsymbol{e}_i, \ \ \ \ \ \ \v\in\H.
\end{align}
Hence, in particular, $\P_m$ is the orthogonal projection from $\H$ onto $\H_m$. 
\begin{lemma}[{See \cite[Lemma 2.4]{Brzezniak+Motyl_2013}}]\label{lem-Pm}
	For every $\v\in\mathbb{U}$ and $s>2$, we have 
	\begin{itemize}
		\item [(i)] $ \|\P_m\v\|_{\U} \leq \|\v\|_{\U}$,
 		\item [(ii)] $\lim\limits_{m\to\infty}\|\P_m\v-\v\|_{\mathbb{U}}=0$,
		\item [(iii)] $\lim\limits_{m\to\infty}\|\P_m\v-\v\|_{\mathbb{V}_s}=0$, 
		\item [(iv)] $\lim\limits_{m\to\infty}\|\P_m\v-\v\|_{\mathbb{V}}=0$,
		\item [(v)] $\lim\limits_{m\to\infty}\|\P_m\v-\v\|_{\mathbb{W}^{1,4}}=0$.
	\end{itemize}
\end{lemma}

	\subsection{A subclass of stochastic third-grade fluids equations}
Here, we provide an abstract formulation of   system \eqref{equationV1}.
On taking projection $\mathcal{P}$ onto the first equation in \eqref{equationV1}, we obtain 
\begin{equation}\label{STGF}
	\left\{
	\begin{aligned}
		\d\v(t)+\{\nu\A\v(t)+\B(\v(t))+\alpha\J(\v(t))+\beta\K(\v(t))\}\d t&=\mathcal{P}\f \d t + \d\mathrm{W}(t), \ \ \ t> 0, \\ 
		\v(0)&=\boldsymbol{x},
	\end{aligned}
	\right.
\end{equation}
where $\boldsymbol{x}\in \H,\ \f\in \H^{-1}(\mathcal{O})$ and $\{\mathrm{W}(t)\}_{t\in \R}$ is a two-sided cylindrical Wiener process in $\H$ with its RKHS $\mathrm{K}$ which satisfies Hypothesis \ref{assump1}.
Now we present the definition of weak solution (in the analytic sense) to system \eqref{STGF} with the initial data $\x\in\H$ at the initial time $s\in \R.$
\begin{definition}\label{Def_u}
	Suppose that \eqref{condition1} and Hypotheses \ref{assump1}-\ref{assumpO} are satisfied. If $\x\in \H$, $s\in \R$, $\f\in \H^{-1}(\O)$ and $\{\W(t)\}_{t\in \R}$ is a two-sided Wiener process with its RKHS $\mathrm{K}$. A process $\{\v(t), \ t\geq s\},$ with trajectories in $\mathrm{C}([s, \infty); \H) \cap \mathrm{L}^{4}_{\mathrm{loc}}([s, \infty); \mathbb{W}^{1,4}(\mathcal{O}))$ is a solution to system \eqref{STGF} if and only if  $\v(s) = \x$ and for any $\boldsymbol{\phi}\in \X$, $t>s,$ $\mathbb{P}$-a.s.,
	\begin{align*}
	&	(\v(t), \boldsymbol{\phi}) - (\v(s), \boldsymbol{\phi}) 
	\nonumber\\  & =  - \int_{s}^{t} \langle \nu\A\v(\xi)+\B(\v(\xi))+\alpha\J(\v(\xi))+\beta\K(\v(\xi))  - \f , \boldsymbol{\phi} \rangle \d \xi  +  \int_{s}^{t} ( \boldsymbol{\phi}, \d\W(\xi)).
	\end{align*}
\end{definition}

The well-posedness of   system \eqref{STGF} is discussed in Theorem \ref{STGF-Sol} below.

	\section{RDS generated by system \eqref{equationV1}}\label{sec3}\setcounter{equation}{0}
In this section, we prove the existence of RDS generated by  system \eqref{equationV1} using a random transformation (known as Doss-Sussmann transformation, \cite{Doss_1977,Sussmann_1978}). The process of generation of RDS has been adopted from the works \cite{Brzezniak+Li_2006,BCLLLR}. 
\subsection{Metric dynamical system}Let us denote $\mathfrak{X} = \H \cap  {\mathbb{W}}^{1,4}(\mathcal{O}) $. Let $\mathrm{E}$ denote the completion of $\A^{-\delta}(\mathfrak{X})$ with respect to the graph norm $\|x\|_{\mathrm{E}}=\|\A^{-\delta} x \|_{\mathfrak{X}}, \text{ for } x\in \mathfrak{X}$,  where $ \|\cdot\|_{\mathfrak{X}} = \|\cdot\|_{2} +  \|\cdot\|_{\mathbb{W}^{1,4}}$. Note that $\mathrm{E}$ is a separable Banach spaces (cf. \cite{Brze2}).

For $\xi \in(0, 1/2)$, let us set 
$$ \|\omega\|_{C^{\xi}_{1/2} (\mathbb{R};\mathrm{E})} = \sup_{t\neq s \in \mathbb{R}} \frac{\|\omega(t) - \omega(s)\|_{\mathrm{E}}}{|t-s|^{\xi}(1+|t|+|s|)^{1/2}}.$$
Furthermore, we define
\begin{align*}
	C^{\xi}_{1/2} (\mathbb{R}; \mathrm{E}) &= \left\{ \omega \in C(\mathbb{R}; \mathrm{E}) : \omega(0)=\boldsymbol{0},\ \|\omega\|_{C^{\xi}_{1/2} (\mathbb{R}; \mathrm{E})} < \infty \right\},\\ \Omega(\xi, \mathrm{E})&= \overline{\{ \omega \in C^\infty_0 (\mathbb{R}; \mathrm{E}) : \omega(0) = 0 \}}^{C^{\xi}_{1/2} (\mathbb{R}; \mathrm{E})}.
\end{align*}
The space $\Omega(\xi, \mathrm{E})$ is a separable Banach space. We also define
$$C_{1/2} (\mathbb{R}; \mathrm{E}) = \left\{ \omega \in C(\mathbb{R}; \mathrm{E}) : \omega(0)=0, \|\omega\|_{C_{1/2} (\mathbb{R}; \mathrm{E})} = \sup_{t \in \mathbb{R}} \frac{\|\omega(t) \|_{\mathrm{E}}}{1+|t|^{\frac{1}{2}}} < \infty \right\}.$$

Let us denote by $\mathcal{F}$, the Borel $\sigma$-algebra on $\Omega(\xi, \mathrm{E}).$ For $\xi\in (0, 1/2)$, there exists a Borel probability measure $\mathbb{P}$ on $\Omega(\xi, \mathrm{E})$ (cf. \cite{Brze}) such that the canonical process $\{w_t, \ t\in \mathbb{R}\}$ defined by 
\begin{align}\label{Wp}
	w_t(\omega) := \omega(t), \ \ \ \omega \in \Omega(\xi, \mathrm{E}),
\end{align}
is an $\mathrm{E}$-valued two-sided Wiener process such that the RKHS of the Gaussian measure $\mathscr{L}(w_1)$ on $\mathrm{E}$ is $\mathrm{K}$. For $t\in \mathbb{R},$ let $\mathcal{F}_t := \sigma \{ w_s : s \leq t \}.$ Then  there exists a unique bounded linear map $\mathrm{W}(t): \mathrm{K} \to \mathrm{L}^2(\Omega(\xi, \mathrm{E}), \mathcal{F}_t  ,  \mathbb{P}).$ Moreover, the family $\{\mathrm{W}(t)\}_{t\in \mathbb{R}}$ is a $\mathrm{K}$-cylindrical Wiener process on a filtered probability space $(\Omega(\xi, \mathrm{E}), \mathcal{F}, \{\mathcal{F}_t\}_{t \in \mathbb{R}} , \mathbb{P})$ (cf. \cite{BP} for more details).

We consider a flow $\theta = (\theta_t)_{t\in \mathbb{R}}$ on the space $C_{1/2} (\mathbb{R}; \mathrm{E}),$  defined by
$$ \theta_t \omega(\cdot) = \omega(\cdot + t) - \omega(t), \ \ \ \omega\in C_{1/2} (\mathbb{R};\mathrm{E}), \ \ t\in \mathbb{R}.$$ 
This flow keeps the spaces $C^{\xi}_{1/2} (\mathbb{R};\mathrm{E})$ and $\Omega(\xi, \mathrm{E})$ invariant and preserves $\mathbb{P}.$ 

Summing up, we have the following result:
\begin{proposition}[{\cite[Proposition 6.13]{Brzezniak+Li_2006}}]\label{m-DS1}
	The quadruple $(\Omega(\xi, \mathrm{E}), \mathcal{F}, \mathbb{P}, {\theta})$ is an MDS. 
\end{proposition}

\subsection{Ornstein-Uhlenbeck process}\label{O-Up}
Let us first recall some analytic preliminaries from \cite{Brzezniak+Li_2006} which will help us to define an Ornstein-Uhlenbeck process. All the results of this subsection are valid for the space $\mathrm{C}^{\xi}_{1/2} (\mathbb{R}; \mathbb{Y})$ replaced by $\Omega(\xi, \mathbb{Y}).$ 
\begin{proposition}[{\cite[Proposition 2.11]{Brzezniak+Li_2006}}]\label{Ap}
	Let  $-\mathbb{A}$ be the generator of an analytic semigroup $\{e^{t\mathbb{A}}\}_{t\geq 0}$ on a separable Banach space $\mathbb{Y}$ such that for some $C>0\ \text{and}\ \gamma>0$
	\begin{align}\label{ASG}
		\| \mathbb{A}^{1+\delta}e^{-t\mathbb{A}}\|_{\mathfrak{L}(\mathbb{Y})} \leq C_{\delta} t^{-1-\delta} e^{-\gamma t}, \ \ t> 0,
	\end{align}
	where $\mathfrak{L}(\mathbb{Y})$ denotes the space of all bounded linear operators from $\mathbb{Y}$ to $\mathbb{Y}$.	For $\xi \in (\delta, 1/2)$ and $\widetilde{\omega} \in  \mathrm{C}^{\xi}_{1/2} (\mathbb{R};\mathbb{Y}),$  define 
	\begin{align}
		\hat{\z}(t) = \hat{\z} (\mathbb{A}; \widetilde{\omega})(t) := \int_{-\infty}^{t} \mathbb{A}^{1+\delta} e^{-(t-r)\mathbb{A}} (\widetilde{\omega}(t) - \widetilde{\omega}(r))\d r, \ \ t\in \mathbb{R}.
	\end{align}
	If $t\in \mathbb{R},$ then $\hat{\z}(t)$ is a well-defined element of $\mathbb{Y}$ and the mapping 
	$$\mathrm{C}^{\xi}_{1/2} (\mathbb{R};\mathbb{Y}) \ni \widetilde{\omega}  \mapsto \hat{\z}(t) \in \mathbb{Y} $$
	is continuous. Moreover, the map $\hat{\z} :  \mathrm{C}^{\xi}_{1/2} (\mathbb{R}; \mathbb{Y}) \to  \mathrm{C}_{1/2} (\mathbb{R}; \mathbb{Y})$  is well defined, linear and bounded. In particular, there exists a constant $C >0$ such that for any $\widetilde{\omega} \in \mathrm{C}^{\xi}_{1/2} (\mathbb{R};\mathbb{Y})$ 
	\begin{align}\label{X_bound_of_z}
		\|\hat{\z}(\widetilde{\omega})(t)\|_{\mathbb{Y}} \leq C(1 + |t|^{1/2})\|\widetilde{\omega}\|_{\mathrm{C}^{\xi}_{1/2} (\mathbb{R}; \mathbb{Y})}, \ \ \ t \in \R.
	\end{align}
	Furthermore, under the same assumption, the following results hold (Corollaries 6.4, 6.6 and 6.8 in \cite{Brzezniak+Li_2006}):
	\begin{itemize}
		\item [1.]For all $-\infty<a<b<\infty$ and $t\in \R$, the map 
		\begin{align}\label{O-U_conti}
			\mathrm{C}^{\xi}_{1/2} (\mathbb{R};\mathbb{Y}) \ni \widetilde{\omega} \mapsto (\hat{\z}(\widetilde{\omega})(t), \hat{\z}(\widetilde{\omega})) \in \mathbb{Y} \times \mathrm{L}^{q} (a, b; \mathbb{Y}),
		\end{align}
		where $q\in [1, \infty]$, is continuous.
		\item [2.] For any $\omega \in \mathrm{C}^{\xi}_{1/2} (\mathbb{R};\mathbb{Y}),$
		\begin{align}\label{stationary}
			\hat{\z}(\theta_s \omega)(t) = \hat{\z}(\omega)(t+s), \ \ t, s \in \mathbb{R}.
		\end{align}
		\item [3.] For $\zeta \in \mathrm{C}_{1/2}(\mathbb{R};\mathbb{Y}),$ if we put $\uptau_s(\zeta(t))=\zeta(t+s), \ t,s \in \R,$ then, for $t \in \R ,\  \uptau_s \circ \hat{\z} = \hat{\z}\circ\theta_s$, that is, 
		\begin{align}\label{IS}
			\uptau_s\big(\hat{\z}(\omega)\big)= \hat{\z}\big(\theta_s(\omega)\big), \ \ \ \omega\in \mathrm{C}^{\xi}_{1/2} (\mathbb{R};\mathbb{Y}).
		\end{align}
	\end{itemize} 
\end{proposition}

Next, we define the Ornstein-Uhlenbeck process under Hypothesis \ref{assump1}. For $\delta$ as in Hypothesis \ref{assump1}, $\nu> 0, \ \chi \geq 0, \ \xi \in (\delta, 1/2)$ and $ \omega \in C^{\xi}_{1/2} (\mathbb{R};\mathrm{E})$ (so that $(\nu \A + \chi\I)^{-\delta}\omega \in C^{\xi}_{1/2} (\mathbb{R};\mathfrak{X})$), we define $$ \boldsymbol{z}_{\chi}(\omega) := \hat{\z}((\nu \A + \chi\I); (\nu \A + \chi\I)^{-\delta}\omega) \ \in C_{1/2}(\mathbb{R};\mathfrak{X}),$$  that is, for any $t\geq 0,$ 
\begin{align}\label{DOu1}
	\boldsymbol{z}_{\chi}(\omega)(t)&=\int_{-\infty}^{t} (\nu \A + \chi\I)^{1+\delta} e^{-(t-\tau)(\nu \A + \chi\I)} ((\nu \A + \chi\I)^{-\delta}\theta_{\tau} \omega)(t-\tau)\d \tau.
\end{align}
For $\omega \in C^{\infty}_0 (\mathbb{R};\mathrm{E})$ with $\omega(0)= \boldsymbol{0},$ using  integration by parts, we obtain 
\begin{align*}
	\frac{\d\boldsymbol{z}_\chi(t)}{\d t} &= -(\nu \A + \chi\I )\int_{-\infty}^{t} (\nu \A + \chi\I)^{1+\delta} e^{-(t-r)(\nu \A + \chi\I)} [(\nu \A + \chi\I)^{-\delta}\omega(t) \\&\qquad\qquad - (\nu \A + \chi\I)^{-\delta}\omega(r)]\d r +  \frac{\d\omega(t)}{\d t}.
\end{align*}
Thus $\boldsymbol{z}_{\chi}(\cdot)$ is the solution of the following equation:
\begin{align}\label{OuE1}
	\frac{\d\boldsymbol{z}_{\chi} (t)}{\d t} + (\nu \A + \chi\I)\boldsymbol{z}_{\chi}(t) = \frac{\d\omega(t)}{\d t}, \ \ t\in \mathbb{R}.
\end{align}
Therefore, from the definition of the space $\Omega(\xi, \mathrm{E}),$ we have 
\begin{corollary}\label{Diff_z1}
	If $\chi_1, \chi_2 \geq 0,$ then the difference $\boldsymbol{z}_{\chi_1} - \boldsymbol{z}_{\chi_2}$ is a solution to 
	\begin{align}\label{Dif_z1}
		\frac{\d(\boldsymbol{z}_{\chi_1} - \boldsymbol{z}_{\chi_2})(t)}{\d t} + \nu\A(\boldsymbol{z}_{\chi_1} - \boldsymbol{z}_{\chi_2})(t) = -(\chi_1 \boldsymbol{z}_{\chi_1} - \chi_2\boldsymbol{z}_{\chi_2})(t), \ \ \ t \in \R.
	\end{align}
\end{corollary}

According to the  definition \eqref{Wp} of Wiener process $\{w_t, \ t\in \R\},$ one can view the formula \eqref{DOu1} as a definition of a process $\{\boldsymbol{z}_{\chi}(t), \ t\in \R\}$ on the probability space $(\Omega(\xi, \mathrm{E}), \mathcal{F}, \mathbb{P})$. Equation \eqref{OuE1} clearly tells that the process $\boldsymbol{z}_{\chi}(\cdot)$ is an Ornstein-Uhlenbeck process. Furthermore, the following results hold for $\boldsymbol{z}_{\chi}(\cdot)$.
\begin{proposition}[{\cite[Proposition 6.10]{Brzezniak+Li_2006}}]\label{SOUP1}
	The process $\{\boldsymbol{z}_{\chi}(t), \ t\in \mathbb{R}\},$ is a stationary Ornstein-Uhlenbeck process on $(\Omega(\xi, \mathrm{E}), \mathcal{F}, \mathbb{P})$. It is a solution of the equation 
	\begin{align}\label{OUPe1}
		\d\boldsymbol{z}_{\chi}(t) + (\nu \A + \chi \I)\boldsymbol{z}_{\chi}(t) \d t = \d\mathrm{W}(t), \ \ t\in \mathbb{R},
	\end{align}
	that is, for all $t\in \mathbb{R},$ $\mathbb{P}$-a.s.,
	\begin{align}\label{oup1}
		\boldsymbol{z}_\chi (t) = \int_{-\infty}^{t} e^{-(t-\xi)(\nu \A + \chi\I)} \d\mathrm{W}(\xi),
	\end{align}
	where the integral is an It\^o integral on the M-type 2 Banach space $\mathfrak{X}$  (cf. \cite{Brze1}). 	In particular, for some $C$ depending on $\mathfrak{X}$,
	\begin{align}\label{E-OUP1}
		\mathbb{E}\left[\|\boldsymbol{z}_{\chi} (t)\|^2_{\mathfrak{X}} \right]&= \mathbb{E}\left[\left\|\int_{-\infty}^{t} e^{-(t-\xi)(\nu \A + \chi\I)} \d\mathrm{W}(\xi)\right\|^2_{\mathfrak{X}}\right] \leq C\int_{-\infty}^{t} \|e^{-(t-\xi)(\nu \A +  \chi\I)}\|^2_{\gamma(\mathrm{K},\mathfrak{X})} \d \xi \nonumber\\&=C \int_{0}^{\infty}  e^{-2\chi \xi} \|e^{-\nu \xi \A}\|^2_{\gamma(\mathrm{K},\mathfrak{X})} \d \xi.
	\end{align} 
	Moreover, $\mathbb{E}\left[\|\boldsymbol{z}_{\chi} (t)\|^2_{\mathfrak{X}}\right]\to 0$ as $\chi \to \infty.$
\end{proposition}

\begin{remark}
	Note that \eqref{OUPe1} is the projected form of an  equation of the following type: 
	\begin{equation}\label{eqn_z_alpha}
		\left\{
		\begin{aligned}
			\d\z_\chi (t)  + \left[(-\nu\Delta +\chi \I) \z_\chi (t) + \nabla  {p} \right]\d t &=\d \mathrm{W}(t) ,  \ \ t\in \mathbb{R},\\
			\mathrm{div}\;\z_\chi &=0,
		\end{aligned}
		\right.
	\end{equation}
	where $p$ is a scalar field associated with projected equation \eqref{OUPe1}.
\end{remark}

Let us now provide some consequences of the previous discussion which will be used in the sequel. 
\begin{lemma}\label{Bddns4}
	For each $\omega\in \Omega$ and $c>0$, we obtain 
	\begin{align*}
		\lim_{t\to - \infty} \|\z_{\chi}(\omega)(t)\|^2_{2}\  e^{c t} = 0.
	\end{align*}
\end{lemma}
\begin{proof}
	Let us fix $\omega\in \Omega$. Because of \eqref{X_bound_of_z}, there exist  $\rho_1= \rho_1(\omega)\geq 0$  and $t_0 \leq 0$ such that
	\begin{align}\label{rho}
		\frac{\|\z_{\chi}(t)\|_{2}}{|t|} \leq \rho_1, \  \frac{\|\z_{\chi}(t)\|_{4}}{|t|} \leq \rho_1 \ \text{ and } \   \frac{\|\z_{\chi}(t)\|_{\mathbb{W}^{1,4}}}{|t|} \leq \rho_1  \ \text{ for }\  t\leq t_0.
	\end{align}
	Therefore, we have, for every $\omega\in \Omega,$
	\begin{align*}
		\lim_{t\to - \infty} \|\z_{\chi}(\omega)(t)\|^2_{2}\  e^{c t}\leq&  \rho_1^2 \lim_{t\to - \infty}  |t|^2 e^{c t} =0,
	\end{align*}
	which completes the proof.
\end{proof}
\begin{lemma}\label{Bddns5}
	For each $\omega\in \Omega$ and $c>0$, we get 
	\begin{align*}
		\int_{- \infty}^{0} \bigg\{ 1 + \|\z_{\chi}(t)\|^2_{2} +  \|\z_{\chi}(t)\|^4_{\mathbb{W}^{1,4}} \bigg\}e^{c t} \d t < \infty.
	\end{align*}
\end{lemma}
\begin{proof}
	Note that for $t_0\leq 0$,
	\begin{align*}
		\int_{t_0}^{0} \bigg\{ 1 + \|\z_{\chi}(t)\|^2_{2} +  \|\z_{\chi}(t)\|^4_{\mathbb{W}^{1,4}} \bigg\}e^{c t } \d t < \infty.
	\end{align*}
	Therefore, we only need to show that the integral 
	\begin{align*}
		\int_{- \infty}^{t_0} \bigg\{ 1+ \|\z_{\chi}(t)\|^2_{2} +  \|\z_{\chi}(t)\|^4_{\mathbb{W}^{1,4}}   \bigg\}e^{c t } \d t < \infty.
	\end{align*}
	In view of  \eqref{rho}, we obtain 
	\begin{align*} 
		&	\int_{- \infty}^{t_0} \bigg\{ \|\z_{\chi}(t)\|^2_{2} +  \|\z_{\chi}(t)\|^4_{\mathbb{W}^{1,4}} \bigg\}e^{c t } \d t  \leq \int_{- \infty}^{t_0} \big\{ \rho^2_1|t|^2+ \rho^4_1|t|^4 \big\}e^{c t } \d t< \infty,
	\end{align*}
	which  completes the proof.
\end{proof}

\begin{definition}\label{RA2}
	A function $\kappa: \Omega\to (0, \infty)$ belongs to the class $\mathfrak{K}$ if and only if 
	\begin{align}
		\lim_{t\to \infty} [\kappa(\theta_{-t}\omega)]^2 e^{-c t } = 0, 
	\end{align}
for all $c>0$.
\end{definition}
Let us denote the class of all closed and bounded random sets $D$ on $\H$ by $\mathfrak{DK}$ such that the radius function $\Omega\ni \omega \mapsto \kappa(D(\omega)):= \sup\{\|x\|_{2}:x\in D(\omega)\}$ belongs to the class $\mathfrak{K}.$ It is straight forward that the constant functions belong to $\mathfrak{K}$. 
It is clear by Definition \ref{RA2} that the class $\mathfrak{K}$ is closed with respect to sum, multiplication by a constant and if $\kappa \in \mathfrak{K}, 0\leq \bar{\kappa} \leq \kappa,$ then $\bar{\kappa}\in \mathfrak{K}.$	
\begin{proposition}\label{radius}
	We define functions $\kappa_{i}:\Omega\to (0, \infty), i= 1, \ldots, 4,$ by the following formulae: for $\omega\in\Omega,$
	\begin{align*}
		[\kappa_1(\omega)]^2 &:= \|\z_{\chi}(\omega)(0)\|_{2},\ \ \ &&
		[\kappa_2(\omega)]^2 := \sup_{s\leq 0} \|\z_{\chi}(\omega)(s)\|^2_{2}\  e^{\nu\lambda\left(1+ \frac{\varepsilon_0}{2}\right) s }, \\
		[\kappa_3(\omega)]^2 &:= \int_{- \infty}^{0}  e^{\nu\lambda\left(1+ \frac{\varepsilon_0}{2}\right) t } \|\z_{\chi}(\omega)(t)\|^2_{2}\  \d t, \ \ \ &&
		[\kappa_4(\omega)]^2 := \int_{- \infty}^{0} e^{\nu\lambda\left(1+ \frac{\varepsilon_0}{2}\right) t } \|\z_{\chi}(\omega)(t)\|^4_{\mathbb{W}^{1,4}}\  \d t.
	\end{align*}
	Then all these functions belongs to class $\mathfrak{K}.$
	\begin{proof}Let us recall from \eqref{stationary} that $\z_{\chi}(\theta_{-t}\omega)(s) = \z_{\chi}(\omega)(s-t)$.
		For any given $c>0$, we consider
		\begin{align*}
			\lim_{t\to  \infty}[\kappa_1(\theta_{-t}\omega)]^2 e^{-c t } =& \lim_{t\to \infty}\|\z_{\chi}(\theta_{-t}\omega)(0)\|^2_{2} e^{-c t }
			 	= \lim_{t\to \infty}\|\z_{\chi}(\omega)(-t)\|^2_{2} e^{-c t }.
		\end{align*}
		Using Lemma \ref{Bddns4}, we have, $\kappa_1 \in \mathfrak{K}.$ 
		
		Next, for any given $c>0$, we define $c_1:=\min\{\frac{c}{2},\nu\lambda\left(1+ \frac{\varepsilon_0}{2}\right)\}$ and consider 
		\begin{align*}
			[\kappa_2(\theta_{-t}\omega)]^2 e^{-ct}
			&= e^{-ct} \sup_{s\leq 0}  \|\z_{\chi}(\omega)(s-t)\|^2_{2}\  e^{\nu\lambda\left(1+ \frac{\varepsilon_0}{2}\right) s }
		\nonumber\\ &	= e^{-ct} \sup_{\sigma\leq -t}  \|\z_{\chi}(\omega)(\sigma)\|^2_{2}\  e^{\nu\lambda\left(1+ \frac{\varepsilon_0}{2}\right) (\sigma+t)}
		\nonumber\\ &	\leq e^{-ct} \sup_{\sigma\leq -t}  \|\z_{\chi}(\omega)(\sigma)\|^2_{2}\  e^{c_1 (\sigma+t)}
		\nonumber\\ &	= e^{-(c-c_1)t} \sup_{\sigma\leq -t}  \|\z_{\chi}(\omega)(\sigma)\|^2_{2}\  e^{c_1 \sigma},
		\end{align*}
		 which implies
		\begin{align*}
			&	\lim_{t\to \infty} [\kappa_2(\theta_{-t}\omega)]^2 e^{-c t } \leq \lim_{t\to \infty}   e^{-(c-c_1)t} \lim_{t\to \infty} \left[\sup_{\sigma\leq -t}  \|\z_{\chi}(\omega)(\sigma)\|^2_{2}\  e^{c_1 \sigma}\right]
			= 0,
		\end{align*}
		where we have used Lemma \ref{Bddns4}. This implies that $\kappa_2\in \mathfrak{K}.$ 
		
		Finally, from Lemma \ref{Bddns5} and absolute continuity of Lebesgue integrals, we obtain 
		\begin{align*}
			\bigg\{[\kappa_3(\theta_{-t}\omega)]^2+ [\kappa_4(\theta_{-t}\omega)]^2\bigg\} e^{-c t}
			& = e^{-c t} \int_{- \infty}^{0} \bigg\{  \|\z_{\chi}(\theta_{-t}\omega)(\sigma)\|^2_{2} + \|\z_{\chi}(\theta_{-t}\omega)(\sigma)\|^2_{\mathbb{W}^{1,4}}  \bigg\}e^{\nu\lambda\left(1+ \frac{\varepsilon_0}{2}\right) \sigma } \d \sigma\\
			& = e^{-c t} \int_{- \infty}^{0} \bigg\{  \|\z_{\chi}(\omega)(\sigma-t)\|^2_{2} + \|\z_{\chi}(\omega)(\sigma-t)\|^2_{\mathbb{W}^{1,4}}  \bigg\}e^{\nu\lambda\left(1+ \frac{\varepsilon_0}{2}\right) \sigma } \d \sigma\\
			& = e^{-c t} \int_{- \infty}^{-t} \bigg\{  \|\z_{\chi}(\omega)(\sigma)\|^2_{2} + \|\z_{\chi}(\omega)(\sigma)\|^2_{\mathbb{W}^{1,4}}  \bigg\}e^{\nu\lambda\left(1+ \frac{\varepsilon_0}{2}\right) (\sigma+t) } \d \sigma
			\\
			& \leq  e^{-c t} \int_{- \infty}^{-t} \bigg\{  \|\z_{\chi}(\omega)(\sigma)\|^2_{2} + \|\z_{\chi}(\omega)(\sigma)\|^2_{\mathbb{W}^{1,4}}  \bigg\}e^{c_1 (\sigma+t) } \d \sigma
			\\
			& \leq  e^{-(c-c_1) t} \int_{- \infty}^{-t} \bigg\{  \|\z_{\chi}(\omega)(\sigma)\|^2_{2} + \|\z_{\chi}(\omega)(\sigma)\|^2_{\mathbb{W}^{1,4}}  \bigg\}e^{c_1 \sigma } \d \sigma
			\\ & \to 0 \ \text{	as  }\  t\to \infty.
		\end{align*}
		This implies that $\kappa_3, \kappa_4 \in \mathfrak{K}$, which completes the proof.
	\end{proof}
\end{proposition}

\subsection{Random dynamical system}

Remember that Hypothesis \ref{assump1} is satisfied and that $\delta$ has the property stated there. Let us fix $\nu> 0$, and the parameters $\chi\geq 0$ and  $\xi \in (\delta, 1/2)$.

Using a random transformation (known as Doss-Sussmann transformation, \cite{Doss_1977,Sussmann_1978}), we get a random partial differential equation which equivalent to   system \eqref{STGF}. Let us define 
\begin{align}\label{D-S_trans}
	\y^{\chi}:=\v - \boldsymbol{z}_{\chi}(\omega).
\end{align}
For convenience, we write $\y^{\chi}(t)=\y(t)$ and $\boldsymbol{z}_{\chi}(\omega)(t)=\boldsymbol{z}(t)$. Then $\y(\cdot)$ satisfies the following system:
\begin{equation}\label{CTGF-with-Pressure}
	\left\{ 
	\begin{aligned}
		\frac{\d \y}{\d t} & = - \underbrace{\nabla (\textbf{P}- {p})}_{:=\nabla\widehat{\textbf{P}}}+\nu \Delta \y-\big((\y+\z)\cdot \nabla\big)(\y+\z)+\alpha\text{div}((\Arm(\y+\z))^2)  && \\ & \quad +\beta \text{div}(|\Arm(\y+\z)|^2\Arm(\y+\z))  +\chi \z + \f  &&  \text{in }  \mathcal{O} \times (0,\infty),\\
		\text{div}\; \y&=0 \quad &&  \text{in } \mathcal{O} \times [0,\infty),\\
		\y &= \boldsymbol{0} &&  \text{on } \partial\mathcal{O}\times [0,\infty),\\
		\y(x,0)&=\x - \boldsymbol{z}_{\chi}(\omega)(0)=:\y_0  \quad &&  \text{in } \mathcal{O},
	\end{aligned}
	\right.
\end{equation}
where $p$ is the scalar appearing in \eqref{eqn_z_alpha}, and the following projected system:
\begin{equation}\label{CTGF}
	\left\{
	\begin{aligned}
		\frac{\d\y}{\d t} &= -\nu \A\y -  \B(\y + \boldsymbol{z})-\alpha\J(\y + \boldsymbol{z})-\beta\K(\y + \boldsymbol{z}) + \chi \boldsymbol{z} + \mathcal{P}\f, \\
		\y(0)& = \boldsymbol{x} - \boldsymbol{z}_{\chi}(\omega)(0)=:\y_0.
	\end{aligned}
	\right.
\end{equation}
Since $\boldsymbol{z}_{\chi}(\omega) \in C_{1/2} (\mathbb{R};\mathfrak{X}), $ then $\boldsymbol{z}_{\chi}(\omega)(0)$ is a well defined element of $\H$. For each fixed $\omega\in\Omega$,   system \eqref{CTGF} is a deterministic system. Let us now provide the definition of weak solution (in the deterministic sense, for each  fixed $\omega$) for \eqref{CTGF}. 
\begin{definition}\label{defn-CTGF}
	Assume that $\y_0 \in \H$, $\boldsymbol{z}\in\mathrm{L}^2_{\mathrm{loc}}([0,\infty);\H)\cap\mathrm{L}^4_{\mathrm{loc}}([0,\infty);\mathbb{W}^{1,4}(\mathcal{O}))$ and $\f\in \H^{-1}(\mathcal{O})$. A function $\y$ is called a \textit{weak solution} of   system \eqref{CTGF} on the time interval $[0, \infty)$, if 
	\begin{align*}
		\y &\in  \mathrm{C}_{w}([0,\infty); \H) \cap \mathrm{L}^{2}_{\mathrm{loc}}(0,\infty; \V)\cap \mathrm{L}^{4}_{\mathrm{loc}}(0,\infty; \mathbb{W}^{1,4}(\mathcal{O})), 
		\\
		 \frac{\d\y}{\d t}&\in\mathrm{L}^{2}_{\mathrm{loc}}(0,\infty;\V') + \mathrm{L}^{\frac43}_{\mathrm{loc}}(0,\infty;\mathbb{W}^{-1,\frac43}(\mathcal{O})),
	\end{align*}
	 and it satisfies 
	\begin{itemize}
		\item [(i)] for any $\boldsymbol{\phi}\in \X,$ 
		\begin{align*}
		&	\left<\frac{\d\y(t)}{\d t}, \boldsymbol{\phi}\right>
		\nonumber\\ &=  - \left\langle \nu \A\y(t)+\B(\y(t)+\boldsymbol{z}(t))+\alpha\J(\y(t)+\boldsymbol{z}(t))+\beta\K(\y(t)+\boldsymbol{z}(t)) - \chi\boldsymbol{z}(t)- \f , \boldsymbol{\phi} \right\rangle,
		\end{align*}
		for a.e. $t\in[0,\infty);$
		\item [(ii)] the initial data:
		$$\y(0)=\y_0 \ \text{ in }\ \H.$$
	\end{itemize}
\end{definition}

{\begin{theorem}\label{solution}
	Assume that \eqref{condition1} is satisfied, $\chi\geq0$, $\v_0 \in \H$, $\f\in \H^{-1}(\mathcal{O}) $ and $\boldsymbol{z}\in\mathrm{L}^2_{\mathrm{loc}}([0,\infty);\H)\cap\mathrm{L}^4_{\mathrm{loc}}([0,\infty);\mathbb{W}^{1,4}(\mathcal{O}))$. Then, there exists a unique weak solution $\y(\cdot)$ to   system \eqref{CTGF} in the sense of Definition \ref{defn-CTGF} which satisfies the following energy equality: 
		\begin{align}\label{eeq}
			&\|\y(t)\|_{2}^2+\nu\int_0^t\|\Arm(\y(s))\|_{2}^2\d s + \beta\int_0^t\|\Arm(\y(s)+\z(s))\|_{4}^4\d s + \alpha\int_0^t \Tr\big([\Arm(\y(s)+\z(s))]^3\big)\d s   \nonumber\\&= \|\y_0\|_{2}^2 +2\int_0^t\langle\B(\y(s)+\z(s))+\alpha\J(\y(s)+\z(s))+\beta\K(\y(s)+\z(s)),\z(s)\rangle\d s 
			\nonumber\\ & \quad +2\int_0^t\langle \chi\z(s)+ \f,\y(s)\rangle\d s,
		\end{align}
		for all  $t\in[0,\infty)$. Consequently, $\y \in  \mathrm{C}([0,\infty); \H)$. 
	\end{theorem}

	\begin{proof}[Proof of Theorem \ref{solution}]
	Let us fix $T > 0$. Note that it is enough to prove the result on the interval
	$[0,T]$. The proof is divided into the following four steps.
		\vskip 2mm
		\noindent
		\textbf{Step I.} \textit{Finite-dimensional approximation and weak limits.} 		
For each $n\in\N$, we search for approximate solutions of the form 
\begin{align}
 \y^m(x,t)=\sum_{k=1}^m g^m_k(t)\boldsymbol{e}_k(x), \;\; \boldsymbol{e}_k\in\H_m,
\end{align} 
where $\{\boldsymbol{e}_k\}_{k\in\N}$ is defined in \eqref{L3}, $\H_{m}$ is defined in Subsection \ref{C_O} and  the coefficients $g^m_1,\ldots,g^m_m$ are solutions of the following $m$ ordinary differential equations:
\begin{equation}\label{finite-dimS}
\left\{
\begin{aligned}
\frac{\d}{\d t} \left(\y^m , \boldsymbol{e}_j\right)&=-\nu (\A_m\y^m-\B_{m}(\y^m+\z)-\alpha\J_{m}(\y^m+\z)-\beta\K_{m}(\y^m+\z)+\chi\z_m+\f_m, \boldsymbol{e}_j),\\
(\y^m(0), \boldsymbol{e}_j)&=(\P_m[\x-\z(\omega)(0)], \boldsymbol{e}_j)=:({\y_0}_m, \boldsymbol{e}_j),
\end{aligned}
\right.
\end{equation}
for $j=1,\cdots,m$, where  
\begin{align*}
\A_m\y^m=\P_m\A\y^m, \quad  \B_{m}(\y^m+\z)=\mathrm{P}_m\B(\y^m+\z), \quad  \J_{m}(\y^m+\z)=\mathrm{P}_m\J(\y^m+\z),\\  \K_{m}(\y^m+\z)=\mathrm{P}_m\K(\y^m+\z), \quad  \z_m=\P_m\z, \quad  \f_m=\P_m[\mathcal{P}\f]  \quad  \text{ and } \quad {\y_{0}}_m=\P_m[\y_{0}].
\end{align*}
Since $\A_m\cdot + \B_{m}(\cdot+\z)+\alpha\J_{m}(\cdot+\z)+\beta\K_{m}(\cdot+\z)$ is locally Lipschitz (see Subsection \ref{sec-operators}), system \eqref{finite-dimS} has a unique local solution $\y^m\in\mathrm{C}([0,T^*_m];\H_m)$, for some $0<T^*_m\leq T$. The following uniform estimates show that the time $T^*_m$ can be extended to time $T$. From the first equation of \eqref{finite-dimS} and \eqref{b0}, we obtain 
\begin{align}\label{S1}
&	\frac{1}{2}\frac{\d}{\d t}\|\y^m(t)\|^2_{2} +\frac{\nu}{2}\|\Arm(\y^m(t))\|^2_{2}+\frac{\beta}{2}\|\Arm(\y^m(t)+\z(t))\|^4_{4}
\nonumber\\	&= - b(\y^m(t)+\z(t), \y^m(t) , \z(t))  +\frac{\alpha}{2} \int_{\O}(\Arm(\y^m(t)+\z(t)))^2:\Arm(\y^m(t))\d x 
\nonumber\\ & \quad +\frac{\beta}{2} \int_{\O}|\Arm(\y^m(t)+\z(t))|^2[\Arm(\y^m(t)+\z(t)):\Arm(\z(t))]\d x + \chi (\z(t),\y^m(t))  +\langle\f,\y^m(t)\rangle,
\end{align}
for all $t\in[0,T]$. Next, using the H\"older's inequality, Lemma \ref{Sobolev-embedding3}, Korn-type inequality \eqref{Korn-ineq}, \eqref{condition1} (see Remark \ref{rem-condition1} also), Gagliardo-Nirenberg inequality \eqref{Gen_lady-4} and Young's inequality, we estimate the  terms of the right hand side of \eqref{S1} as 
\begin{align}
|b(\y^m+\z, \y^m, \z)| &\leq  |b(\y^m, \y^m+\z, \z)| + |b(\z, \y^m, \z)| 
\nonumber\\ & \leq \|\y^m\|_2 \|\nabla(\y^m+\z)\|_{4} \|\z\|_{4} + \|\z\|_4^2 \|\nabla\y^m\|_2 
\nonumber\\ & \leq C \|\Arm(\y^m)\|_2 \|\Arm(\y^m+\z)\|_{4} \|\z\|_{4} + C \|\z\|_4^2 \|\Arm(\y^m)\|_2 
			\nonumber\\ & \leq \frac{\nu\varepsilon_0}{24} \|\Arm(\y^m)\|^2_{2} +\frac{\beta\varepsilon_0}{12}  \|\Arm(\y^m+\z)\|^4_{4}    + C \|\z\|_{\mathbb{W}^{1,4}}^4,\label{Sb}\\
			\left\vert\frac{\alpha}{2} \int_{\O}(\Arm(\y^m+\z))^2:\Arm(\y^m)\d x \right\vert
			& \leq  \frac{\left\vert\alpha\right\vert}{2} \int_{\O} \left\vert \Arm(\y^m+\z)\right\vert^2 \left\vert\Arm(\y^m)\right\vert\d x 
			\nonumber\\ & \leq  \frac{\left\vert\alpha\right\vert}{2}  \| \Arm(\y^m+\z)\|^2_{4} \|\Arm(\y^m)\|_{2} 
			\nonumber\\ & \leq  \frac{\nu(1-\varepsilon_0)}{4} \|\Arm(\y^m)\|^2_{2}  + \frac{\alpha^2}{4\nu(1-\varepsilon_0)}   \| \Arm(\y^m+\z)\|^4_{4} 
			\nonumber\\ & \leq  \frac{\nu(1-\varepsilon_0)}{4} \|\Arm(\y^m)\|^2_{2}  + \frac{\beta(1-\varepsilon_0)}{2}  \|\Arm(\y^m+\z)\|^4_{4} \\
			\left|\frac{\beta}{2} \int_{\O}|\Arm(\y^m+\z)|^2[\Arm(\y^m+\z):\Arm(\z)]\d x\right|  & \leq \frac{\beta}{2} \|\Arm(\y^m+\z)\|^3_4\|\Arm(\z)\|_4
		\nonumber\\ & \leq  \frac{\beta\varepsilon_0}{12}  \|\Arm(\y^m+\z)\|^4_{4}      + C \|\Arm(\z)\|^4_{4},\\
		| \chi (\z,\y^m)|& \leq  \chi \|\z\|_2\|\y^m\|_2 \leq C \|\z\|_2\|\Arm(\y^m)\|_2
		\nonumber\\ &  \leq \frac{\nu\varepsilon_0}{24} \|\Arm(\y^m)\|^2_{2} + C \|\z\|^2_2, \\
			\big|\langle\f,\y^m\rangle\big| 
		& \leq \|\f\|_{\H^{-1}} \|\y^m\|_{\V}  \leq C \|\f\|_{\H^{-1}}\|\Arm(\y^m)\|_{2} 
		\nonumber \\
		& \leq   \frac{\nu\varepsilon_0}{24} \|\Arm(\y^m)\|^2_{2}    + C  \|\f\|^2_{\H^{-1}}.
		\label{S2}
	\end{align}
		Combining \eqref{S1}-\eqref{S2}, we deduce
		\begin{align}\label{S3}
			&\frac{\d}{\d t}\|\y^m(t)\|^2_{2}+\frac{\nu}{2}\left(1+\frac{\varepsilon_0}{2}\right) \|\Arm(\y^m(t))\|^2_{2} + \frac{\beta\varepsilon_0}{2}\|\Arm(\y^m(t)+\z(t))\|^4_{4}
			\nonumber\\ &	\leq C \bigg[\|\f\|^2_{\H^{-1}}  + \|\z(t)\|^2_{2} + \|\z(t)\|^4_{\mathbb{W}^{1,4}}  \bigg],
		\end{align}
		which gives for all $t\in[0,T]$
		\begin{align}\label{S4}
			&	\|\y^m(t)\|^2_{2} + \frac{\nu}{2}\left(1+\frac{\varepsilon_0}{2}\right)\int_{0}^{t}\|\Arm(\y^m(s))\|^2_{2} \d s   +\frac{\beta\varepsilon_0}{2}\int_{0}^{t}\|\Arm(\y^m(s))\|^4_{4}\d s
			\nonumber\\& 
			\leq \|\y^m(0)\|^2_{2}   +C\int_{0}^{t}\bigg[ \|\f\|^2_{\H^{-1}} +  \|\z(s)\|^2_{2} 
			 + \|\z(s)\|^4_{\mathbb{W}^{1,4}}   \bigg]\d s.
		\end{align}
		Hence, an application of Gronwall's inequality and the fact that $\|\y^m(0)\|_{2}\leq\|\y_0\|_{2}$, $\f\in\H^{-1}(\O)$,   and $\boldsymbol{z}\in\mathrm{L}^2(0,T;\H)\cap\mathrm{L}^4(0,T;\mathbb{W}^{1,4}(\mathcal{O}))$, we have from \eqref{S4} that
		\begin{align}\label{S5}
			\{\y^m\}_{m\in\N} \text{ is a bounded sequence in }\mathrm{L}^{\infty}(0,T;\H)\cap\mathrm{L}^{2}(0,T;\V)\cap\mathrm{L}^{4}(0,T;\mathbb{W}^{1,4}(\O)).
		\end{align}
For any arbitrary element $\boldsymbol{\phi}\in \mathrm{L}^4(0,T;\mathbb{U})$, using H\"older's inequality and Lemma \ref{lem-Pm}, we have from \eqref{finite-dimS} that
		\begin{align}\label{eqn-deri-bound}
			&\left|\int_{0}^{T}\left\langle\frac{\d\y^m(t)}{\d t},\boldsymbol{\phi}(t)\right\rangle\d t\right|
			\nonumber\\&\leq \int_{0}^{T}\bigg[\nu\left|(\nabla\y^m(t),\nabla\P_m\boldsymbol{\phi}(t))\right|+ \left| b(\y^m(t)+\z(t),\P_m\boldsymbol{\phi}(t),\y^m(t)+\z(t))\right| + \chi|(\z(t),\P_m\boldsymbol{\phi}(t)) |
			\nonumber\\ & \qquad + | \langle\f,\P_m\boldsymbol{\phi}(t)\rangle |\bigg]\d t 
			 + |\alpha|\int_0^{T} \int_{\O}|\Arm(\y^m(x,t)+\z(x,t))|^2|\nabla\P_m\boldsymbol{\phi}(x,t)|\d x\d t 
			 \nonumber\\ & \qquad + \beta \int_0^T \int_{\O}|\Arm(\y^m(x,t)+\z(x,t))|^3|\nabla\P_m\boldsymbol{\phi}(x,t)|\d x\d t
			 \nonumber\\&\leq \int_{0}^{T}\bigg[\nu\|\y^m(t)\|_{\V}\|\P_m\boldsymbol{\phi}(t)\|_{\V} + \|\y^m(t)+\z(t)\|^2_{4}\|\P_m\boldsymbol{\phi}(t)\|_{\V}  + \chi\|\z(t)\|_2\|\P_m\boldsymbol{\phi}(t)\|_2
			 \nonumber\\ & \qquad +  \|\f\|_{\H^{-1}} \|\P_m\boldsymbol{\phi}(t)\|_{\V} \bigg]\d t 
			 + |\alpha|\int_0^{T} \|\Arm(\y^m(t)+\z(t))\|^2_{4} \|\P_m\boldsymbol{\phi}(t)\|_{\V} \d t 
			 \nonumber\\ & \qquad + \beta \int_0^T \|\Arm(\y^m(t)+\z(t))\|^3_4\|\nabla\P_m\boldsymbol{\phi}(t)\|_4\d t
			 \nonumber\\ & \leq C \bigg[ \|\y^m\|_{\mathrm{L}^{2}(0,T;\V)}\|\P_m\boldsymbol{\phi}\|_{\mathrm{L}^{2}(0,T;\V)} + \|\y^m+\z\|^2_{\mathrm{L}^{4}(0,T;\mathbb{L}^{4})}\|\P_m\boldsymbol{\phi}\|_{\mathrm{L}^{2}(0,T;\V)} + \|\z\|_{\mathrm{L}^{2}(0,T;\H)}\|\P_m\boldsymbol{\phi}\|_{\mathrm{L}^{2}(0,T;\H)} 
			 \nonumber\\ & \quad + T^{\frac12}\|\f\|_{\H^{-1}}  \|\P_m\boldsymbol{\phi}\|_{\mathrm{L}^{2}(0,T;\V)} + \|\Arm(\y^m+\z)\|^2_{\mathrm{L}^{4}(0,T;\mathbb{L}^{4})}\|\P_m\boldsymbol{\phi}\|_{\mathrm{L}^{2}(0,T;\V)} 
			 \nonumber\\ & \quad + \|\Arm(\y^m+\z)\|^3_{\mathrm{L}^{4}(0,T;\mathbb{L}^{4})}\|\nabla\P_m\boldsymbol{\phi}\|_{\mathrm{L}^{4}(0,T;\mathbb{L}^{4})}  \bigg]
			 \nonumber\\ & \leq C \bigg[ \|\y^m\|_{\mathrm{L}^{2}(0,T;\V)}  + \|\y^m+\z\|^2_{\mathrm{L}^{4}(0,T;\mathbb{L}^4)}  + \|\z\|_{\mathrm{L}^{2}(0,T;\H)}  
			  + T^{\frac12}\|\f\|_{\H^{-1}}   + \|\Arm(\y^m+\z)\|^2_{\mathrm{L}^{4}(0,T;\mathbb{L}^{4})} \bigg] \nonumber\\ & \quad \times  \|\P_m\boldsymbol{\phi}\|_{\mathrm{L}^{2}(0,T;\U)} 
			   + C \|\Arm(\y^m+\z)\|^3_{\mathrm{L}^{4}(0,T;\mathbb{L}^{4})}\|\P_m\boldsymbol{\phi}\|_{\mathrm{L}^{4}(0,T;\U)}  
			   \nonumber\\ & \leq C \bigg[ \|\y^m\|_{\mathrm{L}^{2}(0,T;\V)}  + \|\y^m+\z\|^2_{\mathrm{L}^{4}(0,T;\mathbb{L}^4)}  + \|\z\|_{\mathrm{L}^{2}(0,T;\H)}  
			   + T^{\frac12}\|\f\|_{\H^{-1}}   + \|\Arm(\y^m+\z)\|^2_{\mathrm{L}^{4}(0,T;\mathbb{L}^{4})} \bigg] \nonumber\\ & \quad \times  \|\boldsymbol{\phi}\|_{\mathrm{L}^{2}(0,T;\U)} 
			   + C \|\Arm(\y^m+\z)\|^3_{\mathrm{L}^{4}(0,T;\mathbb{L}^{4})}\|\boldsymbol{\phi}\|_{\mathrm{L}^{4}(0,T;\U)},
		\end{align}
		which implies that the sequence
		\begin{align}\label{eqn-seq-deri}
			\left\{\frac{\d \y^m}{\d t}\right\}_{m\in\N}  \text{ is  bounded in } \mathrm{L}^{\frac43}(0,T;\U^{\prime}).  
		\end{align}
		 Taking \eqref{S5}-\eqref{eqn-seq-deri} into account and using the \textit{Banach-Alaoglu theorem}, we obtain the existence of an element $\y\in\mathrm{L}^{\infty}(0,T;\H)\cap\mathrm{L}^{2}(0,T;\V)\cap\mathrm{L}^{4}(0,T;\mathbb{W}^{1,4}(\O))$ with $\frac{\d \y}{\d t}\in \mathrm{L}^{\frac43}(0,T;\U^{\prime})$ such that
		\begin{align}
			\y^m\xrightharpoonup{w^*}&\ \y &&\text{ in }\ \ \ \ \	\mathrm{L}^{\infty}(0,T;\H),\label{S7}\\
			\y^m\xrightharpoonup{w}&\ \y   && \text{ in } \ \ \ \ \ \mathrm{L}^{2}(0,T;\V)\cap \mathrm{L}^{4}(0,T;\mathbb{W}^{1,4}(\O)),\label{S8}\\
			\frac{\d \y^m}{\d t}\xrightharpoonup{w}&\frac{\d \y}{\d t}   && \text{ in }  \ \ \ \ \ \mathrm{L}^{\frac43}(0,T;\U^{\prime}),\label{S8d}
		\end{align}
		along a subsequence (still denoted by the same symbol). In addition, we have that  the sequence
		\begin{align}\label{eqn-seq-G}
			\mbox{$\{\mathscr{G}(\y^m)\}_{m\in\N}$ is  bounded in $\mathrm{L}^{2}(0,T;\V')+\mathrm{L}^{\frac43}(0,T;\mathbb{W}^{-1,\frac{4}{3}}(\O)),$ }
		\end{align}
		where $\mathscr{G}(\cdot)$ given by \eqref{opr-G}, which implies that there exists $\mathscr{G}_{0}\in \mathrm{L}^{2}(0,T;\V')+\mathrm{L}^{\frac43}(0,T;\mathbb{W}^{-1,\frac{4}{3}}(\O))$ such that
		\begin{align}
			\mathscr{G} (\y^m) \xrightharpoonup{w}& \;\; \mathscr{G}_{0}    \text{ in }  \ \ \ \ \ \mathrm{L}^{2}(0,T;\V')+\mathrm{L}^{\frac43}(0,T;\mathbb{W}^{-1,\frac{4}{3}}(\O)).\label{S9d}
		\end{align}	
	 Moreover, we also obtain for further use that $\y^m(t)$ converges to $\y(t)$ weakly in $\H$ for all $t\in[0,T]$ as follows:  
		From \eqref{eqn-deri-bound}, we have $\big\|\frac{\d\y_{m}}{\d t}\big\|_{\mathrm{L}^{\frac{4}{3}}(0, T; \U')}\leq C,$ for some $C>0$ and all $m\in \N.$ Therefore, by the Cauchy-Schwarz inequality, for all $0\leq t \leq t+a \leq T$, $m\in \N$ and $\boldsymbol{\psi}\in \U$, we obtain
		\begin{align}
			|(\y^m(t+a)-\y^m(t), \boldsymbol{\psi})|\leq \int_{t}^{t+a}\bigg|\bigg\langle\frac{\d\y^{m}(s)}{\d t} , \boldsymbol{\psi} \bigg\rangle\bigg|\d s\leq C(T) a^{\frac{1}{4}} \|\boldsymbol{\psi}\|_{\U}.
		\end{align}
	This shows that the sequence $\{(\y^m(\cdot), \boldsymbol{\psi})\}_{m\in\N}$ is uniformly continuous on $[0,T]$. Hence by the Arzela-Ascoli Theorem, there exists a subsequence of  $\{\y^{m}\}_{m\in\N}$ (still denoted by the same symbol) such that $(\y^m(\cdot), \boldsymbol{\psi})\to (\y(\cdot), \boldsymbol{\psi})$ uniformly on $[0,T]$. Since $\U$ is dense in $\H$, and $\sup\limits_{m\in\N, t\in[0,T]} \|\y^m(t)\|_{\H}<\infty$, then for any $\boldsymbol{\psi}\in\H$, $(\y^m(\cdot), \boldsymbol{\psi}) \to (\y(\cdot), \boldsymbol{\psi})$,  uniformly in	 $t\in[0,T]$, which shows that 
	\begin{align}\label{eqn-weak-H}
	\mbox{$\y^m(t)$ converges to $\y(t)$ weakly in $\H$ for all $t\in[0,T]$.}	
	\end{align}
	Note that $\f_m\to \mathcal{P}\f $ in $\H^{-1}(\O)$ and $\z_m\to \z$ in $\mathrm{L}^2(0,T;\H)$.  Therefore, on passing to limit as $m\to\infty$ in \eqref{finite-dimS}, the limit $\y(\cdot)$ satisfies:
		\begin{equation}\label{eqn-TGF-limit}
			\left\langle\frac{\d\y}{\d t},\boldsymbol{e}_j\right\rangle =\left\langle-\mathscr{G}_{0} +\chi\z +\mathcal{P}\f, \boldsymbol{e}_j\right\rangle, 
		\end{equation}
		for $j= 1,2,3,\cdots$. Since $-\mathscr{G}_{0} +\chi\z +\mathcal{P}\f\in \mathrm{L}^{2}(0,T;\V')+\mathrm{L}^{\frac43}(0,T;\mathbb{W}^{-1,\frac{4}{3}}(\O)) $, we deduce from \eqref{eqn-TGF-limit} that  $\frac{\d \y}{\d t}\in \mathrm{L}^{2}(0,T;\V')+\mathrm{L}^{\frac43}(0,T;\mathbb{W}^{-1,\frac43}(\O))$. Due to fact that $\{\boldsymbol{e}_j\}_{j\in\N}$ is dense in $\X$, we have 
		\begin{equation}\label{eqn-TGF-limit-2}
			\left\langle\frac{\d\y}{\d t}, \boldsymbol{\varphi}\right\rangle =\left\langle-\mathscr{G}_{0} +\chi\z +\mathcal{P}\f, \boldsymbol{\varphi} \right\rangle, 
		\end{equation}
		for all $\boldsymbol{\varphi}\in \X$.
		Note that the embedding of $\H\subset \V^{\prime}+\mathbb{W}^{-1,\frac43}(\O)$	is continuous and $\y\in \mathrm{L}^{\infty}(0,T;\H)$ implies $\y\in \mathrm{L}^{\infty}(0,T;\V^{\prime}+\mathbb{W}^{-1,\frac43}(\O))$. Thus, we get $\y,\frac{\d \y}{\d t}\in \mathrm{L}^{\frac43}(0,T;\V^{\prime}+\mathbb{W}^{-1,\frac43}(\O))$ and then invoking \cite[Theorem 2,
		Section 5.9.2]{LCE}, it is immediate that $\y\in\mathrm{C}([0,T];\V^{\prime}+\mathbb{W}^{-1,\frac43}(\O))$. Since $\H$ is reflexive, using \cite[Proposition 1.7.1]{PCAM}, we obtain $\y\in\mathrm{C}_{w}([0,T];\H)$ and the map $t\mapsto\|\y(t)\|_{2}$ is bounded. Thus the condition (ii) in the Definition \eqref{defn-CTGF} makes sense.

		Next, note that \cite[Chapter II, Theorem 1.8]{Chepyzhov+Vishik_2002}, $\y\in\mathrm{L}^{2}(0,T;\V)\cap\mathrm{L}^{4}(0,T;\mathbb{W}^{1,4}(\O))$ and $\frac{\d \y}{\d t}\in \mathrm{L}^{2}(0,T;\V')+\mathrm{L}^{\frac43}(0,T;\mathbb{W}^{-1,\frac43}(\O))$ imply $\y\in \C([0,T];\H)$, the real-valued function $t\mapsto\|\y(t)\|_{2}^2$ is absolutely continuous and the following equality is satisfied:
		\begin{align}\label{EE1}
			\frac{\d}{\d t}\|\y(t)\|_{2}^2 = 2 \left<\frac{\d \y(t)}{\d t},\y(t)  \right>,\  \ \ \ \text{ for a.e. } t\in[0,T].
		\end{align}
        
		\vskip 2mm
		\noindent
		\textbf{Step II.} \textit{Minty-Browder technique:}	
		From \eqref{EE1}, for a measurable function $\eta(t)\geq0$, we have the following equality:
		\begin{align}\label{EE2}
			e^{-2\eta (t)}\|\y(t)\|_{2}^2 + 2 \int_{0}^{t}e^{-2\eta (s)}\left\langle  \mathscr{G}_{0}(s) - \chi\z(s) - \f + \eta^{\prime}(s) \y(s), \y(s) \right\rangle \d s =\|\y(0)\|_{2}^2,
		\end{align}
		for all $t\in[0,T]$. Similar to \eqref{EE2},  for    system \eqref{finite-dimS}, we obtain the following energy equality:
		\begin{align}\label{EE3}
			e^{-2\eta (t)}\|\y^m(t)\|_{2}^2 + 2 \int_{0}^{t}e^{-2\eta (s)}\left\langle  \mathscr{G}(\y^m(s)) -\chi\z_m(s) - \f_m + \eta^{\prime}(s) \y^m(s), \y^m(s) \right\rangle \d s =\|\y^m(0)\|_{2}^2,
		\end{align}
		for all $t\in[0,T]$.	 Remember that $\y^m(0)=\mathrm{P}_m\v(0)$, and hence the initial value $\y^m(0)$ converges strongly in $\H$, that is, we have
		\begin{align}\label{MB1}
			\lim_{m\to\infty} \|\y^m(0)-\y(0)\|_{2}=0.
		\end{align} 
		For any $\boldsymbol{\phi}\in \mathrm{L}^{\infty}(0,T;\H_{n})$ with $n<m$, we define
		$$\eta(t) = \frac{(C_S C_K)^2}{\nu\varepsilon_0}\int_{0}^{t} \|\Arm(\boldsymbol{\phi}(s)+\z(s))\|_{{4}}^2\d s, \mbox{ for all } t\in[0,T].$$ 
		Using \eqref{CL1}, we obtain
		\begin{align}\label{MB2}
			\int_{0}^{T} e^{-2\eta(t)} \{\left<\mathscr{G}(\boldsymbol{\phi}(t))-\mathscr{G}(\y^m(t)),\boldsymbol{\phi}(t)-\y^m(t)\right>  + \eta^{\prime}(t) (\boldsymbol{\phi}(t)-\y^m(t), \boldsymbol{\phi}(t)-\y^m(t))\}\d t \geq0.
		\end{align}
		Making use of \eqref{EE3} in \eqref{MB2}, we obtain
		\begin{align}\label{MB3}
			&	\int_{0}^{T} e^{-2\eta (t)} \left<\mathscr{G}(\boldsymbol{\phi}(t)) +\eta^{\prime}(t) \boldsymbol{\phi} (t),\boldsymbol{\phi}(t)-\y^m(t)\right> \d t
			\nonumber\\ &  \geq \int_{0}^{T} e^{-2\eta( t)}\left<\mathscr{G}(\y^m(t))+\eta^{\prime}(t) \y^m(t),\boldsymbol{\phi}(t)-\y^m(t)\right> \d t 
			\nonumber\\ &  = \int_{0}^{T} e^{-2\eta( t)}\left<\mathscr{G}(\y^m(t))+\eta^{\prime}(t) \y^m(t),\boldsymbol{\phi}(t)\right> \d t  + \frac12\bigg[e^{-2\eta (T)}\|\y^m(T)\|_{2}^2- \|\y^m(0)\|_{2}^2\bigg]
			\nonumber\\ & \qquad -  \chi\int_{0}^{T}e^{-2\eta (t)}\left(   \z_m(t) , \y^m(t) \right) \d t -  \int_{0}^{T}e^{-2\eta (t)}\left\langle   \f_m , \y^m(t) \right\rangle \d t. 
		\end{align}
		Taking limit infimum on both sides of \eqref{MB3}, we deduce
		\begin{align}\label{MB4}
			&	\int_{0}^{T} e^{-2\eta (t)} \left<\mathscr{G}(\boldsymbol{\phi}(t)) +\eta^{\prime}(t) \boldsymbol{\phi} (t),\boldsymbol{\phi}(t)-\y(t)\right> \d t
			\nonumber\\ &  \geq  \int_{0}^{T} e^{-2\eta (t)}\left<\mathscr{G}_{0}(t)+\eta^{\prime}(t) \y(t),\boldsymbol{\phi}(t)\right> \d t  + \frac12 \liminf_{n\to \infty}\bigg[e^{-2\eta (T)}\|\y^m(T)\|_{2}^2 - \|\y^m(0)\|_{2}^2\bigg]
			\nonumber\\ & \qquad -  \chi\int_{0}^{T}e^{-2\eta (t)}\left(   \z(t) , \y(t) \right) \d t -  \int_{0}^{T}e^{-2\eta (t)}\left\langle     \f , \y(t) \right\rangle \d t
			\nonumber\\ &  \geq  \int_{0}^{T} e^{-2\eta (t)}\left<\mathscr{G}_{0}(t)+\eta^{\prime}(t) \y(t),\boldsymbol{\phi}(t)\right> \d t  + \frac12 \bigg[e^{-2\eta (T)}\|\y(T)\|_{2}^2- \|\y(0)\|_{2}^2\bigg]
			\nonumber\\ & \qquad -  \chi\int_{0}^{T}e^{-2\eta (t)}\left(   \z(t) , \y(t) \right) \d t -  \int_{0}^{T}e^{-2\eta (t)}\left\langle    \f , \y(t) \right\rangle \d t,
		\end{align}
		where we have used the weak lower semicontinuity property of the $\H$-norm (Note that $\y^m(T)$ converges to $\y(T)$ weakly in $\H$ by  \eqref{eqn-weak-H}.) and the strong convergence of the initial data \eqref{MB1} in the final inequality. Now, using the equality \eqref{EE2} in \eqref{MB4}, we further have 
		\begin{align}\label{MB5}
			&	\int_{0}^{T} e^{-2\eta (t)} \left<\mathscr{G}(\boldsymbol{\phi}(t)) +\eta^{\prime}(t) \boldsymbol{\phi} (t),\boldsymbol{\phi}(t)-\y(t)\right> \d t
			\nonumber\\ &  \geq   \int_{0}^{T} e^{-2\eta (t)}\left<\mathscr{G}_{0}(t)+\eta^{\prime}(t) \y(t),\boldsymbol{\phi}(t)\right> \d t  - \int_{0}^{T} e^{-2\eta (t)}\left<\mathscr{G}_{0}(t)+\eta^{\prime}(t) \y(t),\y(t)\right> \d t
			\nonumber\\ &  \geq   \int_{0}^{T} e^{-2\eta (t)}\left<\mathscr{G}_{0}(t)+\eta^{\prime}(t) \y(t),\boldsymbol{\phi}(t) - \y(t)\right> \d t.  
		\end{align}
		Note that the estimate \eqref{MB5} holds true for any
		$\boldsymbol{\phi}\in\mathrm{L}^{\infty}(0,T;\H_m)$, $m\in\mathbb{N}$, since the  inequality given in \eqref{MB5} is
		independent of both $m$ and $n$. Using a density
		argument, one can show that the inequality \eqref{MB5} remains true for any
		$\boldsymbol{\phi}\in\mathrm{L}^{\infty}(0,T;\H)\cap\mathrm{L}^2(0,T;\V)\cap\mathrm{L}^4(0,T;\mathbb{W}^{1,4}(\O)).$ In fact, for any
		$\boldsymbol{\phi}\in\mathrm{L}^{\infty}(0,T;\H)\cap\mathrm{L}^2(0,T;\V)\cap\mathrm{L}^4(0,T;\mathbb{W}^{1,4}(\O)),$ there	exists a strongly convergent subsequence	$\boldsymbol{\phi}_m\in\mathrm{L}^{\infty}(0,T;\H)\cap\mathrm{L}^2(0,T;\V)\cap\mathrm{L}^4(0,T;\mathbb{W}^{1,4}(\O)),$ that
		satisfies the inequality \eqref{MB5}.
		
		Taking $\boldsymbol{\phi}=\y+ r \w$, $r>0$, where $\w  \in\mathrm{L}^{\infty}(0,T;\H)\cap\mathrm{L}^2(0,T;\V)\cap\mathrm{L}^4(0,T;\mathbb{W}^{1,4}(\O)),$ and substituting for $\boldsymbol{\phi}$ in \eqref{MB5}, we get
		\begin{align}\label{MB6}
			&	\int_{0}^{T} e^{-2\eta (t)} \left<\mathscr{G}(\y(t)+ r \w(t)) - \mathscr{G}_{0}(t) + r\eta(t) \w(t), r \w(t)\right> \d t
			\geq   0.
		\end{align}
		Dividing the inequality \eqref{MB6} by $r$, using the
		hemicontinuity property of the operator $\mathscr{G}(\cdot)$ (see Lemma \ref{lem-demi-G}), and then passing $r\to 0$, we find 
		\begin{align}\label{MB7}
			&	\int_{0}^{T} e^{-2\eta (t)} \left<\mathscr{G}(\y(t)) - \mathscr{G}_{0}(t),  \w(t)\right> \d t
			\geq   0,
		\end{align}
		for any $\w  \in \mathrm{L}^{\infty}(0,T;\H)\cap\mathrm{L}^2(0,T;\V)\cap \mathrm{L}^4(0,T;\mathbb{W}^{1,4}(\O)).$ Therefore, from \eqref{MB7}, we deduce that $\mathscr{G}(\y(\cdot))=\mathscr{G}_{0}(\cdot).$ 
		In addition, $\y(\cdot)$ satisfies the energy equality \eqref{eeq} for all $t\in[0,T]$.
		\vskip 2mm
		\noindent
		\textbf{Step III.}	\textit{Uniqueness:} Define $\mathcal{Y}=\y_1-\y_2$, where $\y_1$ and $\y_2$ are two weak solutions of   system \eqref{CTGF} in the sense of Definition \ref{defn-CTGF}. Then $\mathcal{Y}\in\mathrm{C}([0,T];\H)\cap\mathrm{L}^{2}(0,T;\V)\cap \mathrm{L}^4(0,T;\mathbb{W}^{1,4}(\O))$ and satisfies
		\begin{equation}\label{Uni}
			\left\{
			\begin{aligned}
				\frac{\d\mathcal{Y}(t)}{\d t} &= -\left[\mathscr{G}(\y_1(t))-\mathscr{G}(\y_2(t))\right], \\
				\mathcal{Y}(0)&= \textbf{0},
			\end{aligned}
			\right.
		\end{equation}
		in the weak sense.	Therefore, we have
		\begin{align}\label{U1}
				\frac{1}{2}\frac{\d}{\d t}\|\mathcal{Y}(t)\|^2_{2} 
				&= - \left<\mathscr{G}(\y_1(t))-\mathscr{G}(\y_2(t)), \y_1(t)-\y_2(t)\right> 
				\nonumber\\ & \leq \frac{(C_S C_K)^2}{\nu\varepsilon_0} \|\Arm(\y_2+\z)(t)\|_{{4}}^2\; \|\mathcal{Y}(t)\|_{2}^2,
		\end{align}
		for a.e. $t\in[0,T]$, where we have used \eqref{CL1}. An application of the variation of constants formula and the fact that $\mathcal{Y}(0)=\textbf{0}$ give $\y_1(t)=\y_2(t)$, for all $t\in[0,T]$ in $\H$, which proves the uniqueness.
	\end{proof}	

}

 Next, we show that the weak solution of system \eqref{CTGF} is continuous with respect to given data (particularly $\x$, $\f$ and $\z$).

\begin{lemma}\label{RDS_Conti1}
	Let \eqref{condition1} be satisfied. For some $T >0$ fixed, assume that $\boldsymbol{x}_m \to \boldsymbol{x}$ in $\H$, $\f_m \to \f \ \text{ in }\ \H^{-1}(\mathcal{O})$ and $\z_m \to \z\ \text{ in }\ \mathrm{L}^2(0,T;\H)\cap\mathrm{L}^4(0,T;\mathbb{W}^{1,4}(\mathcal{O}))$.  Let us denote by $\y(t, \z,\f,\boldsymbol{x}),$ the solution of   system \eqref{CTGF} and by $\y(t, \z_m,\f_m,\boldsymbol{x}_m),$  the solution of   system \eqref{CTGF} with $\z, \f, \boldsymbol{x}$ being replaced by $\z_m, \f_m, \boldsymbol{x}_m$. 
	Then \begin{align}\label{5.20}
		\y(\cdot, \z_m,\f_m,\boldsymbol{x}_m) \to \y(\cdot, \z,\f,\boldsymbol{x}) \ \text{ in } \ \mathrm{C}([0,T];\H)\cap\mathrm{L}^2 (0, T;\V)\cap\mathrm{L}^4(0,T;\mathbb{W}^{1,4}(\mathcal{O})).
	\end{align}
	In particular, $\y(T, \z_m,\f_m,\boldsymbol{x}_m) \to \y(T, \z,\f,\boldsymbol{x})$ in $\H$.
\end{lemma}	
\begin{proof}
	Let us introduce the following notations which help us to simplify the proof: 
	\begin{align*}
		\y_m (\cdot) &:= \y(\cdot, \z_m,\f_m,\boldsymbol{x}_m), \ \  \y(\cdot) := \y(\cdot, \z,\f,\boldsymbol{x}),\ \   \mathfrak{F}_m (\cdot):= \y_m(\cdot) - \y(\cdot),\\     
		\hat{ \z}_m(\cdot) &:= \z_m(\cdot) - \z(\cdot), \ \  \hat{ \f}_m(\cdot) := \mathcal{P}\f_m(\cdot) - \mathcal{P}\f(\cdot).
	\end{align*}
	Then $\mathfrak{F}_m$ satisfies the following system:
	\begin{equation}\label{finite-dimS_1}
		\left\{
		\begin{aligned}
			\frac{\d\mathfrak{F}_m}{\d t}  & = -\left[\mathscr{G}(\y_m)-\mathscr{G}(\y)\right] + \chi \hat{ \z}_m + \hat{ \f}_m, \\
			\mathfrak{F}_m(0)&= \boldsymbol{x}_m - \boldsymbol{x}.
		\end{aligned}
		\right.
	\end{equation}
	Multiplying by $\mathfrak{F}_m(\cdot)$ to the first equation in \eqref{finite-dimS_1}, integrating over $\mathcal{O}$ and using \eqref{CL1}, we obtain 
	\begin{align}\label{5.22}
		\frac{1}{2}& \frac{\d}{\d t}\|\mathfrak{F}_m(t) \|^2_{2}
		\nonumber\\
		&= -\left<\mathscr{G}(\y_m(t))-\mathscr{G}(\y(t)), \mathfrak{F}_m (t)\right>  + \chi( \hat{ \z}_m(t), \ \mathfrak{F}_m(t))+ \langle\hat{ \f}_m, \mathfrak{F}_m(t)\rangle
		\nonumber\\
		& \leq - \frac{\nu\varepsilon_0}{4} \|\Arm(\mathfrak{F}_m(t))\|^2_{2} - \frac{\beta\varepsilon_0}{4}\|\Arm(\mathfrak{F}_m(t))\|^4_{4} + \frac{(C_{S,3} C_K)^2}{\nu\varepsilon_0} \|\Arm(\y(t)+\z(t))\|_{{4}}^2 \|\mathfrak{F}_m(t)\|_{2}^2  
		\nonumber\\ & \quad + \chi(\hat{ \z}_m(t), \ \mathfrak{F}_m(t))+ \langle\hat{ \f}_m, \mathfrak{F}_m(t)\rangle,
	\end{align}
	for a.e. $t\in [0,T]$. Using  H\"older's inequality, \eqref{poin} and Young's inequality, we have 
	\begin{align}
		|\chi(\hat{ \z}_m, \ \mathfrak{F}_m)|& \leq \chi \|\mathfrak{F}_m\|_{2} \|\hat{ \z}_m\|_{2} \leq C \|\Arm(\mathfrak{F}_m)\|_{2} \|\hat{ \z}_m\|_{2} \leq \frac{\nu\varepsilon_0}{16}\|\Arm(\mathfrak{F}_m)\|^2_{2}+C\|\hat{ \z}_m\|^2_{2},\label{5.23}.
	\end{align} 
Similar to \eqref{S2}, we find 
\begin{align}\label{5.24}
	\big|\langle\hat{\f}_{m},\mathfrak{F}_m\rangle\big| 
	& \leq \frac{\nu\varepsilon_0}{16} \|\Arm(\mathfrak{F}_m)\|^2_{2}  
	+  C \|\hat{\f}_{m}\|^2_{\H^{-1}},
\end{align}
	Combining \eqref{5.22} and \eqref{5.24}, we deduce 
	\begin{align*}
		&\frac{\d}{\d t}\|\mathfrak{F}_m(t) \|^2_{2} + \frac{\nu\varepsilon_0}{4} \|\Arm(\mathfrak{F}_m(t))\|^2_{2} + \frac{\beta\varepsilon_0}{4}\|\Arm(\mathfrak{F}_m(t))\|^4_{4}   \nonumber\\ & \leq  \frac{(C_{S,3} C_K)^2}{\nu\varepsilon_0} \|\Arm(\y(t)+\z(t))\|_{{4}}^2 \|\mathfrak{F}_m(t)\|_{2}^2    
	  +  C \|\hat{ \z}_m(t)\|^2_{2} +  C  \|\hat{\f}_{m}\|^2_{\H^{-1}} , 
	\end{align*}
	for a.e. $t\in[0,T]$.	In view of Gronwall's inequality, we obtain
	\begin{align}\label{Energy_esti_n_2_1}
		& \|\mathfrak{F}_m(t)\|^2_{2} + \frac{\nu\varepsilon_0}{4} \int_{0}^{t} \|\Arm(\mathfrak{F}_m(s)) \|^2_{2}\d s + \frac{\beta\varepsilon_0}{4} \int_{0}^{t} \|\Arm(\mathfrak{F}_m(s)) \|^4_{4}\d s  
		\nonumber\\ &   \leq \biggl\{\|\mathfrak{F}_m(0)\|^2_{2}  + C \int_{0}^{t} \bigg[\|\hat{ \z}_m(s)\|^2_{2} + \|\hat{\f}_{m}\|^2_{\H^{-1}} \bigg] \d s\biggr\}
		  \times \exp\left\{\frac{(C_{S,3} C_K)^2}{\nu\varepsilon_0} \int_{0}^{t} \|\Arm(\y(s)+\z(s))\|_{{4}}^2\d s\right\}, 
	\end{align}
	for all $ t\in[0, T]$.	Since, $\lim\limits_{m\to\infty}\|\mathfrak{F}_m(0)\|_{2} = \lim\limits_{m\to\infty} \|\boldsymbol{x}_m- \boldsymbol{x}\|_{2} = 0 $, $$\lim\limits_{m\to\infty} \int_{0}^{T} \bigg[\|\hat{ \z}_m(s)\|^2_{2} + \|\hat{\f}_{m}\|^2_{\H^{-1}} \bigg] \d s = 0,$$ and $\y+\z\in \mathrm{L}^4(0,T;\mathbb{W}^{1,4}(\mathcal{O}))$, then \eqref{Energy_esti_n_2_1} asserts that $$\|\mathfrak{F}_m(t)\|^2_{2} + \frac{\nu\varepsilon_0}{4} \int_{0}^{t} \|\Arm(\mathfrak{F}_m(s)) \|^2_{2}\d s + \frac{\beta\varepsilon_0}{4} \int_{0}^{t} \|\Arm(\mathfrak{F}_m(s)) \|^4_{4}\d s   \to 0$$ as $m\to\infty$ uniformly in $t\in[0, T].$ Since $\y_m(\cdot)$ and $\y(\cdot)$ are continuous, we further have  $$\y(\cdot, \z_m,\f_m,\boldsymbol{x}_m) \to \y(\cdot, \z,\f,\boldsymbol{x})\  \text{ in } \ \C([0, T]; \H)\cap \mathrm{L}^2(0, T; \V)\cap \mathrm{L}^4(0,T;\mathbb{W}^{1,4}(\mathcal{O})),$$
	as $m\to\infty$. This completes the proof.
\end{proof}

\begin{definition}
	We define a map $\Psi_{\chi} : [0,\infty) \times \Omega \times \H \to \H$ by
	\begin{align}
		(t, \omega, \boldsymbol{x}) \mapsto \y^{\chi}(t)  + \z_{\chi}(\omega)(t) \in \H,
	\end{align}
	where $\y^{\chi}(t) = \y(t, \z_{\chi}(\omega)(t),\f,\boldsymbol{x} - \z_{\chi}(\omega)(0))$ is a solution to   system \eqref{CTGF} with the initial condition $\boldsymbol{x} - \z_{\chi}(\omega)(0).$
\end{definition}

\begin{remark}
Let us recall from \eqref{O-U_conti}-\eqref{DOu1} that  
the  Ornstein-Uhlenbeck process is a continuous  mapping 
$\boldsymbol{z}_{\chi}: C^{\xi}_{1/2} (\mathbb{R};\mathrm{E})\to C_{1/2}(\mathbb{R};\mathfrak{X}).$
Therefore, taking into account the continuity result established in  Lemma \ref{RDS_Conti1}, we infer that the  solution  $\y^{\chi}$ of the   system \eqref{CTGF} is a continuous function of the sample point $\omega.$
\end{remark}

\begin{proposition}\label{alpha_ind}
	If $\chi_1, \chi_2 \geq 0$, then $\Psi_{\chi_1} = \Psi_{\chi_2}.$
\end{proposition}
\begin{proof}
	Let us fix $\boldsymbol{x}\in \H.$ We need to prove that $$\y^{\chi_1}(t) + \z_{\chi_1}(t) = \y^{\chi_2}(t) + \z_{\chi_2}(t), \ \ \ t\geq 0, $$ where $\z_{\chi}$ is defined by \eqref{DOu1} and $\y^{\chi}$ is a solution to   system \eqref{CTGF}. 
	
	Since $\z_{\chi_1}, \z_{\chi_2} \in \mathrm{C}([0,T];\H\cap \mathbb{L}^4(\mathcal{O}))$ and  $\z_{\chi_2}-\z_{\chi_1}$  satisfies \eqref{Dif_z1}, it implies from  \cite[Vol. I\!I, Theorem 3.2, p.22]{Lion-Mag} that  
	\begin{align}\label{359}
		\z_{\chi_2}-\z_{\chi_1}\in \mathrm{C}([0,T];\H\cap\mathbb{W}^{1,4}(\mathcal{O}))\cap \mathrm{L}^2(0,T;\V)
	\end{align}
with
\begin{align}\label{359-t}
	 \partial_{t}(\z_{\chi_2}-\z_{\chi_1}) \in \mathrm{L}^2(0,T,\V^\prime).
\end{align}

	From \eqref{CTGF}, we infer that $\y^{\chi_1}(0) - \y^{\chi_2}(0) = - (\z_{\chi_1}(0) - \z_{\chi_2}(0)) $ and 
	\begin{align*}
	&	\frac{\d(\y^{\chi_1}(t) - \y^{\chi_2}(t))}{\d t} 
	\nonumber\\ &= -\nu \A(\y^{\chi_1}(t) - \y^{\chi_2}(t)) + \big[\chi_1 \z_{\chi_1}(t) -  \chi_2 \z_{\chi_2}(t)\big]  - [\B(\y^{\chi_1}(t) + \z_{\chi_1}(t))-\B(\y^{\chi_2}(t) + \z_{\chi_2}(t))]
		\nonumber\\ & \quad - \alpha[\J(\y^{\chi_1}(t) + \z_{\chi_1}(t))-\J(\y^{\chi_2}(t) + \z_{\chi_2}(t))]
		 - \beta[\K(\y^{\chi_1}(t) + \z_{\chi_1}(t))-\K(\y^{\chi_2}(t) + \z_{\chi_2}(t))], 
	\end{align*}
	in $\X'$. It implies from Theorem \ref{solution} that 
	\begin{align}\label{360}
		\y^{\chi_1}-\y^{\chi_2}\in \mathrm{C}([0,T];\H)\cap \mathrm{L}^2(0,T;\V)\cap \mathrm{L}^4(0,T;\mathbb{W}^{1,4}(\mathcal{O})) 
	\end{align}
with
\begin{align}\label{360-t}
 \partial_{t}(\y^{\chi_1}-\y^{\chi_2}) \in \mathrm{L}^2(0,T,\V^\prime)+ \mathrm{L}^{\frac43}(0,T;\mathbb{W}^{-1,\frac43}(\mathcal{O})).
\end{align}
	Adding the equation \eqref{Dif_z1} to the above equation, we obtain 
	\begin{align}\label{Enrgy}
		\frac{\d(\v^{\chi_1}(t) - \v^{\chi_2}(t))}{\d t}& = -\nu \A(\v^{\chi_1}(t) - \v^{\chi_2}(t))  - [\B(\v^{\chi_1}(t))-\B(\v^{\chi_2}(t))]
		\nonumber\\ & \quad - \alpha[\J(\v^{\chi_1}(t))-\J(\v^{\chi_2}(t))]
		- \beta[\K(\v^{\chi_1}(t))-\K(\v^{\chi_2}(t))], 
	\end{align}
	in $\X'$,	where $\v^{\chi_1}(t)=\y^{\chi_1}(t) + \z_{\chi_1}(t),   \v^{\chi_2}(t) = \y^{\chi_2}(t) + \z_{\chi_2}(t), \  t\geq 0$ and $\v^{\chi_1}(0) - \v^{\chi_2}(0)= \boldsymbol{0}.$ In view of \eqref{359}-\eqref{359-t} and \eqref{360}-\eqref{360-t}, we infer that 
	\begin{align*}
		\v^{\chi_1}-\v^{\chi_2}\in \mathrm{C}([0,T];\H)\cap \mathrm{L}^2(0,T;\V)\cap \mathrm{L}^4(0,T;\mathbb{W}^{1,4}(\mathcal{O})) 
	\end{align*}
	with
	\begin{align*}
		\partial_{t}(\v^{\chi_1}-\v^{\chi_2}) \in \mathrm{L}^2(0,T,\V^\prime)+ \mathrm{L}^{\frac43}(0,T;\mathbb{W}^{-1,\frac43}(\mathcal{O})).
	\end{align*}
	Therefore, from  \cite[Chapter II, Theorem 1.8]{Chepyzhov+Vishik_2002}, we have 
	\begin{align*}
		& \frac12 \frac{\d}{\d t} \|\v^{\chi_1}(t) - \v^{\chi_2}(t)\|^2_2
		\nonumber\\ & = -\big\langle \nu \A(\v^{\chi_1}(t) - \v^{\chi_2}(t))  + [\B(\v^{\chi_1}(t) )-\B(\v^{\chi_2}(t) )]
		 - \alpha[\J(\v^{\chi_1}(t) )-\J(\v^{\chi_2}(t) )]
		\nonumber\\ & \qquad\qquad\qquad - \beta[\K(\v^{\chi_1}(t) )-\K(\v^{\chi_2}(t))],  \v^{\chi_1}(t) - \v^{\chi_2}(t) \big\rangle,
	\end{align*}
	for a.e. $t\in[0,T]$. Applying similar steps as we have performed for \eqref{CL1}, we arrive at
	\begin{align}\label{eqn_v_1-v_2}
		\frac{\d}{\d t} \|\v^{\chi_1}(t) - \v^{\chi_2}(t)\|^2_2\leq  \frac{2(C_{S,3} C_K)^2}{\nu\varepsilon_0} \|\Arm(\v^{\chi_2}(t))\|_{{4}}^2 \|\v^{\chi_1}(t) - \v^{\chi_2}(t)\|^2_\H,
	\end{align}
	for a.e. $t\in[0,T]$. 	Since   $\|\v^{\chi_1}(0) - \v^{\chi_2}(0)\|^2_2 = 0$ and $\v^{\chi_2}\in \mathrm{L}^4(0,T;\mathbb{W}^{1,4}(\mathcal{O})) $, by applying Gronwall inequality, we deduce  that $\|\v^{\chi_1}(t) - \v^{\chi_2}(t)\|^2_2 = 0$, for all $t\geq 0$, which completes the proof.
\end{proof}
It is proved in Proposition \ref{alpha_ind} that the map $\Psi_{\chi}$ does not depend on $\chi$ and hence, from now onward, it will be denoted by $\Psi$. A proof of the following result is similar to that in \cite[Theorem 6.15]{Brzezniak+Li_2006} which is based on the uniqueness of the solutions to   system \eqref{CTGF}. Hence we omit it here.

\begin{theorem}
	$(\Psi, \theta)$ is an RDS.
\end{theorem}

Now we are ready to present the solution to system \eqref{STGF} with the initial data $\x\in\H$ at the initial time $s\in \R.$
The following theorem is a consequence of our previous discussion.
 \begin{theorem}\label{STGF-Sol}
	In the framework of Definition \ref{Def_u}, suppose that $\v(t)=\y^{\chi}(t)+\z_{\chi}(t), t\geq s,$ where $\y^{\chi}$ is the unique solution to system \eqref{CTGF} with initial data $\x - \z_{\chi}(s)$ at time $s$. If the process $\{\v(t), \ t\geq s\},$ has trajectories in $\mathrm{C}([s, \infty); \H) \cap  \mathrm{L}^{4}_{\mathrm{loc}}([s, \infty); \mathbb{W}^{1,4}(\mathcal{O}))$, then it is a solution to system \eqref{STGF}. Vice-versa, if a process $\{\v(t), t\geq s\},$ with trajectories in $\mathrm{C}([s, \infty); \H) \cap \mathrm{L}^{4}_{\mathrm{loc}}([s, \infty); \mathbb{W}^{1,4}(\mathcal{O}))$ is a solution to system \eqref{STGF}, then for any $\chi\geq 0,$ a process $\{\y^{\chi}(t), t\geq s\},$ defined by $\y^{\chi}(t) = \v(t)- \z_{\chi}(t), t\geq s,$ is a solution to \eqref{CTGF} on $[s, \infty).$
\end{theorem}

	\section{Random attractors for a subclass of stochastic third-grade fluids equations: Bounded domains}\label{sec4}\setcounter{equation}{0}
In this section, we prove the existence of unique random attractor for system \ref{STGF} on bounded domains. In order to prove the result of this section, we assume that external forcing term $\f\in \H^{-1}(\mathcal{O})$.  Here, the RDS $\Psi$ is considered over the MDS $(\Omega,  {\mathcal{F}},  {\mathbb{P}},  {\theta})$.	The results that we have obtained in the previous sections  provide a unique solution to   system \eqref{STGF}, which is continuous with respect to the data (particularly $\x$ and $\f$). Furthermore, if we define, for $\x \in \H,\ \omega \in \Omega,$ and $t\geq s,$
\begin{align}\label{combine_sol}
	\v(t, \omega ; s, \x) := \Psi(t-s, \theta_s \omega,\x) = \y\big(t, \omega; s, \x - \z(s)\big) + \z(t),
\end{align}
then the process $\{\v(t): \ t\geq s\}$ is a solution to  equation $\eqref{STGF}_{1}$ for $t>s$, for each $s\in \mathbb{R}$ and each $\x \in \H$.

The following lemma helps us to prove the existence of  random $\mathfrak{DK}$-absorbing set.

	\begin{lemma}\label{RA1}
	Let \eqref{condition1} and Hypotheses \ref{assump1}-\ref{assumpO} be satisfied. Suppose that $\y$ solves   system \eqref{CTGF} on the time interval $[a, \infty)$ with $\z \in  \mathrm{L}^2_{\mathrm{loc}}(\R; \H) \cap  \mathrm{L}^{4}_{\mathrm{loc}}(\R; \mathbb{W}^{1,4}(\mathcal{O}))$, $\chi\geq 0$ and $\f\in \H^{-1}(\mathcal{O})$. Then, for any $t\geq \tau \geq a,$
	\begin{align}\label{Energy_esti1}
		 \|\y(t)\|^2_{2} 
		  & \leq 
		\|\y(\tau)\|^2_{2}\  e^{-\nu\lambda\left(1+ \frac{\varepsilon_0}{2}\right)(t-\tau)}  
		  + C\int_{\tau}^{t} \bigg[\|\z(s)\|^{2}_{2}+ \|\z(s)\|^{4}_{\mathbb{W}^{1,4}} + \|\f\|^2_{\H^{-1}} \bigg] e^{-\nu\lambda\left(1+ \frac{\varepsilon_0}{2}\right)(t-s)} \d s.
	\end{align}
\end{lemma}
\begin{proof}
	From \eqref{CTGF}, we obtain 
	\begin{align}\label{Energy_esti3}
	&	\frac{1}{2}\frac{\d}{\d t} \|\y(t)\|^2_{2} 
	\nonumber\\ 	&=  - \frac{\nu}{2} \|\Arm(\y(t))\|^2_{2}   -\left<\B(\y(t)+ \z(t))+\alpha\J(\y(t)+ \z(t))+\beta \K(\y(t)+ \z(t)),\y(t)\right> 
		 +\chi (\z(t),\y(t)) \nonumber\\ & \quad +\big\langle \f, \y(t)\big \rangle \nonumber \\
		&=  - \frac{\nu}{2} \|\Arm(\y(t))\|^2_{2} - \frac{\beta}{2} \|\Arm(\y(t)+\z(t))\|^4_{4}    + b(\y(t)+ \z(t),\y(t),\z(t)) - \alpha \left<\J(\y(t)+ \z(t)) ,\y(t)\right> 
		\nonumber\\ & \quad + \beta \left< \K(\y(t)+ \z(t)),\z(t)\right>  +\chi (\z(t),\y(t))+\big\langle \f, \y(t)\big \rangle,
	\end{align}
	for a.e. $t\in[a,\infty)$. Next, using the H\"older's inequality, Lemma \ref{Sobolev-embedding3}, Korn-type inequality \eqref{Korn-ineq}, Poincar\'e inequality \eqref{poin}, \eqref{condition1} (see Remark \ref{rem-condition1} also), Gagliardo-Nirenberg inequality \eqref{Gen_lady-4} and Young's inequality, we estimate the  terms of the right hand side of \eqref{S1} as (similar to \eqref{Sb}-\eqref{S2}) 
	\begin{align}
		|b(\y+\z, \y, \z)| &  \leq \frac{\nu\varepsilon_0}{24} \|\Arm(\y)\|^2_{2} + \frac{\beta\varepsilon_0}{12}  \|\Arm(\y+\z)\|^4_{4}  + C \|\z\|_4^4,\label{RAS1}\\
		\left\vert\alpha \left<\J(\y+ \z) ,\y\right>  \right\vert
		&  \leq  \frac{\nu(1-\varepsilon_0)}{4} \|\Arm(\y)\|^2_{2}  + \frac{\beta(1-\varepsilon_0)}{2}  \|\Arm(\y+\z)\|^4_{4}  \label{RAS2} \\
		\left| \beta \left< \K(\y+ \z),\z\right> \right|  & \leq 
		     \frac{\beta\varepsilon_0}{12}  \|\Arm(\y+\z)\|^4_{4}   + C \|\Arm(\z)\|^4_{4}, \label{RAS3} \\
		| \chi (\z,\y)|& \leq  \frac{\nu\varepsilon_0}{24} \|\Arm(\y)\|^2_{2} + C \|\z\|^2_2, \label{RAS4}\\
				\big|\langle\f,\y\rangle\big| 	& \leq   \frac{\nu\varepsilon_0}{24} \|\Arm(\y)\|^2_{2}  	+  C \|\f\|^2_{\H^{-1}}.		\label{RAS5}
	\end{align}
	Hence, from \eqref{Energy_esti3} and \eqref{poin}, we deduce 
	\begin{align*}
		\frac{\d}{\d t} \|\y(t)\|^2_{2} 
	  & \leq  - \frac{\nu}{2}\left(1+ \frac{\varepsilon_0}{2}\right) \|\Arm(\y(t))\|^2_{2} - \frac{\beta\varepsilon_0}{2} \|\Arm(\y(t)+\z(t))\|^4_{4}  
	   +  C \bigg[\|\f\|^2_{\H^{-1}}  + \|\z(t)\|^2_{2}+ \|\z(t)\|^4_{\mathbb{W}^{1,4}}  \bigg]
		\nonumber\\ & =  - \nu \left(1+ \frac{\varepsilon_0}{2}\right) \|\nabla\y(t)\|^2_{2} - \frac{\beta\varepsilon_0}{2} \|\Arm(\y(t)+\z(t))\|^4_{4}  
		  +  C \bigg[\|\f\|^2_{\H^{-1}} + \|\z(t)\|^2_{2}+ \|\z(t)\|^4_{\mathbb{W}^{1,4}}  \bigg]
		\nonumber\\ & \leq  - \nu\lambda\left(1+ \frac{\varepsilon_0}{2}\right) \|\y(t)\|^2_{2} - \frac{\beta\varepsilon_0}{2} \|\Arm(\y(t)+\z(t))\|^4_{4}  
		  +  C \bigg[\|\f\|^2_{\H^{-1}}  + \|\z(t)\|^2_{2}+ \|\z(t)\|^4_{\mathbb{W}^{1,4}}  \bigg],
	\end{align*}
for a.e. $t\in[a,\infty)$,	and an application of  the variation of constants formula yields \eqref{Energy_esti1}.
\end{proof}

 In the following lemma, we show that the RDS $\Psi$ is compact operator from $\H$ into itself. 
\begin{lemma}\label{PCB}
	Let \eqref{condition1} and  Hypotheses \ref{assump1}-\ref{assumpO} be satisfied. Then, the solution operator $\Psi$ given in Definition \eqref{combine_sol} is compact from $\H$ into itself.
\end{lemma}
\begin{proof}
	Consider the solution $\Psi(t-s,\theta_s\omega,\cdot)=\y\big(t, \omega; s, \cdot -  \z(s)\big) +  \z(t)$ of \eqref{STGF} for all $t> s$ and $s\in\R$. For simplicity, we write $\y\big(t, s; \omega, \cdot - \z(s)\big)=\y(t,s,\cdot)$. Assume that the sequence $\{\x_m\}_{m\in\N}\subset\H$ is bounded. We have that, for $\z\in\mathrm{C}([s,t];\H)\cap\mathrm{L}^{4}(s,t;\mathbb{W}^{1,4}(\O))$, the sequences
	\begin{equation*}
		\{\y(\cdot,s, \x_m)\}_{m\in\N}  \subset \mathrm{C}([s,t];\H)\cap\mathrm{L}^{2}(s,t;\V)\cap\mathrm{L}^{4}(s,t;\mathbb{W}^{1,4}(\O))
	\end{equation*}
	and 
	\begin{equation*}
		\left\{\frac{\d}{\d t}[\y(\cdot,s, \x_m)]\right\}_{m\in\N}  \subset \mathrm{L}^{2}(s,t;\V^{\prime})+\mathrm{L}^{\frac43}(s,t;\mathbb{W}^{-1,\frac43}(\O))
	\end{equation*}
	are bounded. Since $\mathrm{L}^2(s,t;\V')+\mathrm{L}^{\frac{4}{3}}(s,t;\mathbb{W}^{-1,\frac{4}{3}})\subset\mathrm{L}^{\frac{4}{3}}(s,t;\V'+\mathbb{W}^{-1,\frac{4}{3}})$, the above sequence is bounded in $\mathrm{L}^{\frac{4}{3}}(s,t;\V'+\mathbb{W}^{-1,\frac{4}{3}}(\O))$. Note also that $\V\subset\H\subset \V'+\mathbb{W}^{-1,\frac{4}{3}}$ and the embedding of $\V\subset\H$ is compact. By the \textit{Aubin-Lions compactness lemma}, there exists a subsequence (keeping as it is) and $\widetilde{\y}\in\mathrm{L}^2(s,t;\H)$ such that 
	\begin{align}\label{PCB5}
		\y(\cdot,s, \x_m)\to\widetilde{\y}(\cdot) \ \text{ strongly in } \ \mathrm{L}^{2}(s,t;\H),
	\end{align}
	as $m\to\infty$. Again, choosing one more subsequence (again not relabeling), we infer from \eqref{PCB5} that
	\begin{align}\label{PCB6}
		\y(\tau,s, \x_m)\to\widetilde{\y}(\tau) \text{ in }  \H, \ \text{ for a.e. } \ \tau\in [s,t],
	\end{align}
	as $m\to\infty$. For $s < t$, we obtain from \eqref{PCB6} that there exists $\tau\in(s,t)$ such that \eqref{PCB6} holds true for this particular $\tau$. Then  by Lemma \ref{RDS_Conti1}, we obtain
	\begin{align*}
		\Psi(t-s,\theta_{s}\omega,\x_m)&=\Psi(t-\tau,\theta_{\tau}\omega, \Psi(\tau-s,\theta_{s}\omega,\x_m))\nonumber\\&=\Psi\big(t-\tau,\theta_{\tau}\omega,\y\big(\tau, \omega; s, \x_m - \z(s)\big)+\z(\tau)\big)\nonumber\\&\to \Psi\big(t-\tau,\theta_{\tau}\omega,\widetilde{\y}(\tau)+\z(\tau)\big)
	\end{align*}
as $m\to\infty$, in $\H$, which completes the proof.
\end{proof}

Let us now provide the main result of this section.

	\begin{theorem}\label{Main_theorem_1}
	Let \eqref{condition1} and  Hypotheses \ref{assump1}-\ref{assumpO} be satisfied. Consider the MDS, $\Im = (\Omega, \mathcal{F}, \mathbb{P}, \theta)$ from Proposition \ref{m-DS1}, and the RDS $\Psi$ on $\H$ over $\Im$ generated by the SGMNS equations \eqref{STGF} with additive noise satisfying Hypothesis \ref{assump1}. Then, there exists a unique random $\mathfrak{DK}$-attractor for continuous RDS $\Psi$ in $\H$.
\end{theorem}
\begin{proof} 
	Because of \cite[Theorem 2.8]{BCLLLR}, it is only needed to prove that there exists a $\mathfrak{DK}$-absorbing set $\textbf{B}\in \mathfrak{DK}$ and the RDS $\Psi$ is $\mathfrak{DK}$-asymptotically compact. 	
	\vskip 0.2 cm 
	\noindent
	\textbf{Existence of $\mathfrak{DK}$-absorbing set $\textbf{B}\in \mathfrak{DK}$:}	Let $\mathrm{D}$ be a random set from the class $\mathfrak{DK}$. Let $\kappa_{\mathrm{D}}(\omega)$ be the radius of $\mathrm{D}(\omega)$, that is, $\kappa_{\mathrm{D}}(\omega):= \sup\{\|x\|_{2} : x \in \mathrm{D}(\omega)\}$ for $\omega\in \Omega$.
	
	Let $\omega\in \Omega$ be fixed. For a given $s\leq 0$ and $\boldsymbol{x}\in \H$, let $\y$ be the solution of \eqref{CTGF} on the time interval $[s, \infty)$ with the initial condition $\y(s)= \boldsymbol{x}-\z(s).$ Using \eqref{Energy_esti1} for $t=\xi \text{ and } \tau=s\leq-1$ with $\xi\in[-1,0]$, we obtain 
	\begin{align}\label{Energy_esti5}
		\|\y(\xi)\|^2_{2} &\leq  2 \|\boldsymbol{x}\|^2_{2}\  e^{\nu\lambda\left(1+ \frac{\varepsilon_0}{2}\right) s}  + 2 \|\z(s)\|^2_{2}\  e^{\nu\lambda\left(1+ \frac{\varepsilon_0}{2}\right) s }  
		\nonumber\\ & \quad + C\int_{s}^{0} \bigg\{ \|\z(t)\|^{2}_{2}+ \|\z(t)\|^{4}_{\mathbb{W}^{1,4}} + \|\f\|^2_{\H^{-1}} \bigg\}e^{\nu\lambda\left(1+ \frac{\varepsilon_0}{2}\right) t} \d t.
	\end{align}

	For $\omega\in \Omega,$ let us set
	\begin{align}
		[\kappa_{11}(\omega)]^2 &= 2 +  2\sup_{s\leq 0}\bigg\{ \|\z(s)\|^2_{2}\  e^{\nu\lambda\left(1+ \frac{\varepsilon_0}{2}\right) s}\bigg\} 
		\nonumber\\ & \quad 
		+C \int_{- \infty}^{0} \bigg\{ \|\z(t)\|^{2}_{2}+ \|\z(t)\|^{4}_{\mathbb{W}^{1,4}} + \|\f\|^2_{\H^{-1}} \bigg\}e^{\nu\lambda\left(1+ \frac{\varepsilon_0}{2}\right) t}  \d t,\\	
		\kappa_{12}(\omega) &=   \|\z(\omega)(0)\|_{2}.
	\end{align}

	In view of Lemma \ref{Bddns5} and Proposition \ref{radius}, we deduce that both $\kappa_{11},\kappa_{12}\in  \mathfrak{K}$ and also that $\kappa_{11}+\kappa_{12}=:\kappa_{13} \in \mathfrak{K}$ as well. Therefore the random set $\textbf{B}$ defined by $$\textbf{B}(\omega) := \{\v\in\H: \|\v\|_{2}\leq \kappa_{13}(\omega)\}$$ is such that $\textbf{B}\in\mathfrak{DK}.$ 
	
	Let us now show that $\textbf{B}$ absorbs $\mathrm{D}$. Let $\omega\in\Omega$ be fixed. Since $\kappa_{\mathrm{D}}(\omega)\in \mathfrak{K}$, there exists $t_{\mathrm{D}}(\omega)\geq 1$ such that 
	\begin{align*}
		[\kappa_{\mathrm{D}}(\theta_{-t}\omega)]^2 e^{-\nu\lambda\left(1+ \frac{\varepsilon_0}{2}\right) t } &\leq 1, \  \text{ for }\  t\geq t_{\mathrm{D}}(\omega).
	\end{align*}
	Thus, for $\omega\in\Omega$, if $\boldsymbol{x}\in \mathrm{D}(\theta_{s}\omega)$ and $s\leq- t_{\mathrm{D}}(\omega)\leq -1,$ then by \eqref{Energy_esti5}, we obtain  for $\xi\in[-1,0]$
	\begin{align}\label{AB1}
		\|\y(\xi,\omega; s, \boldsymbol{x}-\z(s))\|_{2}\leq \kappa_{11}(\omega), \  \text{ for } \ \omega\in \Omega.
	\end{align}
	Thus, we conclude that, for $\omega\in\Omega$  
	\begin{align*}
		\|\v(0,\omega; s, \boldsymbol{x})\|_{2} \leq \|\y(0,\omega; s, \boldsymbol{x}-\z(s))\|_{2} + \|\z(\omega)(0)\|_{2}\leq \kappa_{13}(\omega).
	\end{align*}
	The above inequality implies that for $\omega\in \Omega$, $\v(0,\omega; s, \boldsymbol{x}) \in \textbf{B}(\omega)$, for all $s\leq -t_{\mathrm{D}}(\omega).$ This proves  $\textbf{B}$ absorbs $\mathrm{D}$.	
	\vskip 0.2 cm 
	\noindent
	\textbf{The RDS $\Psi$ is $\mathfrak{DK}$-asymptotically compact.} 		
	In order to establish the $\mathfrak{DK}$-asymptotically compactness of $\Psi$, we use the compact Sobolev embedding $\V\subset\H$. Let us assume that $\mathrm{D} \in \mathfrak{DK}$ and $\textbf{B}\in \mathfrak{DK}$ be such that $\textbf{B}$ absorbs $\mathrm{D}$. Let us fix $\omega\in \Omega$ and take a sequence of positive numbers $\{t_m\}^{\infty}_{m=1}$ such that $t_1\leq t_2 \leq t_3 \leq \cdots$ and $t_m \to \infty$. We take an arbitrary $\H$-valued sequence $\{\boldsymbol{x}_m\}^{\infty}_{m=1}$ such that $\boldsymbol{x}_m \in \mathrm{D}(\theta_{-t_m}\omega),$ for all $m\in \mathbb{N}.$
	
	Since $t_m\to\infty$, there exists $M_1=M_1(\tau,\hat{K})\in\N$ such that $t_m\geq t_{\mathrm{D}}(\omega)$ for all $m\geq M_1.$ Therefore, in view of \eqref{AB1}, the sequence
	\begin{align}\label{619}
		 \{\Psi(t_m-1, \theta_{-t_m}\omega, \boldsymbol{x}_m)\}_{m\geq M_1}=\{\v(-1, \omega; -t_m, \boldsymbol{x}_m)\}_{m\geq M_1} \subset \H,
	\end{align} 
is bounded. Combining the fact of \eqref{619} and Lemma \ref{PCB}, we conclude that the sequence
		\begin{align}\label{621}
		\{\Psi(t_m, \theta_{-t_m}\omega, \boldsymbol{x}_m)\}_{m\geq M_1}= \{\Psi(1, \theta_{-1}\omega, \Psi(t_m-1, \theta_{-t_m}\omega, \boldsymbol{x}_m))\}_{m\geq M_1},
	\end{align} 
has a convergent subsequence. Hence, The RDS $\Psi$ is $\mathfrak{DK}$-asymptotically compact. This completes the proof.
\end{proof}

	\section{Random attractors for a subclass of stochastic third-grade fluids equations: Unbounded domains}\label{sec5}\setcounter{equation}{0}

In this section, we prove the existence of unique random attractor for system \ref{STGF} on unbounded domains. In order to prove the result of this section, we assume that external forcing term $\f\in \L^2(\O)$. The following two lemmas help us to prove the existence of random $\mathfrak{DK}$-absorbing set.

\begin{lemma}\label{RA1-ubd}
	Let \eqref{condition1} and   Hypotheses \ref{assump1}-\ref{assumpO} be satisfied. Suppose that $\y$ solves   system \eqref{CTGF} on the time interval $[a, \infty)$ with $\z \in  \mathrm{L}^2_{\mathrm{loc}}(\R; \H) \cap  \mathrm{L}^{4}_{\mathrm{loc}}(\R; \mathbb{W}^{1,4}(\mathcal{O}))$, $\chi\geq 0$ and $\f\in \L^2(\O)$. Then, for any $ t\geq s \geq a,$
	\begin{align}\label{Energy_esti1-ubd}
		&	\|\y(t)\|^2_{2} + \frac{\beta\varepsilon_0}{2} \int_{s}^{t} \|\Arm(\y(\tau)+\z(\tau))\|^{4}_{4} e^{-\nu\lambda\left(1+ \frac{\varepsilon_0}{2}\right)(t-\tau)} \d \tau
		\nonumber\\ &   \leq 
		\|\y(s)\|^2_{2}\  e^{-\nu\lambda\left(1+ \frac{\varepsilon_0}{2}\right)(t-s)}  
		+ C\int_{s}^{t} \bigg[\|\z(\tau)\|^{2}_{2}+ \|\z(\tau)\|^{4}_{\mathbb{W}^{1,4}} + \|\f\|^2_{2}\bigg] e^{-\nu\lambda\left(1+ \frac{\varepsilon_0}{2}\right)(t-\tau)} \d \tau.
	\end{align}
\end{lemma}
\begin{proof}
	From \eqref{CTGF}, we obtain 
	\begin{align}\label{Energy_esti3-ubd}
		\frac{1}{2}\frac{\d}{\d t} \|\y(t)\|^2_{2} 
		&=  - \frac{\nu}{2} \|\Arm(\y(t))\|^2_{2} - \frac{\beta}{2} \|\Arm(\y(t)+\z(t))\|^4_{4}    + b(\y(t)+ \z(t),\y(t),\z(t)) \nonumber\\ & \quad - \alpha \left<\J(\y(t)+ \z(t)) ,\y(t)\right> 
		+ \beta \left< \K(\y(t)+ \z(t)),\z(t)\right>  + (\chi \z(t) + \f,\y(t)),
	\end{align}
	for a.e. $t\in[a,\infty)$. Next, using the H\"older's inequality, Lemma \ref{Sobolev-embedding3}, Korn-type inequality \eqref{Korn-ineq}, Poincar\'e inequality \eqref{poin}, \eqref{condition1} (see Remark \ref{rem-condition1} also), Gagliardo-Nirenberg inequality \eqref{Gen_lady-4} and Young's inequality, we estimate the  terms of the right hand side of \eqref{S1} as (see the proof of Lemma \ref{RA1})
	\begin{align}
		|b(\y+\z, \y, \z)| & \leq   \frac{\nu\varepsilon_0}{16} \|\Arm(\y)\|^2_{2} + \frac{\beta\varepsilon_0}{8}  \|\Arm(\y+\z)\|^4_{4}  + C \|\z\|_4^4,\label{RAS1-ubd}\\
		\left\vert\alpha \left<\J(\y+ \z) ,\y\right>  \right\vert
		&     \leq  \frac{\nu(1-\varepsilon_0)}{4} \|\Arm(\y)\|^2_{2}  + \frac{\beta(1-\varepsilon_0)}{2}  \|\Arm(\y+\z)\|^4_{4}  \label{RAS2-ubd} \\
		\left| \beta \left< \K(\y+ \z),\z\right> \right|  &   \leq  \frac{\beta\varepsilon_0}{8}  \|\Arm(\y+\z)\|^4_{4}   + C \|\Arm(\z)\|^4_{4}, \label{RAS3-ubd} \\
		|  (\chi\z+\f,\y)|&   \leq \frac{\nu\varepsilon_0}{16} \|\Arm(\y)\|^2_{2} + C \|\z\|^2_2+ C \|\f\|^2_2. \label{RAS4-ubd}
	\end{align}
	Hence, from \eqref{Energy_esti3-ubd} and \eqref{poin}, we deduce 
	\begin{align}
		&	\frac{\d}{\d t} \|\y(t)\|^2_{2} 
		\nonumber\\ 	& \leq  - \frac{\nu}{2}\left(1+ \frac{\varepsilon_0}{2}\right) \|\Arm(\y(t))\|^2_{2} - \frac{\beta\varepsilon_0}{2} \|\Arm(\y(t)+\z(t))\|^4_{4}  
		+  C \bigg[\|\f\|^2_{2}  + \|\z(t)\|^2_{2}+ \|\z(t)\|^4_{\mathbb{W}^{1,4}}  \bigg]  \label{RAS5-ubd}\\
		\nonumber & \leq  - \nu\lambda\left(1+ \frac{\varepsilon_0}{2}\right) \|\y(t)\|^2_{2} - \frac{\beta\varepsilon_0}{2} \|\Arm(\y(t)+\z(t))\|^4_{4}  
		+  C \bigg[\|\f\|^2_{2}  + \|\z(t)\|^2_{2}+ \|\z(t)\|^4_{\mathbb{W}^{1,4}}  \bigg],
	\end{align}
	for a.e. $t\in[a,\infty)$,	and an application of  the variation of constants formula yields \eqref{Energy_esti1-ubd}. This completes the proof.
\end{proof}

\begin{lemma}\label{D-abH-ubd}
	Let \eqref{condition1} and Hypotheses \ref{assump1}-\ref{assumpO} be satisfied, and $\f\in\L^2(\O)$. Then, for every $\omega\in\Omega$ and $\D \in\mathfrak{DK}$, and for all $\xi\in[-1,0]$, there exists $t_{\D}(\omega)\geq1$ and a constant $C_{ubd}>0$ such that for all $s\leq -t_{\D}(\omega),$
	\begin{align}\label{S1-ubd}
		&	e^{\nu\lambda\left(1+ \frac{\varepsilon_0}{2}\right) \xi}	 \|\y(\xi,\omega;s,\x-\z(s))\|^2_{2} 
		+ \frac{\beta\varepsilon_0}{2} \int_{s}^{\xi} e^{\nu\lambda\left(1+ \frac{\varepsilon_0}{2}\right)\tau}\|\Arm(\y(\tau,\omega;s,\x-\z(s))) + \z(\tau) \|^4_{4} \d \tau  
		\nonumber\\ & 
		\leq  2+2\sup_{t\leq 0}\bigg\{ \|\z(t)\|^2_{2}\  e^{\nu\lambda\left(1+ \frac{\varepsilon_0}{2}\right) t}\bigg\} 
		+C_{ubd} \int_{- \infty}^{0} \bigg\{ \|\z(\tau)\|^{2}_{2}+ \|\z(\tau)\|^{4}_{\mathbb{W}^{1,4}} + \|\f\|^2_{2}\bigg\}e^{\nu\lambda\left(1+ \frac{\varepsilon_0}{2}\right) \tau}  \d \tau,
	\end{align}
	for any $\x\in \D(\theta_s\omega)$. In addition, there exist $t_{\D}^{\ast}(\omega)$ and $t_{\D}^{\ast\ast}(\omega)$ such that 
	\begin{align}\label{S2-ubd}
		\int_{s}^{0}e^{\nu\lambda\left(1+ \frac{\varepsilon_0}{2}\right)\tau} \|\Arm(\y(\tau,\omega;s,\x-\z(s)))\|^{2}_{2}  \d \tau \leq C, \text{	for all } \; s\leq -t_{\D}^{\ast}(\omega),
	\end{align}
	for any $\x\in \D(\theta_s\omega)$, and 
	\begin{align}\label{S3-ubd}
		\int_{s}^{0}e^{\nu\lambda\left(1+ \frac{\varepsilon_0}{2}\right)\tau} \|\y(\tau,\omega;s,\x-\z(s))\|^{4}_{2}  \d \tau \leq C, \text{	for all } \; s\leq -t_{\D}^{\ast\ast}(\omega),
	\end{align}
	for any $\x\in \D(\theta_s\omega)$.
\end{lemma}
\begin{proof}
	Let $\mathrm{D}$ be a random set from the class $\mathfrak{DK}$. Let $\kappa_{\mathrm{D}}(\omega)$ be the radius of $\mathrm{D}(\omega)$, that is, $\kappa_{\mathrm{D}}(\omega):= \sup\{\|x\|_{2} : x \in \mathrm{D}(\omega)\}$ for $\omega\in \Omega$. Let $\omega\in\Omega$ be fixed. Since $\kappa_{\mathrm{D}}(\omega)\in \mathfrak{K}$, there exists $t_{\mathrm{D}}(\omega)\geq 1$ such that 
	\begin{align}\label{UTE-TRV}
		& [\kappa_{\mathrm{D}}(\theta_{-t}\omega)]^2 e^{-\frac{\nu\lambda}{4}\left(1+ \frac{\varepsilon_0}{2}\right) t } \leq 1, 
	\end{align}
	for $t\geq t_{\mathrm{D}}(\omega).$	 From \eqref{Energy_esti1-ubd}, for $t=\xi\in[-1,0]$ and $s\leq - t_{\D}(\omega)$, we obtain
	\begin{align}\label{UTE1}
		& e^{\nu\lambda\left(1+ \frac{\varepsilon_0}{2}\right) \xi}	\|\y(\xi,\omega;s,\x-\z(s))\|^2_{2} 
		+ \frac{\beta\varepsilon_0}{2} \int_{s}^{\xi} e^{\nu\lambda\left(1+ \frac{\varepsilon_0}{2}\right)\tau}\|\Arm(\y(\tau,\omega;s,\x-\z(s))) + \z(\tau) \|^4_{4} \d \tau  
		\nonumber\\ & \leq \|\x-\z(s)\|^2_{2}\  e^{\nu\lambda\left(1+ \frac{\varepsilon_0}{2}\right)s}  
		+ C\int_{s}^{\xi} \bigg[\|\z(\tau)\|^{2}_{2}+ \|\z(\tau)\|^{4}_{\mathbb{W}^{1,4}} + \|\f\|^2_{2}\bigg] e^{\nu\lambda\left(1+ \frac{\varepsilon_0}{2}\right)\tau} \d \tau
		\nonumber\\ & 
		\leq 2\|\x\|^2_{2} e^{\nu\lambda\left(1+ \frac{\varepsilon_0}{2}\right) s}   +  2 \|\z(s)\|^2_{2}e^{\nu\lambda\left(1+ \frac{\varepsilon_0}{2}\right) s} 
		\nonumber\\ & \quad 
		+C_{ubd} \int_{- \infty}^{0} \bigg\{ \|\z(\tau)\|^{2}_{2}+ \|\z(\tau)\|^{4}_{\mathbb{W}^{1,4}} + \|\f\|^2_{2}\bigg\}e^{\nu\lambda\left(1+ \frac{\varepsilon_0}{2}\right) \tau}  \d \tau
		\nonumber\\ &  \leq
		2+ 2 \sup_{s\leq 0}\bigg\{ \|\z(s)\|^2_{2}\  e^{\nu\lambda\left(1+ \frac{\varepsilon_0}{2}\right) s}\bigg\} 
		+C_{ubd} \int_{- \infty}^{0} \bigg\{ \|\z(\tau)\|^{2}_{2}+ \|\z(\tau)\|^{4}_{\mathbb{W}^{1,4}} + \|\f\|^2_{2}\bigg\}e^{\nu\lambda\left(1+ \frac{\varepsilon_0}{2}\right) \tau}  \d \tau,
	\end{align}
	where we have used \eqref{UTE-TRV}. Note that, in view of Lemmas \ref{Bddns4} and \ref{Bddns5}, the right hand side of \eqref{UTE1} is finite. This completes the proof of \eqref{S1-ubd}.
	
	Form \eqref{RAS5-ubd}, we have 
	\begin{align*}
		&	\frac{\d}{\d t} \|\y(t)\|^2_{2} + \nu\lambda\left(1+ \frac{\varepsilon_0}{2}\right) \|\y(t)\|^2_{2}+ \frac{\nu}{2}\left(1+ \frac{\varepsilon_0}{2}\right) \|\Arm(\y(t))\|^2_{2}
		\nonumber\\ & \leq  C \bigg[\|\f\|^2_{2}  + \|\z(t)\|^2_{2}+ \|\z(t)\|^4_{\mathbb{W}^{1,4}}  \bigg]  + \nu\lambda\left(1+ \frac{\varepsilon_0}{2}\right) \|\y(t)\|^2_{2},
	\end{align*}
	which implies
	\begin{align}\label{UTE2}
		& \frac{\nu}{2}\left(1+ \frac{\varepsilon_0}{2}\right) \int_{s}^{t} \|\Arm(\y(\tau))\|^{2}_{2} e^{-\nu\lambda\left(1+ \frac{\varepsilon_0}{2}\right)(t-\tau)} \d \tau
		\nonumber\\ &   \leq 
		\|\y(s)\|^2_{2}\  e^{-\nu\lambda\left(1+ \frac{\varepsilon_0}{2}\right)(t-s)}  
		+ C\int_{s}^{t} \bigg[\|\z(\tau)\|^{2}_{2}+ \|\z(\tau)\|^{4}_{\mathbb{W}^{1,4}} + \|\f\|^2_{2}\bigg] e^{-\nu\lambda\left(1+ \frac{\varepsilon_0}{2}\right)(t-\tau)} \d \tau
		\nonumber\\ & \quad + \nu\lambda\left(1+ \frac{\varepsilon_0}{2}\right)\int_{s}^{t} \|\y(\tau)\|^{2}_{2} e^{-\nu\lambda\left(1+ \frac{\varepsilon_0}{2}\right)(t-\tau)} \d \tau
		\nonumber\\ & 	\leq 
		\|\y(s)\|^2_{2}\  e^{-\nu\lambda\left(1+ \frac{\varepsilon_0}{2}\right)(t-s)}  
		+ C\int_{s}^{t} \bigg[\|\z(\tau)\|^{2}_{2}+ \|\z(\tau)\|^{4}_{\mathbb{W}^{1,4}} + \|\f\|^2_{2}\bigg] e^{-\nu\lambda\left(1+ \frac{\varepsilon_0}{2}\right)(t-\tau)} \d \tau
		\nonumber\\ & \quad + \nu\lambda\left(1+ \frac{\varepsilon_0}{2}\right)\int_{s}^{t} \bigg[\|\y(s)\|^2_{2}\  e^{-\nu\lambda\left(1+ \frac{\varepsilon_0}{2}\right)(\tau-s)}  
		\nonumber\\ & \qquad + C\int_{s}^{\tau} \bigg\{\|\z(\zeta)\|^{2}_{2}+ \|\z(\zeta)\|^{4}_{\mathbb{W}^{1,4}} + \|\f\|^2_{2}\bigg\} e^{-\nu\lambda\left(1+ \frac{\varepsilon_0}{2}\right)(\tau-\zeta)} \d \zeta\bigg] e^{-\nu\lambda\left(1+ \frac{\varepsilon_0}{2}\right)(t-\tau)} \d \tau.
	\end{align}
	where we have used \eqref{Energy_esti1-ubd} in the final inequality.  Putting $t=0$, in \eqref{UTE2}, we get
	\begin{align}\label{UTE3}
		& \frac{\nu}{2}\left(1+ \frac{\varepsilon_0}{2}\right) \int_{s}^{0}e^{\nu\lambda\left(1+ \frac{\varepsilon_0}{2}\right)\tau} \|\Arm(\y(\tau,\omega;s,\x-\z(s)))\|^{2}_{2}  \d \tau
		\nonumber\\ &   \leq 
		\|\x-\z(s)\|^2_{2}\  e^{\nu\lambda\left(1+ \frac{\varepsilon_0}{2}\right)s}  
		+ C\int_{s}^{0} \bigg[\|\z(\tau)\|^{2}_{2}+ \|\z(\tau)\|^{4}_{\mathbb{W}^{1,4}} + \|\f\|^2_{2}\bigg] e^{\nu\lambda\left(1+ \frac{\varepsilon_0}{2}\right)\tau} \d \tau
		\nonumber\\ & \quad + \nu\lambda\left(1+ \frac{\varepsilon_0}{2}\right)\int_{s}^{0} \bigg[\|\x-\z(s)\|^2_{2}\  e^{-\nu\lambda\left(1+ \frac{\varepsilon_0}{2}\right)(\tau-s)}  
		\nonumber\\ & \qquad + C\int_{s}^{\tau} \bigg\{\|\z(\zeta)\|^{2}_{2}+ \|\z(\zeta)\|^{4}_{\mathbb{W}^{1,4}} + \|\f\|^2_{2}\bigg\} e^{-\nu\lambda\left(1+ \frac{\varepsilon_0}{2}\right)(\tau-\zeta)} \d \zeta\bigg] e^{\nu\lambda\left(1+ \frac{\varepsilon_0}{2}\right)\tau} \d \tau
		\nonumber\\ &   \leq 
		\|\x-\z(s)\|^2_{2}\  e^{\nu\lambda\left(1+ \frac{\varepsilon_0}{2}\right)s}  
		+ C\int_{s}^{0} \bigg[\|\z(\tau)\|^{2}_{2}+ \|\z(\tau)\|^{4}_{\mathbb{W}^{1,4}} + \|\f\|^2_{2}\bigg] e^{\nu\lambda\left(1+ \frac{\varepsilon_0}{2}\right)\tau} \d \tau
		\nonumber\\ & \quad + \nu\lambda\left(1+ \frac{\varepsilon_0}{2}\right)\int_{s}^{0} \bigg[\|\x-\z(s)\|^2_{2}\  e^{\nu\lambda\left(1+ \frac{\varepsilon_0}{2}\right)s}  
		\nonumber\\ & \qquad + C\int_{s}^{\tau} \bigg\{\|\z(\zeta)\|^{2}_{2}+ \|\z(\zeta)\|^{4}_{\mathbb{W}^{1,4}} + \|\f\|^2_{2}\bigg\} e^{\nu\lambda\left(1+ \frac{\varepsilon_0}{2}\right)\zeta} \d \zeta\bigg]  \d \tau
		\nonumber\\ &   \leq 
		\|\x-\z(s)\|^2_{2}\  e^{\nu\lambda\left(1+ \frac{\varepsilon_0}{2}\right)s}  
		+ C\int_{-\infty}^{0} \bigg\{\|\z(\tau)\|^{2}_{2}+ \|\z(\tau)\|^{4}_{\mathbb{W}^{1,4}} + \|\f\|^2_{2}\bigg\} e^{\nu\lambda\left(1+ \frac{\varepsilon_0}{2}\right)\tau} \d \tau
		\nonumber\\ & \quad - \nu\lambda\left(1+ \frac{\varepsilon_0}{2}\right)\|\x-\z(s)\|^2_{2}\  e^{\nu\lambda\left(1+ \frac{\varepsilon_0}{2}\right)s} s 
		\nonumber\\ & \qquad + C \int_{s}^{0} e^{\frac{\nu\lambda}{2}\left(1+ \frac{\varepsilon_0}{2}\right)\tau}\int_{s}^{\tau} \bigg\{\|\z(\zeta)\|^{2}_{2}+ \|\z(\zeta)\|^{4}_{\mathbb{W}^{1,4}} + \|\f\|^2_{2}\bigg\} e^{\frac{\nu\lambda}{2}\left(1+ \frac{\varepsilon_0}{2}\right)\zeta} \d \zeta \d \tau
		\nonumber\\ &   \leq 
		2\left[\|\x\|^2_{2}e^{\frac{\nu\lambda}{4}\left(1+ \frac{\varepsilon_0}{2}\right)s}  + \|\z(s)\|^2_{2}e^{\frac{\nu\lambda}{4}\left(1+ \frac{\varepsilon_0}{2}\right)s}  \right] (1-s) e^{\frac{3\nu\lambda}{4}\left(1+ \frac{\varepsilon_0}{2}\right)s}
		\nonumber\\ & \qquad +C  \int_{-\infty}^{0} \bigg\{\|\z(\zeta)\|^{2}_{2}+ \|\z(\zeta)\|^{4}_{\mathbb{W}^{1,4}} + \|\f\|^2_{2}\bigg\} e^{\frac{\nu\lambda}{2}\left(1+ \frac{\varepsilon_0}{2}\right)\zeta} \d \zeta.
	\end{align}
	Using \eqref{UTE-TRV}, Lemmas \ref{Bddns4} and \ref{Bddns5}, and the fact that $\lim\limits_{s\to - \infty}(1-s) e^{\frac{3\nu\lambda}{4}\left(1+ \frac{\varepsilon_0}{2}\right)s}=0,$ we have that there exists a $t^{\ast}_{\D}(\omega)\geq t_{\D}(\omega)$ such that for all $s\leq -t^{\ast}_{\D}(\omega) $
	\begin{align}\label{UTE4}
		& \frac{\nu}{2}\left(1+ \frac{\varepsilon_0}{2}\right) \int_{s}^{0}e^{\nu\lambda\left(1+ \frac{\varepsilon_0}{2}\right)\tau} \|\Arm(\y(\tau,\omega;s,\x-\z(s)))\|^{2}_{2}  \d \tau
		\nonumber\\ &   \leq 2+ 2 \sup_{s\leq 0}\bigg\{ \|\z(s)\|^2_{2}\  e^{\frac{\nu\lambda}{4}\left(1+ \frac{\varepsilon_0}{2}\right) s}\bigg\} 
		+C \int_{- \infty}^{0} \bigg\{ \|\z(\tau)\|^{2}_{2}+ \|\z(\tau)\|^{4}_{\mathbb{W}^{1,4}} + \|\f\|^2_{2}\bigg\}e^{\frac{\nu\lambda}{2}\left(1+ \frac{\varepsilon_0}{2}\right) \tau}  \d \tau.
	\end{align}
	Again, in view of Lemmas \ref{Bddns4} and \ref{Bddns5}, the right hand side of \eqref{UTE4} is finite. This completes the proof of \eqref{S2-ubd}.
	
	Finally, from \eqref{Energy_esti1-ubd}, we have
	\begin{align*}
		&  \int_{s}^{0} \|\y(\tau)\|^{4}_{2} e^{\nu\lambda\left(1+ \frac{\varepsilon_0}{2}\right)\tau} \d \tau
		\nonumber\\ & \leq \int_{s}^{0} \bigg[\|\y(s)\|^2_{2}\  e^{-\nu\lambda\left(1+ \frac{\varepsilon_0}{2}\right)(\tau-s)}  
		\nonumber\\ & \qquad	+ C\int_{s}^{\tau} \bigg\{\|\z(\zeta)\|^{2}_{2}+ \|\z(\zeta)\|^{4}_{\mathbb{W}^{1,4}} + \|\f\|^2_{2}\bigg\} e^{-\nu\lambda\left(1+ \frac{\varepsilon_0}{2}\right)(\tau-\zeta)} \d \zeta\bigg]^2 e^{\nu\lambda\left(1+ \frac{\varepsilon_0}{2}\right)\tau} \d \tau
		\nonumber\\ & \leq 2\int_{s}^{0} \|\y(s)\|^4_{2}\  e^{-2\nu\lambda\left(1+ \frac{\varepsilon_0}{2}\right)(\tau-s)} e^{\nu\lambda\left(1+ \frac{\varepsilon_0}{2}\right)\tau} \d \tau 
		\nonumber\\ & \qquad	+ 2C\int_{s}^{0} \bigg[\int_{s}^{\tau} \bigg\{\|\z(\zeta)\|^{2}_{2}+ \|\z(\zeta)\|^{4}_{\mathbb{W}^{1,4}} + \|\f\|^2_{2}\bigg\} e^{-\nu\lambda\left(1+ \frac{\varepsilon_0}{2}\right)(\tau-\zeta)} \d \zeta\bigg]^2 e^{\nu\lambda\left(1+ \frac{\varepsilon_0}{2}\right)\tau} \d \tau
		\nonumber\\ & \leq -2 s\|\y(s)\|^4_{2}\  e^{\nu\lambda\left(1+ \frac{\varepsilon_0}{2}\right)s}   
		\nonumber\\ & \qquad	+ 2C\int_{s}^{0} \bigg[\int_{s}^{\tau} \bigg\{\|\z(\zeta)\|^{2}_{2}+ \|\z(\zeta)\|^{4}_{\mathbb{W}^{1,4}} + \|\f\|^2_{2}\bigg\} e^{-\nu\lambda\left(1+ \frac{\varepsilon_0}{2}\right)(\tau-\zeta)} \d \zeta\bigg]^2 e^{\nu\lambda\left(1+ \frac{\varepsilon_0}{2}\right)\tau} \d \tau,
	\end{align*}
	which gives
	\begin{align*}
		&  \int_{s}^{0} \|\y(\tau,\omega;s,\x-\z(s))\|^{4}_{2} e^{\nu\lambda\left(1+ \frac{\varepsilon_0}{2}\right)\tau} \d \tau
		\nonumber\\ & \leq -2 s\|\x-\z(s)\|^4_{2}\  e^{\nu\lambda\left(1+ \frac{\varepsilon_0}{2}\right)s}   
		\nonumber\\ & \qquad	+ 2C\int_{s}^{0} \bigg[\int_{s}^{\tau} \bigg\{\|\z(\zeta)\|^{2}_{2}+ \|\z(\zeta)\|^{4}_{\mathbb{W}^{1,4}} + \|\f\|^2_{2}\bigg\} e^{-\nu\lambda\left(1+ \frac{\varepsilon_0}{2}\right)(\tau-\zeta)} \d \zeta\bigg]^2 e^{\nu\lambda\left(1+ \frac{\varepsilon_0}{2}\right)\tau} \d \tau
		\nonumber\\ & \leq -16 s\left\{\|\x\|^4_{2}  +  \|\z(s)\|^4_{2}\right\}\  e^{\nu\lambda\left(1+ \frac{\varepsilon_0}{2}\right)s}   
		\nonumber\\ & \qquad	+ 2C\int_{s}^{0} \bigg[\int_{s}^{\tau} \bigg\{\|\z(\zeta)\|^{2}_{2}+ \|\z(\zeta)\|^{4}_{\mathbb{W}^{1,4}} + \|\f\|^2_{2}\bigg\} e^{-\frac{\nu\lambda}{4}\left(1+ \frac{\varepsilon_0}{2}\right)(\tau-\zeta)} \d \zeta\bigg]^2 e^{\nu\lambda\left(1+ \frac{\varepsilon_0}{2}\right)\tau} \d \tau
		\nonumber\\ & \leq 16\left\{\|\x\|^4_{2}e^{\frac{\nu\lambda}{2}\left(1+ \frac{\varepsilon_0}{2}\right)s}  +  \|\z(s)\|^4_{2}e^{\frac{\nu\lambda}{2}\left(1+ \frac{\varepsilon_0}{2}\right)s}\right\}\     \left\{-se^{\frac{\nu\lambda}{2}\left(1+ \frac{\varepsilon_0}{2}\right)s}\right\}
		\nonumber\\ & \qquad	+ 2C\int_{s}^{0} \bigg[\int_{s}^{\tau} \bigg\{\|\z(\zeta)\|^{2}_{2}+ \|\z(\zeta)\|^{4}_{\mathbb{W}^{1,4}} + \|\f\|^2_{2}\bigg\} e^{\frac{\nu\lambda}{4}\left(1+ \frac{\varepsilon_0}{2}\right)\zeta} \d \zeta\bigg]^2 e^{\frac{\nu\lambda}{2}\left(1+ \frac{\varepsilon_0}{2}\right)\tau} \d \tau.
	\end{align*}
	Using \eqref{UTE-TRV}, Lemmas \ref{Bddns4} and \ref{Bddns5}, and the fact that $\lim\limits_{s\to - \infty}\left\{-s e^{\frac{\nu\lambda}{2}\left(1+ \frac{\varepsilon_0}{2}\right)s}\right\}=0,$ we have that there exists a $t^{\ast\ast}_{\D}(\omega)\geq t_{\D}(\omega)$ such that for all $s\leq -t^{\ast\ast}_{\D}(\omega) $
	\begin{align*}
		&  \int_{s}^{0} \|\y(\tau,\omega;s,\x-\z(s))\|^{4}_{2} e^{\nu\lambda\left(1+ \frac{\varepsilon_0}{2}\right)\tau} \d \tau
		\nonumber\\ & \leq 16 + 16 \sup_{s\leq 0} \left\{\|\z(s)\|^4_{2}e^{\frac{\nu\lambda}{2}\left(1+ \frac{\varepsilon_0}{2}\right)s}\right\}
		+ C \bigg[\int_{-\infty}^{0} \bigg\{\|\z(\zeta)\|^{2}_{2}+ \|\z(\zeta)\|^{4}_{\mathbb{W}^{1,4}} + \|\f\|^2_{2}\bigg\} e^{\frac{\nu\lambda}{4}\left(1+ \frac{\varepsilon_0}{2}\right)\zeta} \d \zeta\bigg]^2,
	\end{align*}
	which is finite due to Lemmas \ref{Bddns4} and \ref{Bddns5}. This completes the proof of \eqref{S3-ubd}.
\end{proof}

The following result is used to obtain the uniform-tail estimates for the solution of system \eqref{CTGF-with-Pressure} (see Lemma \ref{LR} below).
\begin{lemma}\label{D-convege}
	Let \eqref{condition1} and Hypotheses \ref{assump1}-\ref{assumpO} be satisfied and let us denote by $\y(t, \z,\f,\boldsymbol{x}),$ the solution of   system \eqref{CTGF}. Assume that $\f\in \H^{-1}(\O) $ and $\{\x_{m}\}_{m\in\N}$ be a bounded sequence in $\H$. Then, there exists $\tilde{\y}\in \mathrm{L}^{\infty}(0, T; \H)\cap\mathrm{L}^{2}(0, T; \V)\cap\mathrm{L}^{4}(0, T; \mathbb{W}^{1,4}(\O))$  such that along a subsequence
	\begin{align}\label{D-convege-4}
		\y(\cdot, \z,\f,\boldsymbol{x}_m) \ \to \tilde{\y} \ \text{	in	} \ \mathrm{L}^{2}(0,T;\L^2_{\mathrm{loc}}(\O)),
	\end{align}
	for every $T>0$.
\end{lemma}
\begin{proof}
	Let $\y_m(\cdot)= \y(\cdot, \z,\f,\boldsymbol{x}_m)$. Since $\{\x_{m}\}_{m\in\N}$ is a bounded sequence in $\H$, the sequence
	\begin{align}\label{Bounded}
		\{\y_m\}_{n\in \N} \ \text{ is bounded in }\  \mathrm{L}^{\infty}(0, T; \H)\cap\mathrm{L}^{2}(0, T; \V)\cap\mathrm{L}^{4}(0, T; \mathbb{W}^{1,4}(\O)).
	\end{align}
	Hence, there exists a subsequence $\{\y_{m'}\}_{m'\in \N}$ of $\{\y_m\}_{m\in \N}$ and  $$\widetilde{\y}\in \mathrm{L}^{\infty}(0, T; \H)\cap\mathrm{L}^{2}(0, T; \V)\cap\mathrm{L}^{4}(0 , T; \mathbb{W}^{1,4}(\O)),$$ such that, as $m' \to \infty,$ (by the Banach-Alaoglu theorem)
	\begin{align}\label{lim1}
		\begin{cases}
			\y_{m'} \xrightharpoonup{w^*} \widetilde{\y} \ \text{  in } \ \mathrm{L}^{\infty}(0, T; \H)\\ \y_{m'} \xrightharpoonup{w} \widetilde{\y} \ \text{  in }\ \mathrm{L}^{2}(0, T; \V)\cap\mathrm{L}^{4}(0, T; \mathbb{W}^{1,4}(\O)).
		\end{cases}
	\end{align}		
	From Theorem \ref{solution}, we have $\big\|\frac{\d\y_{m}}{\d t}\big\|_{\mathrm{L}^{\frac{4}{3}}(0, T; \V'+\mathbb{W}^{-1,\frac{4}{3}})}\leq C,$ for some $C>0$ and all $m\in \N.$ Therefore, by the Cauchy-Schwarz inequality, for all $0\leq t \leq t+a \leq T$, $m\in \N$ and $\boldsymbol{\varphi}\in \V\cap\mathbb{W}^{1,4}(\O)$, we obtain
	\begin{align}\label{335} 
		|(\y_m(t+a)-\y_m(t), \boldsymbol{\varphi})|\leq \int_{t}^{t+a}\bigg|\bigg\langle\frac{\d\y_{m}(s)}{\d t} , \boldsymbol{\varphi} \bigg\rangle\bigg|\d s\leq C(T) a^{\frac{1}{4}} \|\boldsymbol{\varphi}\|_{\V\cap\mathbb{W}^{1,4}}.
	\end{align}
	Since $\y_m(t+a)-\y_m(t)\in\V\cap\mathbb{W}^{1,4}(\O)$, for a.e. $t\in(0, T)$, choosing $\boldsymbol{\varphi}=\y_m(t+a)-\y_m(t)$ in \eqref{335}, we obtain 
	\begin{align}
		\|\y_m(t+a)-\y_m(t)\|_{2}^2\leq C(T)a^{\frac{1}{4}}\|\y_m(t+a)-\y_m(t)\|_{\V\cap\mathbb{W}^{1,4}}.
	\end{align}
	Integrating from $0$ to $T-a$, we further find 
	\begin{align}
		\int_{0}^{T-a}	\|\y_m(t+a)-\y_m(t)\|_{2}^2\d t
		 &\leq C(T)a^{\frac{1}{4}}\int_{0}^{T-a}\|\y_m(t+a)-\y_m(t)\|_{\V\cap\mathbb{W}^{1,4}}\d t\nonumber\\&\leq C(T)a^{\frac{1}{4}}\bigg[(T-a)^{1/2}\left(\int_{0}^{T-a}\|\y_m(t+a)-\y_m(t)\|_{\V}^2\d t\right)^{1/2}
		\nonumber\\ & \qquad+(T-a)^{\frac{3}{4}}\left(\int_{0}^{T-a}\|\y_m(t+a)-\y_m(t)\|_{\mathbb{W}^{1,4}}^{4}\d t\right)^{\frac{1}{4}}\bigg]\nonumber\\&\leq \widetilde{C}(T)a^{\frac{1}{4}},
	\end{align}
	where we have used the H\"older's inequality and \eqref{Bounded}. Also, $\widetilde{C}(T)$ is an another positive constant independent of $m$. Furthermore, we have 
	\begin{align}\label{3p43}
		\lim_{a\to 0}\sup_m	\int_{0}^{T-a}	\|\y_m(t+a)-\y_m(t)\|_{2}^2\d t=0.
	\end{align}
	
	Let us now consider a truncation function $\chi\in\C^1(\R^+)$ with $\chi(s)=1$ for $s\in[0,1]$ and $\chi(s)=0$ for $s\in[4,\infty)$. For each $k>0$, let us define $$\y_{m,k}(x):=\chi\left(\frac{|x|^2}{k^2}\right)\y_m(x), \ \text{ for }x\in\O_{2k},$$
	where $\O_k= \{x\in\O: |x|< k\}$. It can easily be seen from \eqref{3p43} that 
	\begin{align}
		\lim_{a\to 0}\sup_m	\int_{0}^{T-a}	\|\y_{m,k}(t+a)-\y_{m,k}(t)\|_{\L^2(\O_{2k})}^2\d t=0,
	\end{align}
	for all $T,k>0$. Moreover, from \eqref{Bounded}, for all $T,k>0$, we infer 
	\begin{align}
		&  \{\y_{m,k}\}_{m\in\N}\ \text{ is bounded in }\ \mathrm{L}^{\infty}(0 , T;\L^2(\O_{2k}))\cap\mathrm{L}^2(0 , T;\H^1_0(\O_{2k})).
	\end{align}
	Since the injection $\H_0^1(\O_{2k})\subset\L^2{(\O_{2k})}$ is compact, we can apply  \cite[Theorem 16.3]{MMR} (see \cite[Theorem 13.3]{Te1} also) to obtain 
	\begin{align}\label{3.47}
		\{\y_{m,k}\}_{m\in\N}\ \text{ is relatively compact in } \ \mathrm{L}^{2}(0, T;\L^2(\O_{2k})),
	\end{align}
	for all $T,k>0$. From \eqref{3.47}, we further infer 
	\begin{align}\label{3.48}
		\{\y_{m}\}_{m\in\N}\ \text{ is relatively compact in } \ \mathrm{L}^{2}(0 , T;\L^2(\O_{k})),
	\end{align}
	for all $T,k>0$. 
	Using \eqref{lim1} and \eqref{3.48}, we can extract a subsequence of $\{\y_m\}_{m\in\N}$ (not relabeling) such that 
	\begin{align*}
		\y_m \to \tilde{\y} \ \text{ in } \mathrm{L}^{2} (0,  T;\L^{2}(\O_k)), 
	\end{align*}
	for all $T,k>0$. This completes the proof.
\end{proof}

Next, we show that the solution to system \eqref{CTGF-with-Pressure} satisfies the uniform-tail estimates which will help us to establish the $\mathfrak{DK}$-asymptotic compactness of $\Psi$ on unbounded domains. 

Let $\Lambda$ be a smooth function such that $0\leq\Lambda(\xi)\leq 1$ for $\xi\in\R$ and
\begin{align}\label{Cut-off}
	\Lambda(\xi)=\begin{cases*}
		0, \text{ for } |\xi|\leq 1,\\
		1, \text{ for } |\xi|\geq2.
	\end{cases*}
\end{align}
Then, there exists a positive constant $C^{\ast}$ such that $|\Lambda'(\xi)|\leq C^{\ast}$ and $|\Lambda''(\xi)|\leq C^{\ast}$ for all $\xi\in\R$.

\begin{lemma}\label{LR}
	Let \eqref{condition1} and Hypotheses \ref{assump1}-\ref{assumpO} be satisfied,  $\f\in\L^2(\O)$, $\omega\in\Omega$ and $\D \in\mathfrak{DK}$. Then, for each $\varepsilon>0$ and for each $\xi\in[-1,0]$, there exists $k_0:=k_0(\varepsilon,\xi,\omega,\D)\in\N$  such that  the solution $\y(\cdot)$ of system \eqref{CTGF-with-Pressure} satisfies
	\begin{align}
		\left\|\y(\xi,\omega;s,\x-\z(s))\right\|^2_{\L^2(\O^{c}_{k})}\leq\varepsilon, 
	\end{align}
	for all $\x\in \D(\theta_{s}\omega)$, for all $s\leq -\mathcal{T}=-\max\{t_{\D}^{\ast}(\omega), t_{\D}^{\ast\ast}(\omega)\}$ and for all $k\geq k_0$,	where  $\O_{k}=\{x\in\O:\,|x|< k\}$, $\O^{c}_{k}=\O\backslash\O_{k}$,
	and $t_{\D}^{\ast}(\omega)$ and $t_{\D}^{\ast\ast}(\omega)$ are the same time given in Lemma \ref{D-abH-ubd}.
\end{lemma}

\begin{proof} 
Note that an application of Poincar\'e inequality \eqref{poin} allow us to write
	\begin{align*}
			 \lambda \int_{\O}\Lambda^2\left(|x|^2k^{-2}\right)\left| \y\right|^2\d x  
		 & \leq  \int_{\O}\left|\nabla\left(\Lambda\left(|x|^2k^{-2}\right) \y\right)\right|^2\d x
		\nonumber\\ &\leq  \frac{1}{2} \int_{\O}\left|\Arm\left(\Lambda\left(|x|^2k^{-2}\right) \y\right)\right|^2\d x
		\nonumber\\ & =\frac{1}{2} \int_{\O}\left|\Lambda\left(|x|^2k^{-2}\right) \Arm(\y)+\frac{2}{k^2}\Lambda^{\prime}\left(|x|^2k^{-2}\right) x^{T}\y + \frac{2}{k^2}\Lambda^{\prime}\left(|x|^2k^{-2}\right) \y^{T} x \right|^2\d x
		\nonumber\\ & \leq \frac{1}{2} \int_{\O}\Lambda^2\left(|x|^2k^{-2}\right) |\Arm(\y)|^2\d x + \frac{C}{k}\int_{\O}\left[|\Arm(\y)||\y|+|\y|^2\right]\d x
		\nonumber\\ & \leq \frac{1}{2} \int_{\O}\Lambda^2\left(|x|^2k^{-2}\right) |\Arm(\y)|^2\d x + \frac{C}{k}\|\Arm(\y)\|_{2}^2,
	\end{align*}
	which implies
	\begin{align}\label{poin-weight}
		- \frac{1}{2} \int_{\O}\Lambda^2\left(|x|^2k^{-2}\right) |\Arm(\y)|^2\d x \leq	- \lambda \int_{\O}\Lambda^2\left(|x|^2k^{-2}\right)\left| \y\right|^2\d x   + \frac{C}{k}\|\Arm(\y)\|_{2}^2.
	\end{align}

	Taking the inner product of the first equation of \eqref{CTGF-with-Pressure} with $\Lambda^2\left(|x|^2k^{-2}\right)\y$, we have
	\begin{align}\label{ep1}
		& \frac{1}{2} \frac{\d}{\d t}\int_{\O}\Lambda^2\left(|x|^2k^{-2}\right)|\y|^2\d x 
		\nonumber\\& = \underbrace{\nu\left\langle \Delta\y,  \Lambda^2\left(|x|^2k^{-2}\right) \right\rangle}_{I_{1}(k,t)} -\underbrace{ b\left(\y+\z,\y+\z,\Lambda^2\left(|x|^2k^{-2}\right)(\y+\z)\right)}_{I_{2}(k,t)} + \underbrace{ b\left(\y+\z,\y+\z,\Lambda^2\left(|x|^2k^{-2}\right)\z\right)}_{I_{3}(k,t)} 
		\nonumber\\&\quad  + \underbrace{\alpha \left\langle \text{div}((\Arm(\y+\z))^2) , \Lambda^2\left(|x|^2k^{-2}\right) \y \right\rangle}_{I_{4}(k,t)} 
		   + \underbrace{ \beta \left\langle \text{div}(|\Arm(\y+\z)|^2\Arm(\y+\z)) , \Lambda^2\left(|x|^2k^{-2}\right) (\y+\z) \right\rangle }_{I_{5}(k,t)} 
		\nonumber\\&\quad - \underbrace{ \beta \left\langle \text{div}(|\Arm(\y+\z)|^2\Arm(\y+\z)) ,  \Lambda^2\left(|x|^2k^{-2}\right) \z \right\rangle}_{I_{6}(k,t)}  - \underbrace{ \left\langle \nabla \widehat{\textbf{P}}, \Lambda^2\left(|x|^2k^{-2}\right)\y \right\rangle}_{I_{7}(k,t)} 
		\nonumber\\ & \quad + \underbrace{\chi\int_{\O}\z\Lambda^2\left(|x|^2k^{-2}\right)\y\d x}_{I_{8}(k,t)}
		  + \underbrace{\int_{\O}\f\Lambda^2\left(|x|^2k^{-2}\right)\y\d x}_{I_{9}(k,t)}.
	\end{align}	
	Let us now estimate each term on the right hand side of \eqref{ep1}. Integration by parts, \eqref{condition1} (see Remark \ref{rem-condition1} also), and H\"older's, Gagliardo-Nirenberg's \eqref{Gen_lady-4}, Korn's \eqref{Korn-ineq} and Young's inequalities help us to obtain
	\begin{align}
		I_{1}(k,t)
		&= -\frac{\nu}{2} \int_{\O} \Lambda^2\left(|x|^2k^{-2}\right) |\Arm(\y)|^2  \d x - \frac{4\nu}{k^2} \int\limits_{\O}\Lambda\left(|x|^2k^{-2}\right) \Lambda^{\prime}\left(|x|^2k^{-2}\right)
		[(x\cdot\nabla) \y\cdot\y+(\y\cdot\nabla) \y\cdot x]\d x\nonumber\\
		&\leq -\frac{\nu}{2} \int_{\O}\Lambda^2\left(|x|^2k^{-2}\right)  |\Arm(\y)|^2   \d x+  \frac{4\sqrt{2}\nu}{k} \int\limits_{\O\cap\{k\leq|x|\leq \sqrt{2}k\}} \left|\Lambda^{\prime}\left(|x|^2k^{-2}\right)\right| \left|\y\right| \left|\nabla \y\right| \d x\nonumber\\&\leq -\frac{\nu}{2} \int_{\O}\Lambda^2\left(|x|^2k^{-2}\right)  |\Arm(\y)|^2    \d x+  \frac{C}{k} \int_{\O}\left|\y\right| \left|\nabla \y\right| \d x
		\nonumber\\ & \leq  - \frac{\nu}{2} \int_{\O} \Lambda^2\left(|x|^2k^{-2}\right)  |\Arm(\y)|^2    \d x +   \frac{C}{k} \left(\|\y\|^2_{2}+\|\nabla\y\|^2_{2}\right)
		\nonumber\\ & \leq -\frac{\nu}{2} \int_{\O}\Lambda^2\left(|x|^2k^{-2}\right)  |\Arm(\y)|^2    \d x +  \frac{C}{k}\|\Arm(\y)\|^2_{2},\label{ep2}
		\\  
		|I_{2}(k,t)|&=\left|\frac{2}{k^2}\int_{\O} \Lambda\left(|x|^2k^{-2}\right)\Lambda^{\prime}\left(|x|^2k^{-2}\right)x\cdot(\y+\z) |\y+\z|^2 \d x\right|\nonumber\\&= \left|\frac{2}{k^2} \int\limits_{\O\cap\{k\leq|x|\leq \sqrt{2}k\}}\Lambda\left(|x|^2k^{-2}\right) \Lambda^{\prime}\left(|x|^2k^{-2}\right)x\cdot(\y+\z) |\y+\z|^2 \d x\right|\nonumber\\&\leq \frac{2\sqrt{2}}{k} \int\limits_{\O\cap\{k\leq|x|\leq \sqrt{2}k\}} \left|\Lambda^{\prime}\left(|x|^2k^{-2}\right)\right| |\y+\z|^3 \d x \nonumber\\& \leq \frac{C}{k}\|\y+\z\|^3_{ 3} \leq\frac{C}{k} \left[\|\y+\z\|^4_{4}+\|\y+\z\|^2_{2}\right] 
		\nonumber\\ & \leq\frac{C}{k} \left[\|\y\|^4_{4}+\|\z\|^4_{4}+\|\Arm(\y)\|^2_{2}+\|\z\|^2_{2}\right] \nonumber\\ & \leq\frac{C}{k} \left[\|\Arm(\y)\|^4_{4}+ \|\y\|^4_{2}+\|\Arm(\y)\|^2_{2} + \|\z\|^4_{4}+\|\z\|^2_{2} \right]
		\nonumber\\ & \leq\frac{C}{k} \left[\|\Arm(\y+\z)\|^4_{4}+ \|\y\|^4_{2}+\|\Arm(\y)\|^2_{2} + \|\z\|^4_{\mathbb{W}^{1,4}}+\|\z\|^2_{2} \right],\label{ep3}
		\\
		|I_{3}(k,t)| & \leq  \int_{\O} |\y+\z||\nabla(\y+\z)|\Lambda^2\left(|x|^2k^{-2}\right)|\z| \d x 
		\nonumber\\ & \leq \|\y+\z\|_{2}\|\nabla(\y+\z)\|_{4}\left(\int_{\O}\Lambda^8\left(|x|^2k^{-2}\right)|\z|^4\d x\right)^{\frac14}
		\nonumber\\ & \leq C \{\|\Arm(\y)\|_{2}+\|\z\|_{2}\}\|\Arm(\y+\z)\|_{4}\left(\int_{\O\cap\{|x|\geq k\}}|\z|^4\d x\right)^{\frac14},\label{ep4}
		\\
		|I_{4}(k,t)| & \leq  \left|	\frac{\alpha}{2}  \int_{\O} \Lambda^2\left(|x|^2k^{-2}\right)[(\Arm(\y+\z))^2: \Arm(\y)]  \d x \right|
		\nonumber\\ & \qquad + \left|\frac{4\alpha}{k^2} \int_{\O} \Lambda\left(|x|^2k^{-2}\right)\Lambda^{\prime}\left(|x|^2k^{-2}\right) (\Arm(\y+\z))^2: (x^{T} \y) \d x\right|
		\nonumber\\ & \leq  \frac{\nu(1-\varepsilon_0)}{4} \int_{\O} \Lambda^2\left(|x|^2k^{-2}\right)|\Arm(\y)|^2   \d x  + \frac{\beta(1-\varepsilon_0)}{2}   \int_{\O} \Lambda^2\left(|x|^2k^{-2}\right)|\Arm(\y+\z)|^4 \d x
		\nonumber\\ & \quad + \frac{C}{k} \left[\|\Arm(\y+\z)\|^4_{4}+\|\y\|^2_{2}\right]
		\nonumber\\ & \leq  \frac{\nu(1-\varepsilon_0)}{4} \int_{\O} \Lambda^2\left(|x|^2k^{-2}\right)|\Arm(\y)|^2 \d x  + \frac{\beta(1-\varepsilon_0)}{2}   \int_{\O} \Lambda^2\left(|x|^2k^{-2}\right)|\Arm(\y+\z)|^4 \d x
		\nonumber\\ & \quad + \frac{C}{k} \left[\|\Arm(\y+\z)\|^4_{4}+\|\Arm(\y)\|^2_{2}\right],\label{ep5}
		\\
		I_{5}(k,t) & = -	\frac{\beta}{2}  \int_{\O} \Lambda^2\left(|x|^2k^{-2}\right)|\Arm(\y+\z)|^4 \d x 
		\nonumber\\ & \qquad - \frac{4\beta}{k^2} \int_{\O} \Lambda\left(|x|^2k^{-2}\right)\Lambda^{\prime}\left(|x|^2k^{-2}\right) [|\Arm(\y+\z)|^2\Arm(\y+\z): (x^{T} (\y+\z))] \d x
		\nonumber\\ & = -	\frac{\beta}{2}  \int_{\O} \Lambda^2\left(|x|^2k^{-2}\right)|\Arm(\y+\z)|^4 \d x 
		\nonumber\\ & \qquad - \frac{4\beta}{k^2} \int\limits_{\O\cap\{k\leq|x|\leq \sqrt{2}k\}} \Lambda^{\prime}\left(|x|^2k^{-2}\right) [|\Arm(\y+\z)|^2\Arm(\y+\z): (x^{T} (\y+\z))] \d x
		\nonumber\\ & \leq -	\frac{\beta}{2}  \int_{\O} \Lambda^2\left(|x|^2k^{-2}\right)|\Arm(\y+\z)|^4 \d x + \frac{C}{k} \int\limits_{\O\cap\{k\leq|x|\leq \sqrt{2}k\}}  |\Arm(\y+\z)|^3 |\y+\z| \d x
		\nonumber\\ & \leq -	\frac{\beta}{2}  \int_{\O} \Lambda^2\left(|x|^2k^{-2}\right)|\Arm(\y+\z)|^4 \d x + \frac{C}{k}\|\Arm(\y+\z)\|^3_{4}\|\y+\z\|_{4}
		\nonumber\\ & \leq -	\frac{\beta}{2}  \int_{\O} \Lambda^2\left(|x|^2k^{-2}\right)|\Arm(\y+\z)|^4 \d x + \frac{C}{k}\left[\|\Arm(\y+\z)\|^4_{4}+\|\y\|^4_{4} +\|\z\|^4_{4}\right]
		\nonumber\\ & \leq -	\frac{\beta}{2}  \int_{\O} \Lambda^2\left(|x|^2k^{-2}\right)|\Arm(\y+\z)|^4 \d x + \frac{C}{k}\left[\|\Arm(\y+\z)\|^4_{4}+\|\y\|^4_{2}+ \|\z\|^4_{\mathbb{W}^{1,4}}\right],\label{ep6}
		\\
		 |I_{6}(k,t)| & \leq 	\beta  \int_{\O} \Lambda^2\left(|x|^2k^{-2}\right)|\Arm(\y+\z)|^3|\nabla\z| \d x 
		  + \frac{4\beta}{k^2} \int_{\O} \Lambda\left(|x|^2k^{-2}\right)\Lambda^{\prime}\left(|x|^2k^{-2}\right) |\Arm(\y+\z)|^3 |x| |\z| \d x
		\nonumber\\ & \leq	C  \|\Arm(\y+\z)\|^3_{4} \left(\int_{\O\cap\{|x|\geq k\}}|\nabla\z|^4\d x\right)^{\frac14}  + \frac{C}{k} \int\limits_{\O\cap\{k\leq|x|\leq \sqrt{2}k\}}  |\Arm(\y+\z)|^3 |\z| \d x
		\nonumber\\ & \leq C  \|\Arm(\y+\z)\|^3_{4} \left(\int_{\O\cap\{|x|\geq k\}}|\nabla\z|^4\d x\right)^{\frac14} + \frac{C}{k}\|\Arm(\y+\z)\|^3_{4}\|\z\|_{4}
		\nonumber\\ & \leq C  \|\Arm(\y+\z)\|^3_{4} \left(\int_{\O\cap\{|x|\geq k\}}|\nabla\z|^4\d x\right)^{\frac14} + \frac{C}{k}\left[\|\Arm(\y+\z)\|^4_{4}+\|\z\|^4_{4}\right],\label{ep7}
		\\
		|I_7(k,t)| & \leq  \left|\left\langle \nabla \widehat{\textbf{P}}_1, \Lambda^2\left(|x|^2k^{-2}\right)\y \right\rangle\right| + \left|\left\langle \nabla \widehat{\textbf{P}}_2, \Lambda^2\left(|x|^2k^{-2}\right)\y \right\rangle\right|
        \nonumber\\ &  \leq  \|\nabla \widehat{\textbf{P}}_1\|_{\H^{-1}} \| \Lambda^2\left(|x|^2k^{-2}\right)\y \|_{\H^1} + \| \nabla \widehat{\textbf{P}}_2\|_{\mathbb{W}^{-1,\frac43}}\| \Lambda^2\left(|x|^2k^{-2}\right)\y \|_{\mathbb{W}^{1,4}}
        \nonumber\\ &   \leq  \|\nabla \widehat{\textbf{P}}_1\|_{\H^{-1}} \bigg(\| \Lambda^2\left(|x|^2k^{-2}\right)\y \|^2_{2} + \| \nabla[\Lambda^2\left(|x|^2k^{-2}\right)\y ]\|^2_{2}\bigg)^{\frac12} 
        \nonumber\\ &  \quad + \| \nabla \widehat{\textbf{P}}_2\|_{\mathbb{W}^{-1,\frac43}}\bigg(\| \Lambda^2\left(|x|^2k^{-2}\right)\y \|^4_{4} + \| \nabla[\Lambda^2\left(|x|^2k^{-2}\right)\y] \|^4_{4}\bigg)^{\frac14}
        \nonumber\\ &  \leq C \|\nabla \widehat{\textbf{P}}_1\|_{\H^{-1}} \bigg( \int_{\O\cap\{|x|\geq k\}} |\y|^2\d x + \int_{\O\cap\{|x|\geq k\}} |\Arm(\y)|^2\d x\bigg)^{\frac12} 
        \nonumber\\ &   + C \| \nabla \widehat{\textbf{P}}_2\|_{\mathbb{W}^{-1,\frac43}}\bigg(\bigg\{\int_{\O\cap\{|x|\geq k\}} |\y|^2\d x\bigg\}^2 + \int_{\O\cap\{|x|\geq k\}} |\Arm(\y+\z)|^4\d x + \int_{\O\cap\{|x|\geq k\}} |\nabla\z|^4\d x\bigg)^{\frac14},\label{ep8}
		\\
		|I_{8}(k,t)|&\leq \frac{\nu\lambda\varepsilon_0}{8}\int_{\O}\Lambda^2\left(|x|^2k^{-2}\right)|\v|^2\d x + C\int_{\O}\Lambda^2\left(|x|^2k^{-2}\right)|\f|^2\d x, \label{ep9}
		\\
		|I_{9}(k,t)|&\leq \frac{\nu\lambda\varepsilon_0}{8}\int_{\O}\Lambda^2\left(|x|^2k^{-2}\right)|\v|^2\d x + C\int_{\O}\Lambda^2\left(|x|^2k^{-2}\right)|\z|^2\d x.\label{ep10}
	\end{align}
	Combining \eqref{ep1}-\eqref{ep10} and using \eqref{poin-weight}, we reach at
	\begin{align*}
		&	\frac{1}{2} \frac{\d}{\d t}\int_{\O}\Lambda^2\left(|x|^2k^{-2}\right)|\y|^2\d x + \frac{\beta\varepsilon_0}{2}   \int_{\O} \Lambda^2\left(|x|^2k^{-2}\right)|\Arm(\y+\z)|^4 \d x
	\nonumber\\ & \leq - \frac{\nu\lambda}{2}(1+\varepsilon_0) \int_{\O} \Lambda^2\left(|x|^2k^{-2}\right)|\v|^2   \d x + \frac{C}{k} \|\Arm(\v)\|^2_2  + \frac{\nu\lambda\varepsilon_0}{4}\int_{\O}\Lambda^2\left(|x|^2k^{-2}\right)|\v|^2\d x 
\nonumber\\ & \quad + C\int_{\O}\Lambda^2\left(|x|^2k^{-2}\right)|\f|^2\d x+ C\int_{\O}\Lambda^2\left(|x|^2k^{-2}\right)|\z|^2\d x \nonumber\\ & \quad + C \{\|\Arm(\y)\|_{2}+\|\z\|_{2}\}\|\Arm(\y+\z)\|_{4}\left(\int_{\O\cap\{|x|\geq k\}}|\z|^4\d x\right)^{\frac14} 
		  + C  \|\Arm(\y+\z)\|^3_{4} \left(\int_{\O\cap\{|x|\geq k\}}|\nabla\z|^4\d x\right)^{\frac14}
          \nonumber\\ & \quad +  C \|\nabla \widehat{\textbf{P}}_1\|_{\H^{-1}} \bigg( \int_{\O\cap\{|x|\geq k\}} |\y|^2\d x + \int_{\O\cap\{|x|\geq k\}} |\Arm(\y)|^2\d x\bigg)^{\frac12} 
        \nonumber\\ &  \quad + C \| \nabla \widehat{\textbf{P}}_2\|_{\mathbb{W}^{-1,\frac43}}\bigg(\bigg\{\int_{\O\cap\{|x|\geq k\}} |\y|^2\d x\bigg\}^2 + \int_{\O\cap\{|x|\geq k\}} |\Arm(\y+\z)|^4\d x + \int_{\O\cap\{|x|\geq k\}} |\nabla\z|^4\d x\bigg)^{\frac14}
		\nonumber\\ & \quad + \frac{C}{k} \left[\|\y\|^4_{2} + \|\Arm(\y)\|^2_2  + \|\Arm(\y+\z)\|^4_{4}+ \|\z\|^4_{\mathbb{W}^{1,4}} + \|\z\|^2_{2}\right],
	\end{align*}
	which implies 
	\begin{align}\label{ep11}
		&	 \frac{\d}{\d t}\int_{\O}\Lambda^2\left(|x|^2k^{-2}\right)|\v|^2\d x + \nu\lambda\left(1+\frac{\varepsilon_0}{2}\right)  \int_{\O} \Lambda^2\left(|x|^2k^{-2}\right)|\v|^2   \d x +  \beta\varepsilon_0   \int_{\O} \Lambda^2\left(|x|^2k^{-2}\right)|\Arm(\v)|^4 \d x
		\nonumber\\ & \leq   C\int_{\O\cap\{|x|\geq k\}}|\f|^2\d x + C\int_{\O\cap\{|x|\geq k\}}|\z|^2\d x + C  \|\Arm(\y+\z)\|^3_{4} \left(\int_{\O\cap\{|x|\geq k\}}|\nabla\z|^4\d x\right)^{\frac14}
		\nonumber\\ & \quad + C \bigg\{\|\Arm(\y)\|_{2}+\|\z\|_{2}\bigg\}\|\Arm(\y+\z)\|_{4}\left(\int_{\O\cap\{|x|\geq k\}}|\z|^4\d x\right)^{\frac14} 
        \nonumber\\ & \quad +  C \|\nabla \widehat{\textbf{P}}_1\|_{\H^{-1}} \bigg( \int_{\O\cap\{|x|\geq k\}} |\y|^2\d x + \int_{\O\cap\{|x|\geq k\}} |\Arm(\y)|^2\d x\bigg)^{\frac12} 
        \nonumber\\ &  \quad + C \| \nabla \widehat{\textbf{P}}_2\|_{\mathbb{W}^{-1,\frac43}}\bigg(\bigg\{\int_{\O\cap\{|x|\geq k\}} |\y|^2\d x\bigg\}^2 + \int_{\O\cap\{|x|\geq k\}} |\Arm(\y+\z)|^4\d x + \int_{\O\cap\{|x|\geq k\}} |\nabla\z|^4\d x\bigg)^{\frac14}
		\nonumber\\ & \quad + \frac{C}{k} \bigg[\|\y\|^4_{2} + \|\Arm(\y)\|^2_2  + \|\Arm(\y+\z)\|^4_{4}
		  + \|\z\|^4_{\mathbb{W}^{1,4}} + \|\z\|^2_{2} + \|\f\|^2_{2} \bigg].
	\end{align}
	An application of variation of constants formula to \eqref{ep8} gives for all $\omega\in\Omega$, for all $\xi\in[-1,0]$ and for all $s\leq - \mathcal{T}= -\max\{t_{\D}^{\ast}(\omega),t_{\D}^{\ast\ast}(\omega)\}$ 
	\begin{align}\label{ep12}
		&e^{-\nu\lambda\left(1+\frac{\varepsilon_0}{2}\right)}	\int_{\O}\Lambda^2\left(|x|^2k^{-2}\right)|\y(\xi,\omega;s,\x-\z(s))|^2\d x 
		\nonumber\\&\leq e^{\nu\lambda\left(1+\frac{\varepsilon_0}{2}\right) s}\int_{\O}\Lambda^2\left(|x|^2k^{-2}\right)|\x(x)-\z(x,s)|^2\d x + C\int_{s}^{0} e^{\nu\lambda\left(1+\frac{\varepsilon_0}{2}\right) \tau}\int_{\O\cap\{|x|\geq k\}}|\f(x)|^2\d x \d \tau 
		\nonumber\\ & \quad + C\int_{s}^{0} e^{\nu\lambda\left(1+\frac{\varepsilon_0}{2}\right) \tau}\int_{\O\cap\{|x|\geq k\}}|\z(x,\tau)|^2\d x \d \tau + \frac{C}{k}\int_{s}^{0} e^{\nu\lambda\left(1+\frac{\varepsilon_0}{2}\right) \tau}\bigg[\|\y(\tau,\omega;s,\x-\z(s))\|^4_{2} \nonumber\\& \qquad +  \|\Arm(\y(\tau,\omega;s,\x-\z(s)))\|^2_{2}+ \|\Arm(\y(\tau,\omega;s,\x-\z(s))+\z(\tau))\|^4_{4} + \|\z(\tau)\|^4_{\mathbb{W}^{1,4}}+ \|\z(\tau)\|^2_{2}  + \|\f\|^2_{2} 
		\bigg] \d \tau
		\nonumber\\ & \quad + C\int_{s}^{0} e^{\nu\lambda\left(1+\frac{\varepsilon_0}{2}\right) \tau} \|\Arm(\y(\tau,\omega;s,\x-\z(s))+\z(\tau))\|^3_{4} \left(\int_{\O\cap\{|x|\geq k\}}|\nabla\z(x,\tau)|^4\d x\right)^{\frac14} \d \tau
		\nonumber\\ & \quad + C\int_{s}^{0} e^{\nu\lambda\left(1+\frac{\varepsilon_0}{2}\right) \tau} \{\|\Arm(\y(\tau,\omega;s,\x-\z(s)))\|_{2}+\|\z(\tau)\|_{2}\}\|\Arm(\y(\tau,\omega;s,\x-\z(s))+\z(\tau))\|_{4}
		\nonumber\\ & \qquad \times \left(\int_{\O\cap\{|x|\geq k\}}|\z(x,\tau)|^4\d x\right)^{\frac14} \d\tau
        \nonumber\\ & \quad +   C \int_{s}^{0} e^{\nu\lambda\left(1+\frac{\varepsilon_0}{2}\right) \tau} \|\nabla \widehat{\textbf{P}}_1(\tau,\omega;s,\x-\z(s))\|_{\H^{-1}} \bigg( \int_{\O\cap\{|x|\geq k\}} |\y(\tau,\omega;s,\x-\z(s))|^2\d x 
        \nonumber\\ & \qquad  + \int_{\O\cap\{|x|\geq k\}} |\Arm(\y(\tau,\omega;s,\x-\z(s)))|^2\d x\bigg)^{\frac12}  \d \tau
        \nonumber\\ &   \quad + C \int_{s}^{0} e^{\nu\lambda\left(1+\frac{\varepsilon_0}{2}\right) \tau}  \| \nabla \widehat{\textbf{P}}_2(\tau,\omega;s,\x-\z(s))\|_{\mathbb{W}^{-1,\frac43}}\bigg(\bigg\{\int_{\O\cap\{|x|\geq k\}} |\y(\tau,\omega;s,\x-\z(s))|^2\d x\bigg\}^2
        \nonumber\\ & \qquad   + \int_{\O\cap\{|x|\geq k\}} |\Arm(\y(\tau,\omega;s,\x-\z(s))+\z(\tau)|^4\d x + \int_{\O\cap\{|x|\geq k\}} |\nabla\z(\tau)|^4\d x\bigg)^{\frac14} \d \tau
		\nonumber\\&\leq e^{\nu\lambda\left(1+\frac{\varepsilon_0}{2}\right) s}\int_{\O\cap\{|x|\geq k\}}|\x(x)-\z(x,s)|^2\d x + C\int_{s}^{0} e^{\nu\lambda\left(1+\frac{\varepsilon_0}{2}\right) \tau}\int_{\O\cap\{|x|\geq k\}}|\f(x)|^2\d x \d \tau 
		\nonumber\\ & \quad + C\int_{s}^{0} e^{\nu\lambda\left(1+\frac{\varepsilon_0}{2}\right) \tau}\int_{\O\cap\{|x|\geq k\}}|\z(x,\tau)|^2\d x \d \tau + \frac{C}{k}
		    + C \left(\int_{s}^{0} e^{\nu\lambda\left(1+\frac{\varepsilon_0}{2}\right) \tau} \int_{\O\cap\{|x|\geq k\}}|\nabla\z(x,\tau)|^4\d x \d \tau \right)^{\frac14} 
		\nonumber\\ & \quad + C  \left(\int_{s}^{0} e^{\nu\lambda\left(1+\frac{\varepsilon_0}{2}\right) \tau} \int_{\O\cap\{|x|\geq k\}}|\z(x,\tau)|^4\d x \d\tau \right)^{\frac14} 
        \nonumber\\ & \quad +   C \bigg( \int_{s}^{0} e^{\nu\lambda\left(1+\frac{\varepsilon_0}{2}\right) \tau}   \int_{\O\cap\{|x|\geq k\}} |\y(\tau,\omega;s,\x-\z(s))|^2\d x \d \tau
        \nonumber\\ & \qquad   +  \int_{s}^{0} e^{\nu\lambda\left(1+\frac{\varepsilon_0}{2}\right) \tau} \int_{\O\cap\{|x|\geq k\}} |\Arm(\y(\tau,\omega;s,\x-\z(s)))|^2\d x  \d \tau \bigg)^{\frac12}
        \nonumber\\ &   \quad + C  \bigg( \int_{s}^{0} e^{\nu\lambda\left(1+\frac{\varepsilon_0}{2}\right) \tau}  \bigg\{\int_{\O\cap\{|x|\geq k\}} |\y(\tau,\omega;s,\x-\z(s))|^2\d x\bigg\}^2\d \tau
        \nonumber\\ & \qquad   + \int_{s}^{0} e^{\nu\lambda\left(1+\frac{\varepsilon_0}{2}\right) \tau} \int_{\O\cap\{|x|\geq k\}} |\Arm(\y(\tau,\omega;s,\x-\z(s))+\z(\tau)|^4\d x \d \tau 
      \nonumber\\ & \qquad    + \int_{s}^{0} e^{\nu\lambda\left(1+\frac{\varepsilon_0}{2}\right) \tau} \int_{\O\cap\{|x|\geq k\}} |\nabla\z(\tau)|^4\d x \d \tau \bigg)^{\frac14}
		\nonumber\\&\to 0\text{ as } k\to\infty,
	\end{align}
	where we have used the finiteness of integral obtained in \eqref{S1-ubd}-\eqref{S3-ubd}, Lemmas \ref{Bddns4}-\ref{Bddns5} and \eqref{estimate-pressure} (see Remark \ref{rem-pressure}). Hence, from \eqref{ep12}, we conclude that for any $\varepsilon>0$, for any  $\omega\in\Omega$ and for any $\xi\in[-1,0]$, there exists a $k_0=k_0(\varepsilon,\xi, \omega)\in\N$ such that 
	\begin{align*} 
		&	\int_{\O\cap\{|x|\geq k\}}|\y(\xi,\omega;s,\x-\z(s))|^2 \d x \leq \varepsilon,
	\end{align*}
	for all $k\geq k_0$ and $s\leq -\mathcal{T}$. This completes the proof.
\end{proof}

Let us now provide the main result of this section.

	\begin{theorem}\label{Main_theorem_2}
	Let \eqref{condition1} and Hypotheses \ref{assump1}-\ref{assumpO} be satisfied. Consider the MDS, $\Im = (\Omega, \mathcal{F}, \mathbb{P}, \theta)$ from Proposition \ref{m-DS1}, and the RDS $\Psi$ on $\H$ over $\Im$ generated by system \eqref{STGF} with additive noise satisfying Hypothesis \ref{assump1}. Then, there exists a unique random $\mathfrak{DK}$-attractor for continuous RDS $\Psi$ in $\H$.
\end{theorem}
\begin{proof} 
	Because of \cite[Theorem 2.8]{BCLLLR}, it is only needed to prove that there exists a $\mathfrak{DK}$-absorbing set $\widetilde{\textbf{B}}\in \mathfrak{DK}$ and the RDS $\Psi$ is $\mathfrak{DK}$-asymptotically compact. 	
	\vskip 0.2 cm 
	\noindent
	\textbf{Existence of $\mathfrak{DK}$-absorbing set $\widetilde{\textbf{B}}\in \mathfrak{DK}$:}	Let $\mathrm{D}$ be a random set from the class $\mathfrak{DK}$. Let $\kappa_{\mathrm{D}}(\omega)$ be the radius of $\mathrm{D}(\omega)$, that is, $\kappa_{\mathrm{D}}(\omega):= \sup\{\|x\|_{2} : x \in \mathrm{D}(\omega)\}$ for $\omega\in \Omega$.
	
	Let $\omega\in \Omega$ be fixed.  Also, let us set
	\begin{align}
		[\tilde{\kappa}_{11}(\omega)]^2 &= 2+2\sup_{t\leq 0}\bigg\{ \|\z(t)\|^2_{2}\  e^{\nu\lambda\left(1+ \frac{\varepsilon_0}{2}\right) t}\bigg\} 
		+C_{ubd} \int_{- \infty}^{0} \bigg\{ \|\z(\tau)\|^{2}_{2}+ \|\z(\tau)\|^{4}_{\mathbb{W}^{1,4}} + \|\f\|^2_{2}\bigg\}e^{\nu\lambda\left(1+ \frac{\varepsilon_0}{2}\right) \tau}  \d \tau.
	\end{align}
	
	In view of Lemma \ref{Bddns5} and Proposition \ref{radius}, we deduce that both $\tilde{\kappa}_{11}\in  \mathfrak{K}$ and also that $\tilde{\kappa}_{11}+\kappa_{12}=:\tilde{\kappa}_{13} \in \mathfrak{K}$ as well, where $\kappa_{12}$ is same as in the proof of Theorem \ref{Main_theorem_1}. Therefore the random set $\widetilde{\textbf{B}}$ defined by $$\widetilde{\textbf{B}}(\omega) := \{\v\in\H: \|\v\|_{2}\leq \tilde{\kappa}_{13}(\omega)\}$$ is such that $\widetilde{\textbf{B}}\in\mathfrak{DK}.$ 
	
	Let us now show that $\widetilde{\textbf{B}}$ absorbs $\mathrm{D}$.  For $\omega\in\Omega$, $\boldsymbol{x}\in \mathrm{D}(\theta_{s}\omega)$ and $s\leq- t_{\mathrm{D}}(\omega),$ from \eqref{S1-ubd}, we have
	\begin{align}\label{AB1-ubd}
		\|\y(0,\omega; s, \boldsymbol{x}-\z(s))\|_{2}\leq \tilde{\kappa}_{11}(\omega), \  \text{ for } \ \omega\in \Omega.
	\end{align}
	Thus, we conclude that, for $\omega\in\Omega$  
	\begin{align*} 
		\|\v(0,\omega; s, \boldsymbol{x})\|_{2} \leq \|\y(0,\omega; s, \boldsymbol{x}-\z(s))\|_{2} + \|\z(\omega)(0)\|_{2}\leq \tilde{\kappa}_{13}(\omega).
	\end{align*}
	The above inequality implies that for $\omega\in \Omega$, $\v(0,\omega; s, \boldsymbol{x}) \in \widetilde{\textbf{B}}(\omega)$, for all $s\leq -t_{\mathrm{D}}(\omega).$ This proves  $\widetilde{\textbf{B}}$ absorbs $\mathrm{D}$.	
	\vskip 0.2 cm 
	\noindent
	\textbf{The RDS $\Psi$ is $\mathfrak{DK}$-asymptotically compact.} 		
	In order to establish the $\mathfrak{DK}$-asymptotically compactness of $\Psi$, we use uniform-tail estimates obtained in Lemma \ref{LR}. Let us assume that $\mathrm{D} \in \mathfrak{DK}$ and $\widetilde{\textbf{B}}\in \mathfrak{DK}$ be such that $\widetilde{\textbf{B}}$ absorbs $\mathrm{D}$. Let us fix $\omega\in \Omega$ and take a sequence of positive numbers $\{t_m\}^{\infty}_{m=1}$ such that $t_1\leq t_2 \leq t_3 \leq \cdots$ and $t_m \to \infty$. We take an $\H$-valued sequence $\{\boldsymbol{x}_m\}^{\infty}_{m=1}$ such that $\boldsymbol{x}_m \in \mathrm{D}(\theta_{-t_m}\omega),$ for all $m\in \mathbb{N}.$ Our aim is to show that the sequence $\Psi(t_m,\theta_{-t_m}\omega,\x_{m})$ or $\y(0,\omega;-t_m,\x_{m}-\z(-t_m))$ of solutions to system \eqref{CTGF} has a convergent subsequence in $\H$.

	Lemma \ref{D-abH-ubd} implies that there exists $t_{\D}(\omega) \geq 1$ such that for all $s\leq - t_{\D}(\omega)$ and $\x\in\D(\theta_{s}\omega)$,
	\begin{align}\label{Uac1}
		\|\y(-1,\omega;s,\x-\z(s))\|_2\leq C,
	\end{align}
	where $C$ is a positive constant independent of $s$ and  $\x.$
	Since $t_m\to \infty$, there exists $M_2=M_2(\omega,\D)\in\N$ such that $t_m\geq t_{\D}(\omega),$ for all $m\geq M_2$. Since $\x_{m} \in \D(\theta_{-t_m}\omega)$, \eqref{Uac1} implies that for all $m\geq M_2$, {the sequence 
		\begin{align}\label{Uac2}
			\{\y(-1,\omega;-t_m,\x_{m}-\z(-t_m))\}_{m\geq M_2} \subset\H
		\end{align}
		is bounded in $\H$.}	We infer from \eqref{Uac2} and Lemma \ref{D-convege} that there exists $\xi\in(-1,0)$, $\hat{\y}\in\H$ and a subsequence (not relabeling) such that for every $k\in\N$ 
	\begin{align}\label{Uac3}
		\y(\xi,\omega;-t_m,\x_{m}-\z(-t_m))=\y(\xi,\omega;-1,\y(-1,\omega;-t_m,\x_{m}-\z(-t_m))-\z(-1))\to \hat{\y} \  \text{ in }\   \L^2(\O_k).
	\end{align}
	as $m\to\infty$. Therefore, we infer from {the proof of } Lemma \ref{RDS_Conti1} that there exists a positive constant $C_{Lip}$ such that
	\begin{align}\label{Uac7}
		&\|\y(0,\omega;\xi,\y(\xi,\omega;-t_m,\x_{m}-\z(-t_m))-\z(\xi))-\y(0,\omega;\xi,\hat{\y}-\z(\xi))\|^2_{2}
		\nonumber\\ & \leq C_{Lip}\|\y(\xi,\omega;-t_m,\x_{m}-\z(-t_m))-\hat{\y}\|^2_{2}.
	\end{align}
	Let us now choose $\eta>0$ and fix it. Since $\hat{\y}\in\H$, there exists $K_1=K_1(\eta,\omega,\D)>0$ such that for all $k\geq K_1$, 
	\begin{align}\label{Uac4}
		\int_{\O\cap\{|x|\geq k\}}|\hat{\y}|^2\d x<\frac{\eta^2}{6C_{Lip}},
	\end{align}
	where $C_{Lip}>0$ is a constant defined in \eqref{Uac7}. Also, we know from Lemma \ref{LR} that there exists $M_3=M_3(\xi,\eta,\omega,\D)\in\N$ and $K_2=K_2(\xi,\eta,\omega,\D)\geq K_1$ such that for all $m\geq M_3$ and $k\geq K_2$,
	\begin{align}\label{Uac5}
		\int_{\O\cap\{|x|\geq k\}}|\y(\xi,\omega;-t_m,\x_{m}-\z(-t_m))|^2\d x<\frac{\eta^2}{6C_{Lip}}.
	\end{align}
	From \eqref{Uac3}, we have that there exists $M_4=M_4(\xi,\eta,\omega,\D)>M_3$ such that for all $m\geq M_4$,
	\begin{align}\label{Uac6}
		\int_{\O\cap\{|x|< K_2\}}|\y(\xi,\omega;-t_m,\x_{m}-\z(-t_m))-\hat{\y}|^2\d x<\frac{\eta^2}{3C_{Lip}}.
	\end{align}
	Finally, we infer from \eqref{Uac7} that
	\begin{align}\label{Uac8}
		&\|\y(0,\omega;\xi,\y(\xi,\omega;-t_m,\x_{m}-\z(-t_m))-\z(\xi))-\y(0,\omega;\xi,\hat{\y}-\z(\xi))\|^2_{2}\nonumber\\&\leq  C_{Lip}\bigg[\int_{\O\cap\{|x|<K_2\}}|\y(\xi,\omega;-t_m,\x_{m}-\z(-t_m))-\hat{\y}|^2\d x 
		\nonumber\\ & \qquad + \int_{\O\cap\{|x|\geq K_2\}} |\y(\xi,\omega;-t_m,\x_{m}-\z(-t_m))-\hat{\y}|^2\d x\bigg]\nonumber\\&\leq C_{Lip}\bigg[\int_{\O\cap\{|x|<K_2\}}|\y(\xi,\omega;-t_m,\x_{m}-\z(-t_m))-\hat{\y}|^2\d x \nonumber\\ & \qquad +2\int_{\O\cap\{|x|\geq K_2\}}(|\y(\xi,\omega;-t_m,\x_{m}-\z(-t_m))|^2+|\hat{\y}|^2)\d x\bigg]
		\nonumber\\ & < \eta^2,
	\end{align}
	for every $m\geq M_3$, where we have used \eqref{Uac4}-\eqref{Uac6}. Since $\eta>0$ is arbitrary, we conclude the proof.
\end{proof}

	\section{Invariant measures}\label{sec6}\setcounter{equation}{0}
In this section, we first show the existence of invariant measures for a subclass of third-grade fluids equations in $\H$. It is established in \cite{CF} that the existence of a compact invariant random set is a sufficient condition for the existence of invariant measures, that is, if a RDS $\Psi$ has compact invariant random set, then there exists an invariant measure for $\Psi$ (\cite[Corollary 4.4]{CF}). Since, the random attractor itself is a compact invariant random set, the existence of invariant measures for system \eqref{STGF} is a direct consequence of \cite[Corollary 4.4]{CF} and Theorem \ref{Main_theorem_1}. The uniqueness of invariant measures is still an open problem.

Let us define the transition operator $\{\mathrm{T}_t\}_{t\geq 0}$ by 
\begin{align}\label{71}
	\mathrm{T}_t f(\x)=\int_{\Omega}f(\Psi(\omega,t,\x))\d\mathbb{P}(\omega)=\E\left[f(\Psi(t,\x))\right], 
\end{align}
for all $f\in\mathcal{B}_b(\H)$, where $\mathcal{B}_b(\H)$ is the space of all bounded and Borel measurable functions on $\H$ and $\Psi$ is the RDS corresponding to   system \eqref{STGF}, which is defined by \eqref{combine_sol}. The continuity of $\Psi$ (cf. Lemma \ref{RDS_Conti1}), \cite[Proposition 3.8]{Brzezniak+Li_2006} provides the following result: 
\begin{lemma}\label{Feller}
	Let \eqref{condition1} and Hypotheses \ref{assump1}-\ref{assumpO} be satisfied. Then, the family $\{\mathrm{T}_t\}_{t\geq 0}$ is Feller, that is, $\mathrm{T}_tf\in\C_{b}(\H)$ if $f\in\C_b(\H)$, where $\C_b(\H)$ is the space of all bounded and continuous functions on $\H$. Furthermore, for any $f\in\C_b(\H)$, $\mathrm{T}_tf(\x)\to f(\x)$ as $t\downarrow 0$. 
\end{lemma}
Analogously as in the proof of \cite[Theorem 5.6]{CF}, one can prove that $\Psi$ is a Markov RDS, that is, $\mathrm{T}_{t_1+t_2}=\mathrm{T}_{t_1}\mathrm{T}_{t_2}$, for all $t_1,t_2\geq 0$. Since, we know by \cite[Corollary 4.4]{CF} that if a Markov RDS on a Polish space has an invariant compact random set, then there exists a Feller invariant probability measure $\upmu$ for $\Psi$. 
\begin{definition}
	A Borel probability measure $\upmu$ on $\H$  is called an \textit{invariant measure} for a Markov semigroup $\{\mathrm{T}_t\}_{t\geq 0}$ of Feller operators on $\C_b(\H)$ if and only if $$\mathrm{T}_{t}^*\upmu=\upmu, \ t\geq 0,$$ where $(\mathrm{T}_{t}^*\upmu)(\Gamma)=\int_{\H}\mathrm{T}_{t}(\u,\Gamma)\upmu(\d\u),$ for $\Gamma\in\mathcal{B}(\H)$ and  $\mathrm{T}_t(\u,\cdot)$ is the transition probability, $\mathrm{T}_{t}(\u,\Gamma)=\mathrm{T}_{t}(\chi_{\Gamma})(\u),\ \u\in\H$.
\end{definition}

By the definition of random attractors, it is clear  that there exists an invariant compact random set in $\H$. A Feller invariant probability measure for a Markov RDS $\Psi$ on $\H$ is, by definition, an invariant probability measure for the semigroup $\{\mathrm{T}_t\}_{t\geq 0}$ defined by \eqref{71}. Hence, we have the following result on the existence of invariant measures for   system \eqref{STGF} in $\H$.
\begin{theorem}\label{thm6.3}
	Let \eqref{condition1} and Hypotheses \ref{assump1}-\ref{assumpO} be satisfied. Then, there exists an invariant measure for   system \eqref{STGF} in $\H$ with external forcing $\f\in \H^{-1}(\O)$ and $\f\in \L^{2}(\O)$ on bounded and unbounded domains, respectively.
\end{theorem}

\begin{remark}
	 Note that we discuss the existence of an invariant measure based on the existence of random attractor which is a compact set. But the converse is not true always (see \cite{Scheutzow}). In particular, the authors in \cite{Scheutzow} discuss an example of a random dynamical system in $\R^d$, $d \geq 2$ namely an
isotropic Brownian flow with drift, whose single-point motion is an ergodic
diffusion process and which does not have a weak attractor. This also implies that there is no pathwise attractor for such process.
\end{remark}

\begin{appendix}
	\renewcommand{\thesection}{\Alph{section}}
	\numberwithin{equation}{section}

\section{Retrieval of Pressure }\label{PR} \numberwithin{equation}{section}\setcounter{equation}{0}
We have shown that the system \eqref{CTGF} has unique weak solution in the sense of Definition \ref{defn-CTGF} (see Theorem \ref{solution}). In order to retrieve the pressure, we follow the approach which has been carried out in the work \cite[Section 8]{SS11}. 

Let $\C_c^{\infty}(\mathcal{O};\R^d)$ be the space of all infinitely differentiable functions  ($\R^d$-valued) with compact support in $\mathcal{O}\subset\R^d$.
Let us introduce the space of vector-valued test functions  $\mathrm{C}_c^{\infty}(0,T;\C_c^{\infty}(\mathcal{O};\R^d))$, which consists of all infinitely differentiable $\C_c^{\infty}(\mathcal{O};\R^d)$-valued functions having compact support in $(0,T)$. Moreover, we denote by $\mathscr{D}'((0,T) \times \mathcal{O})$, the dual space of $\mathrm{C}_c^{\infty}(0,T;\C_c^{\infty}(\mathcal{O};\R^d))$. Let us define a functional $\upchi\in\mathscr{D}'((0,T) \times \mathcal{O})$ by
\begin{align}\label{press8}
	&\int_0^T \langle\upeta,\varphi\rangle\d t\nonumber\\&:=\int_0^T 
	\biggl\{\bigg\langle\frac{\partial\y}{\partial t} - \nu \Delta \y + [(\y+\z)\cdot \nabla](\y+\z) - \alpha\text{div}((\Arm(\y + \z))^2)  -  \beta \text{div}(|\Arm(\y + \z)|^2\Arm(\y + \z)) 
    \nonumber\\ & \qquad \qquad - \chi \z - \f,\varphi\bigg\rangle \biggr\}\d t,
\end{align}
for all $\varphi\in\mathrm{C}_c^{\infty}(0,T;\C_c^{\infty}(\mathcal{O};\R^d))$.  Since $\y \in  \mathrm{C}([0,T]; \H) \cap \mathrm{L}^{2}(0,T; \V)\cap \mathrm{L}^{4} (0,T; \mathbb{W}^{1,4}(\mathcal{O})),$ $\z \in \mathrm{L}^{2}(0,T; \H)\cap \mathrm{L}^{4} (0,T; \mathbb{W}^{1,4}(\mathcal{O}))$ and $\frac{\partial\y}{\partial t} \in \mathrm{L}^{2} (0,T;\V') + \mathrm{L}^{\frac43} (0,T;\mathbb{W}^{-1,\frac43}(\mathcal{O}))$, we have (see e.g. \eqref{eqn-deri-bound})
\begin{align}
    \left| \int_0^T \langle\upeta,\varphi\rangle\d t \right| \leq C(\alpha, \beta, \mu, \|\y_0\|_{2}, \|\f\|_{\H^{-1}}) \|\varphi\|_{\mathrm{L}^{2}(0,T; \H^1_0)\cap \mathrm{L}^{4} (0,T; \mathbb{W}^{1,4})},
\end{align}
for all $\varphi\in\mathrm{C}_c^{\infty}(0,T;\C_c^{\infty}(\mathcal{O};\R^d))$. Therefore, it follows that
\begin{align*}
	\upeta\in \mathrm{L}^{2} (0,T;\H^{-1}(\mathcal{O})) + \mathrm{L}^{\frac43} (0,T;\mathbb{W}^{-1,\frac43}(\mathcal{O})).
\end{align*}
Since $\y(\cdot)$ is unique solution of system \eqref{CTGF} in the sense of Definition \ref{defn-CTGF},  \eqref{press8} gives 
\begin{align}\label{press11}
	\int_0^T \langle\upeta,\varphi\rangle\d t=0, \  \text{ for all } \  \mathrm{C}_c^{\infty}(0,T;\mathcal{V}).
\end{align}
Next, we choose a test function of the form $\varphi(x,t)=\upxi(t)\psi(x)$ with $\upxi\in\mathrm{C}_c^{\infty}(0,T)$ and $\psi\in\mathcal{V}$. Then, from \eqref{press11}, we have 
\begin{align*}
	\int_0^T\upxi(t)\langle\upeta,\psi\rangle\d t=0, \ \text{ for all } \  \upxi\in\mathrm{C}_c^{\infty}(0,T) \  \text{ and for all } \  \psi\in\mathcal{V}.
\end{align*}
Then, we deduce that $\upeta$ satisfies the following equation   for a.e. $t\in[0,T]$:
\begin{align}\label{press12}
& \langle\upeta,\psi\rangle 
\nonumber\\ &=  
	\bigg\langle\frac{\partial\y}{\partial t} - \nu \Delta \y  + [(\y+\z)\cdot \nabla](\y+\z) - \alpha\text{div}((\Arm(\y + \z))^2) - \beta \text{div}(|\Arm(\y + \z)|^2\Arm(\y + \z)) 
     - \chi \z - \f,\psi\bigg\rangle, 
\end{align}
 for all $ \psi\in  \C_c^{\infty}(\mathcal{O};\R^d)$,  and 
\begin{align*}
	\langle\upeta ,\psi\rangle=0,  \  \text{ for all } \ \psi\in\mathcal{V}.
\end{align*}
Finally, by an application of \cite[Proposition 1.1]{Temam_1984}, there exists a $\hat{\mathbf{P}} \in \mathscr{D}'((0,T)\times \mathcal{O})$ such that
$$\upeta = - \nabla \hat{\mathbf{P}} \in \mathrm{L}^{2} (0,T;\H^{-1}(\mathcal{O})) + \mathrm{L}^{\frac43} (0,T;\mathbb{W}^{-1,\frac43}(\mathcal{O})).$$

\begin{remark}\label{rem-pressure}
    Indeed, we have 
\begin{align*}
    \nabla \hat{\mathbf{P}} & = (\mathrm{I}- \mathcal{P}) \left\{  \nu \Delta \y  - [(\y+\z)\cdot \nabla](\y+\z) + \alpha\text{div}((\Arm(\y + \z))^2) + \beta  \text{div}(|\Arm(\y + \z)|^2\Arm(\y + \z))  + \f \right\}
    \nonumber\\ & =:\nabla \hat{\mathbf{P}}_1 + \nabla \hat{\mathbf{P}}_2,
\end{align*}
where 
\begin{align*}
    \nabla \hat{\mathbf{P}}_1 & = (\mathrm{I}- \mathcal{P}) \left\{  \nu \Delta \y  - [(\y+\z)\cdot \nabla](\y+\z) + \alpha\text{div}((\Arm(\y + \z))^2)   + \f \right\} \\
    \nabla \hat{\mathbf{P}}_2 & = \beta (\mathrm{I}- \mathcal{P})   \text{div}(|\Arm(\y + \z)|^2\Arm(\y + \z)).
\end{align*}
Note that, by Lemma \ref{D-abH-ubd}, we have for $s\leq -\mathcal{T}=\max\{t_{\D}^{\ast}, t_{\D}^{\ast\ast}\}$ that
\begin{align}\label{estimate-pressure}
		e^{\frac{\nu\lambda}{2}\left(1+\frac{\varepsilon_0}{2}\right) \cdot}\; \nabla \widehat{\textbf{P}}_1\in	\mathrm{L}^2(s,0;\H^{-1}(\mathcal{O})) \; \text{ and }\; e^{\frac{3\nu\lambda}{4}\left(1+\frac{\varepsilon_0}{2}\right) \cdot}\; \nabla \widehat{\textbf{P}}_2\in	\mathrm{L}^{\frac{4}{3}}(s,0;\mathbb{W}^{-1,\frac43}(\mathcal{O})).
	\end{align}
\end{remark}

\end{appendix}

\vskip 6mm
\noindent
\textbf{Acknowledgments:}  %We express our deep gratitude to the anonymous referee for the careful reading	and valuable suggestions and comments that have improved this manuscript significantly. 
“This work is funded by national funds through the FCT - Fundação para a Ciência e a Tecnologia, I.P., under the scope of the projects UIDB/00297/2020 (https://doi.org/10.54499/UIDB/00297/2020) and UIDP/00297/2020 (https://doi.org/10.54499/UIDP/00297/2020) (Center for Mathematics and Applications)”. K. Kinra would like to thank Prof. Manil T. Mohan, Department of Mathematics, Indian Institute of Technology Roorkee, Roorkee, India for introducing him into the research area of attractor theory.

\medskip
\medskip\noindent
\textbf{Data availability:} No data was used for the research described in the article.

\medskip
\medskip\noindent
\textbf{Declarations}: During the preparation of this work, the authors have not used AI tools.

\medskip
\medskip\noindent
\textbf{Conflict of interest:} The authors declare no conflict of interest.

\end{document}